\documentclass[12pt,oneside]{amsart}
\usepackage{amsmath,amssymb,cite,mathrsfs,tikz-cd}
\usepackage[all]{xy}
\usepackage{typearea} 
\usepackage{hyperref} 
\usepackage{color} 

\usepackage{pdfcomment} 

\newcounter{maincounter}
\numberwithin{maincounter}{section}
\numberwithin{equation}{section}

\newtheorem{lemma}[maincounter]{Lemma}
\newtheorem{proposition}[maincounter]{Proposition}
\newtheorem{corollary}[maincounter]{Corollary}
\newtheorem{remark}[maincounter]{Remark}
\newtheorem{theorem}[maincounter]{Theorem}

\newtheorem{definition}[maincounter]{Definition}

\def\AA{\mathbb{A}}

\def\NN{\mathbb{N}}
\def\RR{\mathbb{R}}
\def\TT{\mathbb{T}}
\def\CC{\mathbb{C}}
\def\ZZ{\mathbb{Z}}
\def\PP{\mathbb{P}}
\def\QQ{\mathbb{Q}}

\newcommand{\cal}{\mathcal}

\newcommand{\cA}{\cal{A}}

\newcommand{\sman}[1]{{#1}^{\mathrm{sm,an}}}
\newcommand{\ansm}[1]{{#1}^{\mathrm{an,sm}}}
\newcommand{\sm}[1]{{#1}^{\mathrm{sm}}}
\newcommand{\anE}{\mathrm{an}}
\newcommand{\an}[1]{{#1}^{\anE}}

\newcommand{\IP}{{\PP}}

\newcommand{\IC}{{\CC}}
\newcommand{\IR}{{\RR}}
\newcommand{\IT}{{\TT}}

\newcommand{\IQbar}{{\overline{\QQ}}}

\newcommand{\IZ}{{\ZZ}}
\newcommand{\IN}{{\NN}}
\newcommand{\IA}{{\AA}}
\newcommand{\IQ}{{\QQ}}

\newcommand{\cE}{{\mathcal E}}
\newcommand{\cF}{{\mathcal F}}
\newcommand{\cK}{{\mathcal K}}
\newcommand{\cL}{{\mathcal L}}
\newcommand{\cM}{{\mathcal M}}
\newcommand{\cO}{{\mathcal O}}

\newcommand{\ssm}{\setminus}

\renewcommand{\subset}{\subseteq} 
\renewcommand{\supset}{\supseteq}

\newcommand{\jac}{\mathrm{Jac}}

\newcommand{\pullbackcorner}[1][dr]{\save*!/#1-1.7pc/#1:(-1.5,1.5)@^{|-}\restore}

\usepackage{color}

\newcommand{\bestlevel}{3} 

\begin{document}
\title{Uniformity in Mordell--Lang for curves}

\author{Vesselin Dimitrov, Ziyang Gao and Philipp Habegger}

\address{Department of Mathematics, University of Toronto, 40 St. George Street, Toronto, Ontario, Canada M5S 2E4}
\email{dimitrov@math.toronto.edu}
\address{CNRS, IMJ-PRG, 4 place Jussieu, 75005 Paris, France}
\email{ziyang.gao@imj-prg.fr}
\address{Department of Mathematics and Computer Science, University of Basel, Spiegelgasse 1, 4051 Basel, Switzerland}
\email{philipp.habegger@unibas.ch}






\maketitle

\begin{abstract}
  Consider a smooth, geometrically irreducible, projective curve of genus
$g\ge 2$ defined over a number field of degree $d \ge 1$.
It has at most finitely many rational points by
the Mordell Conjecture, a theorem of 
Faltings. 
We show that the number of rational points is bounded only
in terms of $g$, $d$, and the Mordell--Weil rank of the curve's
Jacobian, thereby answering in the affirmative a question of Mazur. 
In
addition we obtain uniform bounds, in $g$ and $d$, for the number of
geometric torsion points of the Jacobian which lie in the image of an
Abel--Jacobi map. Both estimates generalize our previous work for $1$-parameter families. Our proof uses Vojta's approach to the Mordell Conjecture, and the key new ingredient is the generalization of a height inequality due to the second- and third-named
authors.
\end{abstract}

\tableofcontents

\section{Introduction}
Let $F$ be a field. By a curve defined over $F$ we mean a geometrically
irreducible, projective variety of dimension $1$ defined over $F$. Let $C$ be
a smooth curve of genus at least $2$ defined over a number field $F$. As was
conjectured by Mordell and proved by Faltings~\cite{Faltings:ES},
$C(F)$, the set of $F$-rational points of $C$, is finite.

We let $\jac(C)$ denote the Jacobian of $C$. Recall
that $\jac(C)(F)$ is a finitely generated abelian group by the
Mordell--Weil Theorem. 

The aim of this paper is to bound $\# C(F)$ from above. Here is our
first result. 

\begin{theorem}\label{ThmBdRatIntro}
  Let $g \ge 2$ and $d \ge 1$ be integers. Then there exists a constant
  $c = c(g, d) \ge 1$ with the following property. If $C$ is a smooth curve
  of genus $g$  defined
  over a number field $F$  with $[F:\IQ] \le d$, then
  \begin{equation}
    \label{eq:CKexpinrank}
      \#C(F) \le c^{1+\rho}
  \end{equation}
  where $\rho$ is the rank of  $\mathrm{Jac}(C)(F)$.
\end{theorem}

This theorem gives an affirmative answer to a question posed by
Mazur~\cite[Page~223]{Mazur:00}.
See also~\cite[top of
page 234]{mazur1986arithmetic} for an earlier question. Before this,
 Lang formulated a related  conjecture~\cite[page
 140]{Lang:EllipticCurves78} on the number of integral
points of elliptic curves.

The method of our theorem builds up on the work of many others. At the
core we follow Vojta's proof~\cite{Vojta:siegelcompact} of the
Mordell Conjecture.
Vojta's proof was later simplified  by
Bombieri~\cite{Bombieri:Mordell} and further developed by
Faltings~\cite{Faltings:DAAV}. Silverman~\cite{Silverman:twists}
proved a bound of the quality (\ref{eq:CKexpinrank}) if $C$ ranges
over twists of a given smooth curve. 
The bound by de
Diego~\cite{deDiego:97} 
is of the form $c(g)
7^{\rho}$, where $c(g)>0$ depends only on $g$; the value $7$ had already
arisen in Bombieri's work. But she only counts points
whose height is large in terms of a height of $C$.
Work of David--Philippon \cite{DPvarabII} and R\'emond~\cite{Remond:Decompte}
led to explicit estimates. 
Recently, Alpoge \cite{AlpogeRatPt} \cite[Theorem~6.1.1]{AlpogeThesis} improved  $7$
to $1.872$ and, for $g$ large enough, even to $1.311$. 

On combining the Vojta and Mumford Inequalities
one gets an upper bound for the number of \textit{large points} in
$C(F)$; these are points whose height is sufficiently large relative
to a suitable height of $C$. A lower bound for the N\'eron--Tate
height, such as proved by David--Philippon \cite{DPvarabII}, can be used to
count the number of remaining points which we sometimes call
\textit{small points}. 
Indeed, R\'emond~\cite{Remond:Decompte} made the Vojta and Mumford
Inequalities explicit and obtained 
explicit upper bounds for the number of rational points on curves
embedded in abelian varieties. 
The resulting cardinality bounds
depend on a suitable notion of height of $C$, an artifact of the lower
bounds for the N\'eron--Tate height. 
Later, David--Philippon~\cite{DaPh:07} proved stronger height lower
bounds
in a power of an
elliptic curve. They then 
 obtained uniform estimates
of the quality (\ref{eq:CKexpinrank}) for a curve in a power of
elliptic curves,  thus providing evidence that Mazur's Question had a
positive answer, see also
David--Nakamaye--Philippon's work~\cite{DaNaPh:07}.

We give an overview of the general method  in more detail in
$\mathsection$\ref{subsec:ntdist} below. 

The main innovation of this paper is to prove a lower bound for the
N\'eron--Tate height that is sufficiently strong to eliminate the
dependency on the height of $C$. This leads to a uniform estimate as
in Theorem \ref{ThmBdRatIntro}. In prior work~\cite{DGH1p} we
applied the earlier height lower bound~\cite{GaoHab}
to recover a variant of Theorem~\ref{ThmBdRatIntro} in
a one-parameter family of smooth curves.

We now explain some further  results that follow from the approach
described above. 
For an integer $g\ge 1$, let $\mathbb{A}_{g,1}$ denote the coarse moduli
space of principally polarized abelian varieties of dimension $g$.
This is an irreducible quasi-projective variety which we can take as
defined over $\IQbar$, the algebraic closure of $\IQ$ in $\IC$.
Suppose we are presented with an immersion
$\iota\colon \mathbb{A}_{g,1}\rightarrow \IP^m_\IQbar$ into projective
space.
Let $h\colon \IP^m_\IQbar(\IQbar)\rightarrow\IR$ denote the absolute
logarithmic Weil height,
\textit{cf.}~\cite[$\mathsection$1.5.1]{BG}.
For brevity, we sometimes call $h$ the Weil height.
If $C$ is a smooth curve of genus $g\ge 2$ defined over $\IQbar$ and if
$P_0\in C(\IQbar)$, then we can consider $C-P_0$ as a curve in
$\mathrm{Jac}(C)$ via the Abel--Jacobi map.
We use $[\mathrm{Jac}(C)]$ to denote the point in $ \mathbb{A}_{g,1}(\IQbar)$  parametrizing $\mathrm{Jac}(C)$.

An abelian group $\Gamma$ is said to have finite rank if
$\Gamma\otimes\IQ$ is a finite dimensional $\IQ$-vector space. In this
case $\dim\Gamma\otimes\IQ$ is the rank of $\Gamma$. 
Consider an abelian variety $A$ defined over $\IC$ and let $\Gamma$ be
a finite rank subgroup of $A(\IC)$. Lang~\cite{Lang:Division}
conjectured that a
curve in $A$ intersects $\Gamma$ in a finite set
unless the curve is smooth of genus $1$. The Conjecture follows from
Faltings's Theorem~\cite{Faltings:ES} and work of
Raynaud~\cite{Raynaud:MMcurves}.

The following theorem is more in the spirit of \cite[top of page
234]{mazur1986arithmetic}.

\begin{theorem}
\label{ThmBdFinRank}
Let $g \ge 2$ and let $\iota$ be as above. Then there exist two
constants $c_1= c_1(g,\iota) \ge 0$ and $c_2 = c_2(g,\iota) \ge 1$ with the following property. Let $C$
be a  smooth curve of genus $g$  defined over $\IQbar$,
let $P_0 \in C(\overline \IQ)$, and let $\Gamma$ be a
subgroup of $\mathrm{Jac}(C)(\overline \IQ)$ of finite rank $\rho\ge
0$. If $h(\iota([\mathrm{Jac}(C)])) \ge c_1$, then
\[
\#(C(\overline \IQ)-P_0) \cap \Gamma  \le c_2^{1+\rho}.
\]
\end{theorem}

The following corollary follows from Theorem \ref{ThmBdFinRank} applied to 
$\Gamma = \mathrm{Jac}(C)(\IQbar)_{\mathrm{tors}}$, which has rank $0$.

\begin{corollary}
Let $g \ge 2$ and let $\iota$ be as above.
Then there exist two
constants $c_1= c_1(g,\iota) \ge 0$ and $c_2 = c_2(g,\iota) \ge 1$ with the following property. Let $C$
be a  smooth curve of genus $g$  defined over $\overline \IQ$ and  let $P_0 \in C(\overline \IQ)$.
If $h(\iota([\mathrm{Jac}(C)])) \ge c_1$, then
\[
\# (C(\overline \IQ)-P_0) \cap \mathrm{Jac}(C)(\IQbar)_{\mathrm{tors}}  \le c_2.
\]
\end{corollary}

As in Theorem \ref{ThmBdRatIntro} we can drop the condition on the
height of the Jacobian by working over a number field of bounded
degree. 

\begin{theorem}\label{thm:BdTorIntroNF}
  Let $g \ge 2$ and $d \ge 1$ be integers. Then there exists a constant
  $c = c(g, d) \ge 1$ with the following property. Let $C$ be a 
  smooth curve of genus $g$ defined over a
  number field $F\subset\IQbar$ 
  with $[F:\IQ] \le d$ and let $P_0 \in C(\IQbar)$, then
  \begin{equation*}
    \# (C(\overline \IQ)-P_0) \cap \mathrm{Jac}(C)(\IQbar)_{\mathrm{tors}} \le c.
  \end{equation*}
\end{theorem}

Let us recall some previous results towards Mazur's Question for
rational points, \textit{i.e.}, towards Theorem~\ref{ThmBdRatIntro}.
Based on the method of Vojta, Alpoge~\cite{AlpogeRatPt} proved that
the average number of rational points on a curve of genus $2$ with a
marked Weierstrass point is bounded. Let $C$ be a smooth curve of
genus $g\ge 2$ defined
over a number field $F\subset\IQbar$. The Chabauty--Coleman
approach~\cite{Chabauty, Coleman:effCha} yields estimates under an additional
hypothesis on the rank of Mordell--Weil group. For example, if
$\jac(C)(F)$ has rank at most $g-3$, Stoll~\cite{Stoll:Uniform} showed
that $\#C(F)$ is bounded solely in terms of $[F:\IQ]$ and $g$ if $C$ is
hyperelliptic; Katz--Rabinoff--Zureick-Brown~\cite{KatzRabinoffZB}
later, under the same rank hypothesis, removed the hyperelliptic
hypothesis. Checcoli,
Veneziano, and Viada~\cite{CVV:17} obtain an effective height bound
under a restriction on 
 the Mordell--Weil rank.

As for algebraic torsion points, \textit{i.e.}, in the direction of  Theorem~\ref{thm:BdTorIntroNF}, DeMarco--Krieger--Ye \cite{DeMarcoKriegerYeUniManinMumford} proved a bound on the cardinality of torsion points on any genus $2$
curve that admits a degree-two map to an elliptic curve when the
Abel--Jacobi map is based at a Weierstrass point. Moreover, their bound is
independent of $[F:\IQ]$.

\subsection{N\'{e}ron--Tate distance of algebraic points on curves}
\label{subsec:ntdist}
Let $C$ be a smooth curve defined over $\overline \IQ$ of genus $g \ge
2$, let $P_0 \in C(\overline \IQ)$, and let $\Gamma$ be a subgroup of
$\jac(C)(\overline \IQ)$ of finite rank $\rho$. For simplicity we identify $C$ with its image under the Abel--Jacobi embedding $C \rightarrow \jac(C)$ via $P_0$. 

Our proof of Theorem~\ref{ThmBdFinRank} follows the spirit of the
method of Vojta, later generalized by  Faltings.
Let $\hat h\colon \jac(C)(\IQbar)\rightarrow \IR$ denote the
N\'eron--Tate height attached to a symmetric and ample line bundle on $\jac(C)$. 
We divide $C(\overline \IQ) \cap \Gamma$ into two parts:
\begin{itemize}
\item Small points $\left\{P \in C(\overline \IQ) \cap \Gamma
: \hat{h}(P) \le B(C)  \right\}$;
\item Large points $\left\{P \in C(\overline \IQ) \cap \Gamma : \hat{h}(P) > B(C) \right\}$
\end{itemize}
where $B(C)$ is allowed to depend on a suitable height of $C$. It
turns out that
we can take $B(C)=c_0 \max\{1,h(\iota([\jac(C)]))\}$
for some $c_0 = c_0(g,\iota) > 0$. The constant $c_0$ is chosen in a
way that accommodates  both the \textit{Mumford inequality} and
the \textit{Vojta inequality}. Combining these
two inequalities yields an upper bound on the number of large points
by $c_1(g)^{1+\rho}$, see for example Vojta's~\cite[Theorem
6.1]{Vojta:siegelcompact} in the important case where $\Gamma$ is the
group of points of $\jac(C)$ rational over
 a number field or more generally in the work of
 David--Philippon~\cite{DPvarabII,DaPh:07} and R\'emond~\cite{Remond:Decompte}.

Thus in order to prove Theorem~\ref{ThmBdFinRank}, it suffices to bound the number of small points.
In this paper we find such a  bound by studying the
N\'{e}ron--Tate distance of points in $C(\overline \IQ)$.

Roughly speaking,
we find positive constants $c_1,c_2,c_3,$ and $c_4$  that depend on
$g$ and $\iota$, but not 
on $C$, such that  if $h(\iota([\jac(C)]))\ge
c_1$ then for all $P\in C(\IQbar)$ we have the following  alternative.
\begin{itemize}
\label{eq:alternative}
\item Either $P$ lies in  a subset of $C(\overline \IQ)$ of cardinality at most $c_2$,
\item or $\left\{Q \in C(\overline \IQ) : \hat{h}(Q-P) \le h(\iota([\jac(C)]))/c_3 \right\} < c_4$.
\end{itemize}

This dichotomy is stated in Proposition~\ref{PropAlgPtFar}.
In this paper, we make the statement precise by referring to the universal family of genus $g$ smooth curves with suitable level structure, and the N\'eron--Tate height on each Jacobian attached to the tautological line bundle. 
The setup is done in $\mathsection$\ref{SectionSettingUp}.

This proposition can be seen as a relative version of the Bogomolov
conjecture for abelian varieties with large height. It has the
following upshot: If $h(\iota([\jac(C)])) \ge c_1$, then the small
points in $C(\overline \IQ) \cap \Gamma$ lie in a set of uniformly
bounded cardinality, or are contained in $(1+c_0c_3)^\rho$ balls in the
N\'eron--Tate metric, with each ball containing at most $c_4$ points.
This will yield the desired bound in Theorem~\ref{ThmBdFinRank}, as
executed in $\mathsection$\ref{SectionRatPt}.

\subsection{Height inequality and non-degeneracy}\label{subsec:hgtineqnondeg}
We follow the framework presented in our previous work \cite{DGH1p}.
In \textit{loc.cit.} we proved the result for $1$-parameter families,
as an application of the second- and third-named authors' height
inequality \cite[Theorem~1.4]{GaoHab}. Passing to general cases
requires generalizing this height inequality to higher dimensional
bases. The generalization has two parts:  generalizing the inequality
itself under the non-degeneracy condition and generalizing the
criterion of non-degenerate subvarieties.
We execute the first part in the current paper while the second part was 
done by the second-named author in \cite{GaoBettiRank}. Let us explain the setup.

Let $k$ be an algebraically closed subfield of $\IC$. Let $S$ be a
regular, irreducible, quasi-projective variety defined over $k$ that
is Zariski open in an irreducible projective variety
$\overline S \subset\IP_k^m$. Let $\pi \colon \cA \rightarrow S$ be an
abelian scheme of relative dimension $g\ge 1$. We suppose that we are
presented with a closed immersion $\cA\rightarrow\IP_k^n\times S$ over
$S$. On the generic fiber of $\pi$ we assume that this immersion comes
from a basis of the global sections of the $l$-th power of a symmetric
and ample line bundle with $l\ge 4$. If $k=\IQbar$ and as described in
$\mathsection\ref{SubsectionHeight}$ we obtain two height functions,
the restriction of the Weil height
$h\colon \overline{S}(\IQbar)\rightarrow \IR$ and the N\'eron--Tate
height $\hat h_{\cA} \colon \cA(\IQbar)\rightarrow \IR$.

Let  $\ell\ge \bestlevel$ be an integer. Throughout the whole paper, by \textit{level-$\ell$-structure} we always mean \textit{symplectic level-$\ell$-structure}. For the purpose of our main applications, including
Theorems~\ref{ThmBdRatIntro} and \ref{ThmBdFinRank}, it suffices to
work under the following hypothesis. 

\begin{center}
  \label{def:hyp}
  {\tt (Hyp):} $\cA\rightarrow S$ carries a principal polarization and has level-$\ell$-structure for some $\ell \ge 3$.
\end{center}
So in the main body of the paper, we will focus on the case {\tt
(Hyp)}. The general case where {\tt (Hyp)} is not assumed will be
handled in Appendix~\ref{AppHtIneqApp}.

The non-degenerate subvarieties of $\cA$ are defined using the
\textit{Betti map} which we briefly describe here; 
the precise definition will be given by
Proposition~\ref{PropBettiMapApp} and in Proposition~\ref{PropBettiMap} under {\tt (Hyp)}.

For any $s \in S(\IC)$, there exists an open neighborhood $\Delta
\subseteq S^{\mathrm{an}}$ of $s$ which we may assume is simply-connected. Then one can define a basis $\omega_1(s),\ldots,\omega_{2g}(s)$ of the period lattice of each fiber $s \in \Delta$ as holomorphic functions of $s$. Now each fiber $\cA_s = \pi^{-1}(s)$ can be identified with the complex torus $\IC^g/(\IZ \omega_1(s)\oplus \cdots \oplus \IZ\omega_{2g}(s))$, and each point $x \in \cA_s(\IC)$ can be expressed as the class of $\sum_{i=1}^{2g}b_i(x) \omega_i(s)$ for real numbers $b_1(x),\ldots,b_{2g}(x)$. Then $b_{\Delta}(x)$ is defined to be the class of the $2g$-tuple $(b_1(x),\ldots,b_{2g}(x)) \in \IR^{2g}$ modulo $\IZ^{2g}$.
We obtain with a real-analytic map
\[
b_{\Delta} \colon \cA_{\Delta} = \pi^{-1}(\Delta) \rightarrow \mathbb{T}^{2g},
\]
which is fiberwise a group isomorphism and
where $\mathbb{T}^{2g}$ is the real torus of dimension $2g$.

\begin{definition}\label{DefinitionNonDegenerate}
An irreducible subvariety $X$ of $\cA$ is said to be non-degenerate if
there exists an open non-empty subset $\Delta$ of $S^{\mathrm{an}}$,
with the Betti map
$b_{\Delta} \colon \cA_{\Delta}:=\pi^{-1}(\Delta) \rightarrow \mathbb{T}^{2g}$, 
such that
\begin{equation}\label{EqDefnNonDeg}
\max_{x \in \sman{X} \cap \cA_{\Delta}} \mathrm{rank}_{\IR} (\mathrm{d}b_{\Delta}|_{\sman{X}})_x = 2\dim X
\end{equation}
where $\mathrm{d}b_\Delta$ denotes the differential and $\sman{X}$ is
the analytification of the regular locus of $X$.
\end{definition}

As the inequality $\le$ in \eqref{EqDefnNonDeg} trivially holds true, \eqref{EqDefnNonDeg} is equivalent to: there exists $x \in \sman{X} \cap \cA_{\Delta}$ such that $\mathrm{rank}_{\IR} (\mathrm{d}b_{\Delta}|_{\sman{X}})_x = 2\dim X$.

In Proposition~\ref{PropBettiForm}(iii) we give another characterization of non-degenerate
subvarieties. We can now formulate the height inequality. 

\begin{theorem}\label{ThmHtInequality}
  Suppose that $\cA$ and $S$ are as above with $k=\IQbar$;
  in particular, $\cA$ satisfies {\tt (Hyp)}.
  Let $X$ be a closed irreducible subvariety of $\cA$
  defined over $\IQbar$ that dominates $S$.   Suppose $X$ is non-degenerate,
  as defined in Definition~\ref{DefinitionNonDegenerate}. Then there
  exist constants $c_1>0$ and $c_2\ge 0$ and a Zariski open dense subset
  $U$ of $X$ with 
  \begin{equation*}
    \hat{h}_{\cA}(P) \ge c_1 h( \pi(P) ) - c_2 \quad \text{for all}\quad P \in U(\IQbar).
  \end{equation*}
\end{theorem}

Note that \cite[Theorem~1.4]{GaoHab} is, up to some minor reduction,
precisely Theorem~\ref{ThmHtInequality} for $\dim S = 1$ together with
the criterion for $X$ to be non-degenerate when $\dim S = 1$. In general, the degeneracy behavior of $X$ is fully studied in \cite{GaoBettiRank}. See \cite[Theorem~1.1]{GaoBettiRank} for the criterion. However in practice, we sometimes still want to understand the height comparison on some degenerate $X$. One way to achieve this is by applying \cite[Theorem~1.3]{GaoBettiRank}, which asserts the following statement: If $X$ satisfies some reasonable properties, then we can apply Theorem~\ref{ThmHtInequality} after doing some simple operations with $X$. 

For the purpose of proving Proposition~\ref{PropAlgPtFar} and
furthermore Theorem~\ref{ThmBdFinRank}, we work in the following
situation.

Let $\mathbb{A}_{g,\ell}$
denote the moduli space of principally polarized $g$-dimensional
abelian varieties with level-$\ell$-structure. It is
a classical fact that $\mathbb{A}_{g,\ell}$ is represented by an
 irreducible, regular, quasi-projective
variety defined over a number field,
see~\cite[Theorem~7.9 and below]{MFK:GIT94} or \cite[Theorem 1.9]{OortSteenbrink}, so it is a
fine moduli space.
Let $\mathbb{M}_{g,\ell}$ be the fine moduli space of smooth curves of genus
$g$ whose Jacobian is equipped with  level-$\ell$-structure;
see~\cite[(5.14)]{DM:irreducibility} or \cite[Theorem
1.8]{OortSteenbrink} for the existence.
Then $\mathbb{M}_{g,\ell}$ is an irreducible, regular, quasi-projective
variety defined over a number field. 

To avoid confusion on different notations in different references, we
make the following convention throughout  the paper. We will take
$\mathbb{A}_{g,\ell}$ and $\mathbb{M}_{g,\ell}$ as geometrically
irreducible varieties. 
 Some authors define
$\mathbb{A}_{g,\ell}$ over $\mathbb{Z}[1/\ell]$ (or over $\mathbb{Z}$)
and then consider it over $\IQbar$ by base change. The
$\IQbar$-variety thus obtained may not be irreducible, and each
irreducible component is defined over $\IQ(\zeta_{\ell})$ for some
 root of unity $\zeta_{\ell}$ of order $\ell$.
Choosing a geometrically irreducible component of $\mathbb{A}_{g,\ell}$
amounts to fixing a complex root of unity of order $\ell$.
We fix such a choice once and for all and consider
$\mathbb{A}_{g,\ell}$ as an irreducible variety defined over $\IQbar$.
The same  holds
for $\mathbb{M}_{g,\ell}$.
We will usually fix $\ell$ and abbreviate $\mathbb{A}_{g,\ell}$ (resp.
$\mathbb{M}_{g,\ell}$) by $\mathbb{A}_g$ (resp. $\mathbb{M}_g$).
It is often convenient to consider $\mathbb{A}_g$ and $\mathbb{M}_g$
as over $\IQbar$, but sometimes we will recall that both arise from
varieties defined over the number field $\IQ(\zeta_\ell)$.
We denote the coarse moduli space of smooth curves of genus $g$
with $\mathbb{M}_{g,1}$.

Furthermore, let
$\mathfrak{C}_g \rightarrow \mathbb{M}_g$ be the universal curve and
$\mathfrak{A}_g \rightarrow \mathbb A_g$ be the universal abelian
variety. Taking the Jacobian of a smooth curve leads to the  Torelli morphism
$\mathbb{M}_g \rightarrow \mathbb A_g$ which is finite-to-$1$  (but
 not injective  as we
have level structure).
Moreover, for $M\ge 2$ let
$\mathscr{D}_M$ denote the $M$-th Faltings--Zhang morphism fiberwise defined by
sending
\begin{equation}
  \label{eq:faltingszhangintro}
  (P_0, P_1, \ldots, P_M) \mapsto (P_1-P_0, \ldots, P_M-P_0);
\end{equation}
we give a precise definition of this morphism
in $\mathsection$\ref{subsec:univfamily}. 
Roughly speaking, we will  apply Theorem~\ref{ThmHtInequality} to
\[
  X := \mathscr{D}_M(\underbrace{\mathfrak{C}_g \times_{\mathbb{M}_g} \cdots \times_{\mathbb{M}_g} \mathfrak{C}_g}_{(M+1)\text{-copies}}) \subseteq \underbrace{\mathfrak{A}_g \times_{\mathbb M_g} \cdots \times_{\mathbb M_g} \mathfrak{A}_g}_{M\text{-copies}}
\]
for
a suitable  $M$. To verify non-degeneracy we will refer to  
the second-named author's work \cite[Theorem~1.2']{GaoBettiRank} which
applies if $M$ is large in terms of $g$.  So we can apply Theorem~\ref{ThmHtInequality} to such $X$.  This
will eventually lead to Proposition~\ref{PropAlgPtFar}. 

The morphism and its variants are powerful tools in diophantine
geometry, see~\cite[Lemma~4.1]{Faltings:DAAV}. It is closely connected
to problems involving small N\'eron--Tate height,
see~\cite[Lemma~3.1]{ZhangEquidist}
. Stoll~\cite{Stoll:Uniform} used a
variant of \eqref{eq:faltingszhangintro} to show that a conjecture of
Pink~\cite{Pink} on unlikely intersections implies
Theorem~\ref{ThmBdFinRank} with the condition
$h(\iota([\mathrm{Jac}(C)])) \ge c_1$ removed and with $C$ allowed to
be defined over $\IC$.
 
At this stage it is worth outlining the main steps of the proof of \cite[Theorem~1.2']{GaoBettiRank}, or the more general \cite[Theorem~1.3]{GaoBettiRank}, due to its importance to the current paper. The major step is to establish a criterion, \textit{in simple geometric terms}, for an irreducible subvariety $X$ of the universal abelian variety $\mathfrak{A}_g$ to be degenerate. Roughly speaking, the proof of the desired criterion is divided into two steps. Step~1 transfers the degeneracy property to an \textit{unlikely intersection problem} in $\mathfrak{A}_g$ by invoking the \textit{mixed Ax--Schanuel theorem} for $\mathfrak{A}_g$ \cite[Theorem~1.1]{GaoMixedAS}. More precisely we show that $X$ is degenerate if and only if $X$ is the union of subvarieties satisfying an appropriate unlikely intersection property. Step~2 solves this unlikely intersection problem, and the key point is to use \cite[Theorem~1.4]{GaoMixedAS} to prove that the union mentioned above is a finite union. In this step the notion of \textit{weakly optimal subvarieties} introduced by the third-named author and Pila \cite{HabeggerPilaENS} is involved.

\subsection{General notation}
We collect here an overview of notation used throughout the text. 

Let $S$ be an irreducible, quasi-projective variety defined over an
algebraically closed field $k$. Then $\sm{S}$ denotes the regular
locus of $X$. If $\pi\colon \cA\rightarrow S$ is an abelian scheme
then $[N]\colon \cA\rightarrow\cA$ is the multiplication-by-$N$
morphism for all $N\in\IN=\{1,2,3,\ldots\}$, and if $s\in S(k)$, the fiber
$\cA_s=\pi^{-1}(s)$ is an abelian variety defined over $k$.
If $k\subset \IC$, then
$\an{S}$ denotes the analytification of $S$; it carries a natural
topology that is Hausdorff.

We write $\mathbb{T}$ for the circle group $\{z\in \IC : |z|=1\}$.

\subsection*{Acknowledgements} The authors would like to thank
Shou-wu Zhang for relevant discussions and Gabriel Dill for the
argument involving torsion points to bound $h_1$ on
page~\pageref{page:boundh1}. 
We would also like
to thank Lars K\"uhne and Ngaiming Mok  for discussions on Hermitian Geometry; our paper 
is much influenced by K\"uhne's approach towards bounded height
\cite{kuehne:semiabelianbhc}
and by Mok's approach to study the Mordell--Weil rank over function
fields \cite{Mok11Form}.
We thank Gabriel Dill, Lars K\"uhne, Fabien Pazuki, and Joseph H. Silverman for corrections and comments
on a draft of this paper. 
We also thank the referees for their careful reading and valuable comments. 
VD would like
to thank the NSF and the Giorgio and Elena Petronio Fellowship Fund II
for financial support for this work. VD has received funding from the European Union's Seventh Framework Programme (FP7/2007--2013) / ERC grant agreement n$^\circ$ 617129. ZG has received
funding from the French National Research Agency grant
ANR-19-ERC7-0004, and the European Research Council (ERC) under the
European Union's Horizon 2020 research and innovation programme (grant
agreement n$^\circ$ 945714). PH has received funding
from the Swiss National Science Foundation project n$^\circ$ 200020\_184623.
Both VD and ZG would like to
thank the Institute for Advanced Study and the special year ``Locally
Symmetric Spaces: Analytical and Topological Aspects'' for its
hospitality during this work.


\section{Betti map and Betti form}\label{SectionBettiRevisited}
The goals of this section are to revisit the Betti map, the Betti form and make a link between them. 
In this paper we construct the Betti map using the universal family of
principally polarized abelian varieties with level-$\ell$-structure and
bypass the ad-hoc construction found in \cite{GaoHab}.

 In this section we will make
the following assumptions.
All varieties  are defined over the field $\IC$.
Let $S$ be an  irreducible, regular, quasi-projective variety over $\IC$. 
Let $\pi\colon \cA\rightarrow S$ be an abelian scheme of  relative
dimension $g$, that carries a principal polarization, and such that
 $\cA$ is equipped with level-$\ell$-structure, for some
 $\ell\ge\bestlevel$, \textit{i.e.}, {\tt (Hyp)} is satisfied. 

\begin{proposition}\label{PropBettiMap}
Let $s_0 \in S(\IC)$. Then there exist an open neighborhood $\Delta$ of $s_0$ in $S^{\mathrm{an}}$, and a map $b_{\Delta} \colon \cA_{\Delta} := \pi^{-1}(\Delta) \rightarrow \mathbb T^{2g}$, called the Betti map, with the following properties.
\begin{enumerate}
\item[(i)] For each $s \in \Delta$ the restriction $b_{\Delta}|_{\cA_s(\IC)} \colon \cA_s(\IC) \rightarrow \mathbb T^{2g}$ is a group isomorphism.
\item[(ii)] For each $\xi \in \mathbb T^{2g}$ the preimage $b_{\Delta}^{-1}(\xi)$ is a complex analytic subset of $\cA_{\Delta}$.
\item[(iii)] The product
$(b_{\Delta},\pi) \colon \cA_{\Delta} \rightarrow \mathbb
T^{2g} \times \Delta$ is a real analytic isomorphism.%
\end{enumerate}
\end{proposition}

The properties (i) -- (iii) do not uniquely determine $b_\Delta$.
Indeed,  composing $b_\Delta$ with an automorphism of the
topological group $\IT^{2g}$, \textit{i.e.}, an element of
$\mathrm{GL}_{2g}(\IZ)$, leads to a new Betti map satisfying (i) --
(iii). After shrinking   $\Delta$ we may assume that it is connected.
In this case, an application of the
Baire Category Theorem shows that $b_\Delta$ is
uniquely determined by (i) and (iii)
up to composition with a unique element of $\mathrm{GL}_{2g}(\IZ)$.

Andr\'{e}, Corvaja, and Zannier~\cite{ACZBetti} recently began the
study of the maximal rank of the Betti map, especially the submersivity, using a slightly different
definition. A full study of this maximal rank was realized in \cite{GaoBettiRank}. Closely related to the
Betti map is the Betti form, a semi-positive $(1,1)$-form on
$\an{\cA}$, which was first introduced in Mok \cite{Mok11Form}.

\begin{proposition}\label{PropBettiForm}
There exists a closed $(1,1)$-form $\omega$ on $\an{\cA}$, called the Betti form, such that the following properties hold.
\begin{enumerate}
\item[(i)] The $(1,1)$-form $\omega$ is semi-positive, \textit{i.e.},  at each point
the associated Hermitian form is positive semi-definite.
\item[(ii)] For all $N\in\IZ$ we have $[N]^*\omega = N^2 \omega$.
\item[(iii)] If $X$ is an irreducible subvariety of $\cA$ of dimension
$d$ and $\Delta\subset \an{S}$ is open with $\sman{X}\cap\cA_\Delta\not=\emptyset$, then
\[
\omega|_{\sman{X}}^{\wedge d} \not\equiv 0 \quad\text{if and only
if}\quad  \max_{x \in
\sman{X} \cap \cA_{\Delta}} \mathrm{rank}_{\IR}
(\mathrm{d}b_{\Delta}|_{X^{\mathrm{sm,an}}})_x = 2d. 
\]

\end{enumerate}
\end{proposition}

We will prove both propositions during the course of this section using the universal abelian variety. A dynamical approach can be found in \cite[$\mathsection$2]{CGHX:18}.

\subsection{Betti map for the universal abelian variety}\label{SubsectionBettiMap}
Our proof of Proposition~\ref{PropBettiMap} follows the construction in \cite[$\mathsection$3-$\mathsection$4]{GaoBettiRank}. We divide it into several steps.

We start to prove Proposition~\ref{PropBettiMap} for $S = \mathbb
A_g$, the moduli space of principally polarized abelian variety of
dimension $g$ with level-$\ell$-structure;  it is a fine moduli
space. Let $\pi^{\mathrm{univ}} \colon \mathfrak{A}_g \rightarrow \mathbb A_g$ be the universal abelian variety.

The universal covering $\mathfrak{H}_g \rightarrow \an{\mathbb A}_g$,
where $\mathfrak{H}_g$ is the Siegel upper half space, gives a
polarized family of abelian varieties
$\cA_{\mathfrak{H}_g} \rightarrow \mathfrak{H}_g$ fitting into the diagram
\[
\xymatrix{
\cA_{\mathfrak{H}_g} := \mathfrak{A}_g \times_{\an{\mathbb A}_g}\mathfrak{H}_g \ar[r]^-{u_B} \ar[d] & \an{\mathfrak{A}}_g \ar[d]^{\pi^{\mathrm{univ}}} \\
\mathfrak{H}_g \ar[r] & \an{\mathbb A}_g.
}
\]
For the universal covering $u \colon \IC^g \times \mathfrak{H}_g
\rightarrow \cA_{\mathfrak{H}_g}$ and for each $Z \in
\mathfrak{H}_g$, the kernel of $u|_{\mathbb{C}^g \times \{Z\}}$ is
$\mathbb{Z}^g + Z \mathbb{Z}^g$. Thus the map $\IC^g \times
\mathfrak{H}_g \rightarrow \IR^g \times \IR^g \times
\mathfrak{H}_g \rightarrow \IR^{2g}$, where the first map is the
inverse of $(a,b,Z) \mapsto (a + Z  b,Z)$ and the second map is the
natural projection, descends to a \textit{real} analytic map
\[
b^{\mathrm{univ}} \colon \cA_{\mathfrak{H}_g} \rightarrow \mathbb T^{2g}.
\]
 Now for each $s_0 \in \mathbb A_g(\IC)$, there exists a contractible,
 relatively compact, open neighborhood $\Delta$ of $s_0$ in $\mathbb A_g^{\mathrm{an}}$ such that $\mathfrak{A}_{g,\Delta}:=(\pi^{\mathrm{univ}})^{-1}(\Delta)$ can be identified with $\cA_{\mathfrak{H}_g,\Delta'}$ for some open subset $\Delta'$ of $\mathfrak{H}_g$. The composite $b_{\Delta} \colon \mathfrak{A}_{g,\Delta} \cong \cA_{\mathfrak{H}_g,\Delta'} \rightarrow \mathbb T^{2g}$ is  real analytic and satisfies the three properties listed in Proposition~\ref{PropBettiMap}. Thus $b_{\Delta}$ is the desired Betti map in this case. Note that for a fixed (small enough) $\Delta$, there are infinitely choices of $\Delta'$; but for $\Delta$ small enough, if $\Delta_1'$ and $\Delta_2'$ are two such choices, then $\Delta_2' = \alpha \cdot \Delta_1'$ for some $\alpha \in \mathrm{Sp}_{2g}(\IZ) \subset \mathrm{SL}_{2g}(\IZ)$. Thus we have proved Proposition~\ref{PropBettiMap} for $\mathfrak{A}_g \rightarrow \mathbb A_g$.

\subsection{Betti form for the universal abelian variety}

\newcommand{\X}{X}
\newcommand{\Y}{Y}
\newcommand{\rmd}{\mathrm{d}}
\newcommand{\tran}[1]{{{#1}}^{\!^{\intercal}}}

For the universal covering $\mathbf{u} = u_B\circ u \colon \IC^g
\times \mathfrak{H}_g \rightarrow \mathfrak{A}_g^{\mathrm{an}}$, we
will use $(w,Z)$ to denote the coordinates on $\IC^g \times
\mathfrak{H}_g$.
Below $\mathrm{Im}$ denotes imaginary part. 

\begin{lemma}\label{LemmaBettiOnUniversalCovering}
Define
\[
\hat{\omega}^{\mathrm{univ}}:=\sqrt{-1}\partial \overline{\partial} \left( 2 (\mathrm{Im}w)^{\!^{\intercal}} (\mathrm{Im} Z)^{-1} (\mathrm{Im}w) \right).
\]
Then $\hat{\omega}^{\mathrm{univ}}$ is a closed  semi-positive
$(1,1)$-form on $\IC^g\times\mathfrak{H}_g$ satisfying 
\begin{equation}
  \label{eq:bettiformaltformula}
  \hat{\omega}^{\mathrm{univ}} = \sqrt{-1} \tran{\left( \rmd Z \Y^{-1}
      \mathrm{Im}(w) - \rmd w\right)} \wedge \Y^{-1} \left( \rmd \overline Z \Y^{-1}
    \mathrm{Im}(w) - \rmd \overline w\right)
\end{equation}
with $\Y = \mathrm{Im}(Z)$; here and below the symbol $\wedge$
is used as a combination of wedge product and matrix multiplication
when appropriate. 
Moreover, if $N\in\IZ$ and if we denote by $\widetilde{N} \colon \IC^g \times \mathfrak{H}_g \rightarrow \IC^g \times \mathfrak{H}_g$ the map $(w,Z) \mapsto (Nw,Z)$, then $\widetilde{N}^*\hat{\omega}^{\mathrm{univ}} = N^2\hat{\omega}^{\mathrm{univ}}$.
\end{lemma}
\begin{proof}
The $(1,1)$-form $\hat{\omega}^{\mathrm{univ}}$ is  closed since $\mathrm{d} = \partial + \overline{\partial}$. We will prove the semi-positivity by direct computation.

   
We have the following formulae for partial derivatives
  \begin{equation*}    
    \begin{aligned}
      \overline\partial \mathrm{Im} w &=& \frac{\sqrt{-1}}{2} \mathrm{d}\overline w, \qquad     
      \overline\partial (\Y^{-1}) &=&& -\frac{\sqrt{-1}}{2} \Y^{-1} \mathrm{d} \overline Z
      \Y^{-1},\\
      \partial \mathrm{Im} w &=& -\frac{\sqrt{-1}}{2} \mathrm{d} w,\qquad
      \partial (\Y^{-1}) &=&& \frac{\sqrt{-1}}{2} \Y^{-1} \mathrm{d}  Z \Y^{-1}.
    \end{aligned}
  \end{equation*}
Let us prove the formulae on the right. We hereby do it for $\partial (\Y^{-1}) = \frac{\sqrt{-1}}{2} \Y^{-1} \mathrm{d}  Z \Y^{-1}$ and the other one is similar. Taking partial derivatives on both sides of $\Y \Y^{-1} = I$, we get $(\partial \Y)\Y^{-1} + \Y \partial (Y^{-1}) = 0$. So $\partial (\Y^{-1}) = - \Y^{-1}(\partial \Y)\Y^{-1}$. But $\partial \Y = \partial \mathrm{Im}Z = - \frac{\sqrt{-1}}{2}\mathrm{d} Z$. Hence we get the desired formula for $\partial (\Y^{-1})$.
  
Using these formulae  
and the Leibniz rule (note that $Z = Z^{\!^{\intercal}}$ and hence $\mathrm{d}Z = \mathrm{d}Z^{\!^{\intercal}}$), we get
\begin{align*}
    \hat\omega^{\mathrm{univ}} = \sqrt{-1} \big( &
 (\mathrm{d}w)^{\!^{\intercal}} \Y^{-1} \wedge  \mathrm{d}\overline{w} + (\mathrm{Im}w)^{\!^{\intercal}} \Y^{-1}  \mathrm{d}Z \wedge \Y^{-1}  \mathrm{d}\overline{Z}  \Y^{-1} (\mathrm{Im}w) \\
 & - (\mathrm{Im}w)^{\!^{\intercal}} \Y^{-1} \mathrm{d}Z
 \Y^{-1} \wedge  \mathrm{d}\overline{w} -
 (\mathrm{d}w)^{\!^{\intercal}}   \wedge \Y^{-1}  \mathrm{d}\overline{Z} \Y^{-1} (\mathrm{Im}w) \big). 
\end{align*}
  Rearranging yields the desired equality (\ref{eq:bettiformaltformula}). 
  The associated  form is 
  \begin{equation*}
    H:\bigl((\xi_w,\xi_Z),(\eta_w,\eta_Z)\bigr) \mapsto 
    \tran{\left( \xi_Z \Y^{-1}
        \mathrm{Im}(w) - \xi_w\right)}  \Y^{-1} \left( \overline{\eta_Z} \Y^{-1}
      \mathrm{Im}(w) -  \overline{\eta_w}\right),
  \end{equation*}
  for $\xi_w,\eta_w \in \IC^g$ and $\xi_Z,\eta_Z\in
  \mathrm{Mat}_g(\IC)$ symmetric, is Hermitian and so
  $\hat \omega^{\mathrm{univ}}$ is real. Moreover, 
    \begin{equation*}
    H\bigl((\xi_w,\xi_Z),(\xi_w,\xi_Z)\bigr) = \tran{v} \Y^{-1} \overline v
    \quad\text{with}\quad v =  \xi_Z \Y^{-1}
        \mathrm{Im}(w) - \xi_w.
  \end{equation*}
  But $\Y$ is positive definite as a real symmetric matrix and thus
  positive definite as a Hermitian matrix. We see that $H$
  is positive semi-definite and this implies that
  $\hat\omega^{\mathrm{univ}}$ is positive semi-definite.

  The ``moreover'' part of the lemma is clear.    
\end{proof}

Next we want to show that $\hat{\omega}^{\mathrm{univ}}$ descends to a $(1,1)$-form on $\an{\mathfrak{A}}_g$. To do this, we first show that $\hat{\omega}^{\mathrm{univ}}$ can be written in a simple form under an appropriate change of coordinates.

Define the complex space $\mathcal X_{2g,\mathrm{a}}$, which is the universal covering of $\an{\mathfrak{A}}_g$, as follows:
\begin{itemize}
\item As a real algebraic space, $\mathcal X_{2g,\mathrm{a}} := \mathbb R^{2g} \times \mathfrak{H}_g$.
\item The complex structure on $\mathcal X_{2g,\mathrm{a}}$ is given by
\begin{equation}\label{EqChangeOfCoor}
\mathbb R^{2g} \times \mathfrak{H}_g = \mathbb R^g \times \mathbb R^g \times \mathfrak{H}_g \cong \mathbb C^g \times \mathfrak{H}_g, \qquad (a,b,Z) \mapsto (a+Zb, Z).
\end{equation}
\end{itemize}

\begin{lemma}\label{LemmaBettiFormComplexReal}
Let $\hat{\omega}^{\mathrm{univ}}$ be as in
Lemma~\ref{LemmaBettiOnUniversalCovering}. Then under the change of
coordinates \eqref{EqChangeOfCoor}, we have 
$\hat{\omega}^{\mathrm{univ}}=2\tran{(\mathrm{d}a)} \wedge \mathrm{d}
b$.
\end{lemma}
\begin{proof} For the moment we write $Z = \X + \sqrt{-1}\Y$ with $\X$ and $\Y$ the real and
imaginary part of $Z\in\mathfrak{H}_g$, respectively. 
Note that $w = a + Zb = (a + \X b) + \sqrt{-1}\Y b$. Hence
$\Y^{-1}(\mathrm{Im}w) = b$ and $\mathrm{d}w = \mathrm{d}a +
Z\mathrm{d}b + \mathrm{d}Z  b$. Using this and noting that $Z$
is symmetric, we have that \eqref{eq:bettiformaltformula} becomes
\begin{equation*}
\begin{aligned}
    \hat \omega^{\mathrm{univ}} &=\sqrt{-1}\left( \sqrt{-1}\tran{(\rmd
        b)} \Y + \tran{(\rmd b)} \X + \tran{(\rmd a)}\right) \wedge \Y^{-1}
    \left(\rmd a + \X \rmd b - \sqrt{-1} \Y \rmd b\right) \\
    &=\sqrt{-1}\bigl(\sqrt{-1} \tran{(\rmd b)}\wedge\rmd a + \tran{(\rmd b)}\wedge
    \Y \rmd b + \tran{(\rmd b)} \X \wedge \Y^{-1} \rmd a+ \tran{(\rmd
      b)}\X\wedge 
    \Y^{-1}\X \rmd b + \\ 
    &\phantom{=\sqrt{-1}\bigl(}\tran{(\rmd a)} \wedge \Y^{-1} \rmd a 
    + \tran{(\rmd a)}\wedge \Y^{-1}\X\rmd b -
    \sqrt{-1} \tran{(\rmd a)}\wedge \rmd b\bigr).
\end{aligned}
\end{equation*}

 Many terms will vanish. Indeed, if $M$ is a
matrix, then $\tran{(\rmd b)} \wedge M \rmd a = - \tran{(\rmd
a)}\wedge \tran{M} \rmd b $. As $\tran{(\X\Y^{-1})} = \Y^{-1}\X$ and as
$\tran{(\rmd b)} \X \wedge \Y^{-1} \rmd a=\tran{(\rmd
b)} \wedge \X\Y^{-1} \rmd a$ we find
$\tran{(\rmd b)}\X \wedge \Y^{-1} \rmd a + \tran{(\rmd a)}\wedge
\Y^{-1}\X \rmd b=0$. Observe that $\Y$ is symmetric, and so $\tran{(\rmd
b)}\wedge \Y \rmd b = -\tran{(\rmd b)}\wedge \Y\rmd b$ vanishes. Arguing
along the same line and using that
$\Y^{-1}$ and $\X\Y^{-1}\X$ are  symmetric we find $\tran{(\rmd
a)}\wedge \Y^{-1} \rmd a = 0$ and $\tran{(\rmd b)}\X\wedge \Y^{-1}\X=\tran{(\rmd b)}\wedge \X\Y^{-1}\X \rmd
b=0$. We are left with
$    \hat \omega^{\mathrm{univ}} = 2\tran{(\rmd a)} \wedge \rmd b$. 
\end{proof}

\begin{corollary}\label{CorBettiFormVanishingDirection}
  Let $\hat{C}$ be an irreducible, $1$-dimensional, complex analytic
  subset of an open subset of $\mathcal X_{2g,\mathrm{a}}
  = \IR^{2g} \times \mathfrak{H}_g$ and $\sm{\hat {C}}$ its smooth
  locus. Then $\hat{\omega}^{\mathrm{univ}}$ restricted to $\sm{\hat C}$
  is trivial if and only if
  $\hat{C} \subseteq \{r\} \times \mathfrak{H}_g$ for some
  $r \in \mathbb R^{2g}$.
\end{corollary}
\begin{proof}
  First, assume that the coordinates $(a,b)$ of $\IR^{2g}$ are constant on
  $\hat C$. Then $\hat{\omega}^{\mathrm{univ}}$, which is simply
  $2 \tran{(\rmd a)} \wedge \mathrm{d}b$ by
  Lemma~\ref{LemmaBettiFormComplexReal}, vanishes on $\sm{\hat C}$.

  Conversely, suppose that $\hat{\omega}^{\mathrm{univ}}$ vanishes
  identically on $\sm{\hat C}$. This time we use (\ref{eq:bettiformaltformula}) from 
  Lemma~\ref{LemmaBettiOnUniversalCovering}. As $\Y^{-1}$ is positive
  definite we find $\rmd Z \Y^{-1} \mathrm{Im}(w) = \rmd w$ on $\sm{\hat C}$.
  Using the change of coordinates $w = a+Zb$ we deduce $\mathrm{Im}(w) = \Y b$
  and $\rmd w = \rmd a
  + \rmd Z b +Z\rmd b$. So $\rmd Z b  =\rmd Z \Y^{-1} \mathrm{Im}(w)
  = \rmd w=  \rmd a
  + \rmd Z b + Z\rmd b$ on $\sm{\hat C}$. This equality simplifies to
  $\rmd a + Z \rmd b = 0$ on $\sm{\hat C}$. As $a$ and $b$ are real
  valued and as $Z\in \mathfrak{H}_g$ we conclude $\rmd a = \rmd b =0$
  on $\sm{\hat C}$. So $a$ and $b$ are constant on $\hat C$.
\end{proof}

\begin{lemma}\label{LemmaBettiFormUnivAb}
Let $\hat{\omega}^{\mathrm{univ}}$ be as in
Lemma~\ref{LemmaBettiOnUniversalCovering}. Then
$\hat{\omega}^{\mathrm{univ}}$ descends to a  semi-positive
$(1,1)$-form $\omega^{\mathrm{univ}}$ on $\mathfrak{A}_g$. Moreover,
for $N\in\IZ$  we have $[N]^*\omega^{\mathrm{univ}} = N^2 \omega^{\mathrm{univ}}$.
\end{lemma}
\begin{proof}
Let $\mathrm{Sp}_{2g}$ be the symplectic group defined over $\mathbb Q$, and let $V_{2g}$ be the vector group over $\mathbb Q$ of dimension $2g$. Then the natural action of $\mathrm{Sp}_{2g}$ on $V_{2g}$ defines a group $P_{2g,\mathrm{a}}:= V_{2g} \rtimes \mathrm{Sp}_{2g}$.

We use the classical action of $\mathrm{Sp}_{2g}(\mathbb R)$  on
$\mathfrak{H}_g$, it is transitive. The real coordinate on $\mathcal X_{2g,\mathrm{a}}$ on the left hand side of \eqref{EqChangeOfCoor} has the following advantage. The group $P_{2g,\mathrm{a}}(\mathbb R)$ acts transitively on $\mathcal X_{2g,\mathrm{a}}$ by the formula
\[
(v,h) \cdot (v',Z) := (v+hv',hZ)
\]
for $(v,h) \in P_{2g,\mathrm{a}}(\mathbb R)$ and $(v',Z) \in \mathbb R^{2g} \times \mathfrak{H}_g = \mathcal X_{2g,\mathrm{a}}$. The space $\mathfrak{A}_g^{\mathrm{an}}$ is then obtained as the quotient of $\mathcal X_{2g,\mathrm{a}}$ by a congruence subgroup of $P_{2g,\mathrm{a}}(\mathbb Q)$. We refer to \cite[10.5--10.9]{PinkThesis} or \cite[Construction~2.9 and Example~2.12]{Pink05} for these facts.


It is clear that both $V_{2g}(\mathbb R)$ and
$\mathrm{Sp}_{2g}(\mathbb R)$ preserve $2\tran{(\rmd a)} \wedge \mathrm{d} b$. Thus this $2$-form is invariant under the action of $P_{2g,\mathrm{a}}(\mathbb R)$ on $\mathcal X_{2g,\mathrm{a}}$.

So by Lemma~\ref{LemmaBettiFormComplexReal}, the previous two paragraphs imply that $\hat{\omega}^{\mathrm{univ}}$ descends to a $(1,1)$-form $\omega^{\mathrm{univ}}$ on $\mathfrak{A}_g$. The semi-positivity of $\omega^{\mathrm{univ}}$ follows from Lemma~\ref{LemmaBettiOnUniversalCovering}.

The property $[N]^*\omega^{\mathrm{univ}} = N^2 \omega^{\mathrm{univ}}$ follows from the ``moreover'' part of Lemma~\ref{LemmaBettiOnUniversalCovering} and the following commutative diagram
\[
\begin{gathered}[b]
\xymatrix{
\IC^g \times \mathfrak{H}_g \ar[r]^-{\widetilde{N}} \ar[d] & \IC^g \times \mathfrak{H}_g \ar[d] \\
\an{\mathfrak{A}}_g \ar[r]^-{[N]} & \an{\mathfrak{A}}_g.
}\\[-\dp\strutbox]
\end{gathered}
\qedhere
\]
\end{proof}


This semi-positive $(1,1)$-form $\omega^{\mathrm{univ}}$ will be the Betti form for $\mathfrak{A}_g \rightarrow \mathbb A_g$, as desired in Proposition~\ref{PropBettiForm}. To show this, it suffices to establish property (iii) of Proposition~\ref{PropBettiForm}. Hence it suffices to prove the following proposition. 

\begin{proposition}\label{PropBettiFormNonDegUnivAb} 
Assume $\cA \rightarrow S$ is $\mathfrak{A}_g \rightarrow \mathbb A_g$. Let $X$ be an irreducible subvariety of $\mathfrak{A}_g$ of dimension
$d$ and let $\Delta$ be an open subset of $\an{S}$ with $\sman{X}\cap\cA_\Delta\not=\emptyset$. Then
\begin{equation}
\label{prop:bettformvsrankuniv}
 \omega^{\mathrm{univ}}|_{\sman{X}}^{\wedge d} \not\equiv 0
\quad\text{if and only if}\quad
\max_{x \in \sman{X}\cap\cA_\Delta} \mathrm{rank}_{\IR} (\mathrm{d}b_\Delta|_{\sman{X}})_{x} = 2d.
\end{equation}
\end{proposition}
\begin{proof}
We begin by  reformulating
Corollary~\ref{CorBettiFormVanishingDirection}.
If ${C}$ is an irreducible, $1$-dimensional, complex analytic subset
of an open subset of ${\cA_\Delta}$, then
\begin{equation}\label{EqBettiFormVanishingDirection}
{\omega}^{\mathrm{univ}}|_{\sm{C}} = 0 \quad\text{if
and only if}\quad b_\Delta({C})\text{ is a point};
\end{equation}
indeed, this claim is local and it follows using the universal
covering
$\mathbf{u} \colon \IC^g\times\mathfrak{H}_g \rightarrow \mathfrak{A}_g$. 

We assume first that the right side of (\ref{prop:bettformvsrankuniv})
is false, \textit{i.e.}, the maximal rank is strictly less than
$2d=2\dim X$.
So  every $x\in \sman{X}\cap\cA_\Delta$  is a non-isolated point of  
$b_\Delta^{-1}(r) \cap \sman{X}$  where $r =
b_\Delta(x)$. Because 
$b_\Delta^{-1}(r)$ is a complex analytic subset of $\cA_\Delta$ (by Proposition~\ref{PropBettiMap}.(ii) for $\mathfrak{A}_g \rightarrow \mathbb{A}_g$) and $\sman{X}$ is complex
analytic in a neighborhood of $x$ in $\an{\cA}$, there exists  an irreducible
complex analytic curve ${C}$ in  $b_\Delta^{-1}(r) \cap \sman{X}$ passing
through $x$. In
particular, $b_\Delta({C})$ is a point
and so ${\omega}^{\mathrm{univ}}|_{\sm{C}} \equiv 0$ by \eqref{EqBettiFormVanishingDirection}.

The upshot of the previous paragraph is that the Hermitian form
attached to the semi-positive $(1,1)$-form ${\omega}^{\mathrm{univ}}|_{\sman{X}}$
 vanishes along the tangent space of $\sm{C}$; it is degenerate.
We can complete a tangent vector of $\sm{C}$  to a basis of
the tangent space of $\sman{X}$. Considering holomorphic local
coordinates
we find ${\omega}^{\mathrm{univ}}|_{\sman{X}}^{\wedge d}=0$ at every
point of $\sm{C}$. By continuity it also vanishes at $x \in C$. Since
$x\in \sman{X}\cap\cA_\Delta$ was arbitrary, we conclude
${\omega}^{\mathrm{univ}}|_{\sman{X}}^{\wedge d}\equiv 0$. 

For the converse we assume
${\omega}^{\mathrm{univ}}|_{\sman{X}}^{\wedge d}\equiv 0$. So the
Hermitian form attached to this semi-positive $(1,1)$-form is
degenerate. Thus for each $x \in \sman{X}$, using holomorphic local coordinates we find an irreducible,  $1$-dimensional,  complex analytic subset $C_x$ which passes through $x$ and is contained in $\sman{X}$ such that $\omega^{\mathrm{univ}}|_{\sman{X}}$ vanishes along the tangent space of $\sm{C}_x$. So ${\omega}^{\mathrm{univ}}|_{\sm{C}_x}\equiv 0$, and hence 
 $b_\Delta(C_x)$ is a point
by \eqref{EqBettiFormVanishingDirection}.
Letting the point $x$ run over $\sman{X}$, we conclude that the rank on the right
side of (\ref{prop:bettformvsrankuniv}) is strictly less than $2d$. 
\end{proof}

\subsection{General case}
We now prove Propositions~\ref{PropBettiMap}
and \ref{PropBettiForm} for
 $\pi \colon \cA \rightarrow S$ as near the beginning of this section.
 In particular, we assume \texttt{(Hyp)}. With the construction in $\mathsection$\ref{SubsectionBettiMap}, the rest of the proof of Proposition~\ref{PropBettiMap} follows the construction in \cite[$\mathsection$4]{GaoBettiRank}. 


As
$\IA_g$ is a fine moduli space there exists a Cartesian diagram
\begin{equation*}
\xymatrix{
\cA \ar[r]^-{\iota} \ar[d]_{\pi} \pullbackcorner & \mathfrak{A}_g \ar[d] \\
S \ar[r]^-{\iota_S} & \mathbb{A}_g.
}
\end{equation*}

Now let $s_0 \in S(\IC)$. Applying Proposition~\ref{PropBettiMap} to the universal abelian variety $\mathfrak{A}_g \rightarrow \mathbb A_g$ and $\iota_S(s_0) \in \mathbb A_g(\IC)$, we obtain an open neighborhood $\Delta_0$ of $\iota_S(s_0)$ in $\mathbb A_g^{\mathrm{an}}$ and a map
\[
b_{\Delta_0} \colon \mathfrak{A}_g|_{\Delta_0} \rightarrow \mathbb T^{2g}
\]
satisfying the properties listed in Proposition~\ref{PropBettiMap}.

Now let $\Delta = \iota_S^{-1}(\Delta_0)$. Then $\Delta$ is an open neighborhood of $s$ in $S^{\mathrm{an}}$. Denote by $\cA_{\Delta} = \pi^{-1}(\Delta)$ and define
\[
b_{\Delta} = b_{\Delta_0} \circ \iota \colon \cA_{\Delta} \rightarrow \mathbb T^{2g}.
\]
Then $b_{\Delta}$ satisfies the properties listed in Proposition~\ref{PropBettiMap} for $\cA \rightarrow S$. Hence $b_{\Delta}$ is our desired Betti map.

Next let us turn to the Betti form. Let $\omega^{\mathrm{univ}}$ be the semi-positive $(1,1)$-form on $\mathfrak{A}_g$ as in Lemma~\ref{LemmaBettiFormUnivAb}. Define $\omega:= \iota^*\omega^{\mathrm{univ}}$. We will show that $\omega$ satisfies the properties listed in Proposition~\ref{PropBettiForm}.

The $(1,1)$-form $\omega$ is semi-positive as it is the pull-back of
the semi-positive form $\omega^{\mathrm{univ}}$. Moreover, it satisfies
$[N]^*\omega = N^2 \omega$ since $\omega^{\mathrm{univ}}$ has this
property. Hence we have established properties (i) and (ii) of
Proposition~\ref{PropBettiForm}.

Let us verify (iii) of Proposition~\ref{PropBettiForm}. Suppose $X$ is
an irreducible subvariety of $\cA$ of dimension $d$.
Let $\Delta$ be an open subset of $\an{S}$ with
$\sman{X}\cap\cA_\Delta\not=\emptyset$; we may shrink $\Delta$ subject
to this condition. 
Let $Z
= \overline{\iota(X)}$ and observe $\dim Z \le d$. 

Since $\omega = \iota^*\omega^{\mathrm{univ}}$, we have
\begin{equation}\label{EqOmegaAndOmegaUniv}
\omega|_{\sman{X}}^{\wedge d} \not\equiv 0 \quad\text{if and only if}\quad \omega^{\mathrm{univ}}|_{\sman{Z}}^{\wedge d} \not\equiv 0.
\end{equation}
Next by definition of $b_{\Delta}$, we have the following property:
 For suitable non-empty open subsets $\Delta$ of $S^{\mathrm{an}}$ and
 $\Delta_0$ of $\mathbb{A}_g^{\mathrm{an}}$
 such that $\iota_S(\Delta) \subseteq \Delta_0$, we have
\begin{equation}\label{EqBettiRankTrivialBound}
\max_{x \in \sman{X} \cap \cA_{\Delta}} \mathrm{rank}_{\IR}
(\mathrm{d}b_{\Delta}|_{\sman{X}})_x
\le \max_{x \in \sman{Z} \cap \mathfrak{A}_{g,\Delta_0}} \mathrm{rank}_{\IR} (\mathrm{d}b_{\Delta_0}|_{Z^{\mathrm{sm,an}}})_x \le 2\dim Z \le 2d.
\end{equation}

Suppose first that $\omega|_{\sman{X}}^{\wedge d}\not\equiv 0$, then
\eqref{EqOmegaAndOmegaUniv} implies
$\omega^{\mathrm{univ}}|_{\sman{Z}}^{\wedge d}\not\equiv 0$ and in
particular $d=\dim Z$. 
We can apply 
Proposition~\ref{PropBettiForm}(iii)  to $Z$ and obtain
$ \max_{x \in \sman{Z} \cap \mathfrak{A}_{g,\Delta_0}} \mathrm{rank}_{\IR}
(\mathrm{d}b_{\Delta_0}|_{Z^{\mathrm{sm,an}}})_x =2d$. 
Now $\iota|_X\colon X\rightarrow Z$ is
generically finite as $\dim X = \dim Z$,
so the first inequality in \eqref{EqBettiRankTrivialBound} is an
equality. We conclude
\begin{equation}
\label{eq:maxrankXgeneralA}
\max_{x \in \sman{X} \cap \cA_\Delta} \mathrm{rank}_{\IR}
(\mathrm{d}b_{\Delta}|_{X^{\mathrm{sm,an}}})_x =2d.
\end{equation}

Conversely, assume \eqref{eq:maxrankXgeneralA} holds true.
Then we have equalities throughout in \eqref{EqBettiRankTrivialBound}.
By Proposition~\ref{PropBettiForm}(iii) applied to $Z$ and
by \eqref{EqOmegaAndOmegaUniv} we get $\omega|_{\sman{X}}^{\wedge
d}\not\equiv 0$.


\section{Setup and notation for the height inequality}\label{SectionSettingUpHtIneq}

In the next few sections we will prove Theorem~\ref{ThmHtInequality}.
Let us first fix  the setting.

All varieties are over an algebraically closed subfield $k$ of $\IC$.
The ambient data is given as above Theorem~\ref{ThmHtInequality}. We
repeat it here.
\begin{itemize}
  \item Let $S$ be a regular, irreducible, quasi-projective variety over  $k$
  that is Zariski open in an irreducible projective variety
  $\overline{S}\subset \IP_k^m$.
  \item Let $\pi\colon \cA\rightarrow S$ be an abelian scheme
     presented by a closed immersion $\cA\rightarrow\IP_k^n\times S$ over
     $S$. 
  \item
    From the previous point, we get a closed immersion of the generic
    fiber $A$ of $\cA\rightarrow S$  into  $\IP_{k(S)}^n$.
    We assume that $A\rightarrow \IP_{k(S)}^n$  
    arises from a basis of the global sections of 
    the $l$-th power $L$ of a symmetric ample
    line bundle with $l\ge 4$. 
  \item
    Finally, we assume \texttt{(Hyp)} as on page \pageref{def:hyp}.
\end{itemize}

From the third bullet power, we see that the image of $A$ is
projectively normal in $\IP_{k(S)}^n$, \textit{cf}.~\cite[Theorem 9]{Mumford:quadeq}. By the fourth bullet point,
Proposition~\ref{PropBettiForm} provides the Betti form $\omega$ on
$\an{\cA}$.

For $s\in S(k)$ we write $\cA_s$ for the abelian variety $\pi^{-1}(s)$. 

\begin{remark}
\label{rmk:projembedding}
Let $S$ be as in the first bullet point. 
Let
$\pi \colon \cA \rightarrow S$ be an abelian scheme.
Suppose $L_0$ is a 
symmetric and ample line bundle on $A$, the generic fiber of $\pi$.
An immersion of $\cA$ as in the second bullet point can  be obtained
as follows. 
 By \cite[Th\'eor\`eme~XI~1.13]{LNM119}  there exists an
$S$-ample line bundle $\cL$ on $\cA$ whose restriction to the
generic fiber of $\cA\rightarrow S$ is isomorphic to $L_0^{\otimes l}$
for some integer $l\ge 4$. 
We may assume in addition that $\cL$
satisfies
$[-1]^* \cL \cong \cL$ and even that $\cL$ becomes trivial when pulled back under the
zero section $S\rightarrow \cA$, see \cite[Remarque~XI~1.3a]{LNM119}.
After replacing $\cL$
by  a sufficiently high power, we may assume that $\cL$ is very ample over
$S$.
We fix a basis of global sections of $L_0^{\otimes l}$ and, as $l\ge
4$, thereby
realize the
 generic fiber of $\pi$ as a projectively normal subvariety of 
$\IP_{k(S)}^n$.
Now we can take $L$ in the third bullet point to be $L_0^{\otimes l}$,
which is the restriction of $\cL$ to $A$.
A closed immersion $\cA \rightarrow \IP^n_k \times S$ as in the second
bullet point 
 arises from 
$\cL \otimes \pi^*\cM$ for some very ample line bundle $\cM$ on $S$; see \cite[Proposition~4.4.10.(ii) and Proposition~4.1.4]{EGAII}.
On restricting to a fiber of $\cA\rightarrow S$ the induced closed immersion  $\cA_s\rightarrow\IP^n_k$ comes from the restriction $\cL|_{\cA_s}$. 
\end{remark}

Write $\overline \cA$ for the Zariski closure
of $\cA$ in $\IP_k^n\times \overline{S}$. Then $\overline \cA$ is irreducible
but not necessarily regular.
On any product of $r$ projective spaces and if $a_1,\ldots,a_r\in\IZ$,
we let  $\cO(a_1,\ldots,a_r)$ denote the tensor product over all $i\in
\{1,\ldots, r\}$ of the
pull-back under the $i$-th projection of $\cO(a_i)$. 
We write $\cL$ for the restriction of $\cO(1,1)$ to $\overline{\cA}$. 

\subsection{Height functions on $\cA$}\label{SubsectionHeight}
If $k=\IQbar$ we have several height functions on $\cA(\IQbar)$.

For any $n \in \IN$, we always consider the absolute logarithmic 
Weil height function $\IP_{\IQbar}^n(\IQbar) \rightarrow \IR$, or just 
Weil height, defined as in \cite[$\mathsection$1.5.1]{BG}.

Now say $P\in \cA(\IQbar)$, we write $P=(P',\pi(P))$ with
$P'\in \IP_{\IQbar}^n(\IQbar)$ and
$\pi(P)\in \IP_{\IQbar}^m(\IQbar)$. The sum of
Weil heights
\begin{equation}\label{EqHeightTotal}
h(P) = h(P') + h(\pi(P))
\end{equation}
defines our first height $\cA(\IQbar)\rightarrow [0,\infty)$ which we
call the \textit{naive height} on $\cA$. It depends on the fixed
immersion of $\cA$.

The line bundle $[-1]^*\cL|_{\cA} \otimes\cL|_{\cA}^{\otimes -1}$ of
${\cA}$ restricted to the generic fiber $A$ of $\cA\rightarrow S$
equals $[-1]^*L \otimes L^{\otimes -1}$. By the third bullet point above this line
bundle is trivial. So it equals $\pi^*\cK$ for some line bundle $\cK$
of $S$ by \cite[Corollaire~21.4.13 (pp.~361 of EGA IV-4, in Errata et Addenda, liste~3)]{EGAIV}. We conclude
$[-1]^*\cL|_{\cA_s}\cong \cL|_{\cA_s}$ for all $s\in S(\IQbar)$. So
the function (\ref{EqHeightTotal}) represents the height function,
defined up-to $O_s(1)$, given by the Height
Machine, \textit{cf.} \cite[Theorem 2.3.8]{BG}, applied to
$(\cA_s,\cL_s)$. As $\cL_s$ is symmetric, the \textit{fiberwise
N\'eron--Tate} or \textit{canonical height} $\hat
h_{\cA} \colon \cA(\IQbar)\rightarrow [0,\infty)$, defined by the
convergent limit
\begin{equation}
\label{eq:fiberwiseNT}
\hat h_{\cA}(P) = \lim_{N\rightarrow\infty} \frac{h([N](P))}{N^2},
\end{equation}
is a quadratic form on $\cA_s(\IQbar)$. In the notation \cite[Chapter~9]{BG} the height
(\ref{eq:fiberwiseNT}) is $\hat h_{\cA_s,\cL_s}$ where $s=\pi(P)$.

\begin{remark}\label{RemarkHtcompatible}
  We use here the notation of Remark~\ref{rmk:projembedding}.
  So that the immersion $\cA\rightarrow \IP^n_S$ arises via
  $L_0^{\otimes l}$. To normalize,
  we divide (\ref{eq:fiberwiseNT}) by $l$ and obtain the 
  N\'eron--Tate height $\hat h_{\cA,L_0}\colon \cA(\IQbar)\rightarrow
  [0,\infty)$.

  Let us verify that $\hat h_{\cA,L_0}$ depends only on $L_0$. Suppose
  $\cL'$ is another line bundle on $\cA$ that restricts to
  $L_0^{\otimes l}$,
  then $\cL'\otimes\cL^{\otimes -1}$ is trivial on $A$. By \cite[Corollaire~21.4.13 (pp.~361 of EGA IV-4, in Errata et Addenda, liste~3)]{EGAIV}, this difference is the pull-back of some line
  bundle on $S$ under $\cA\rightarrow S$. So the restriction of
  $\cL'\otimes\cL^{\otimes -1}$ to $\cA_s$ for each $s\in S(\IC)$ is
  trivial. Thus $\cL|_{\cA_s}$ and $\cL'|_{\cA_s}$ induce the same N\'eron--Tate
  height on $\cA_s(\IQbar)$, see~\cite[$\mathsection$9.2]{BG}. 
\end{remark}

\subsection{Integration against the Betti form}\label{SubsectionSetUpHtIneqBetti}
Let $\cA$  and $S$ be as in the beginning of this section, so they are
defined over an algebrically closed subfield $k$ of $\IC$. 
Recall that $\omega$ is the Betti form on $\an{\cA}$ as provided by
Proposition~\ref{PropBettiForm}. In particular, it is a
semi-positive $(1,1)$-form on $\an{\cA}$ such that $[N]^*\omega =
N^2 \omega$ for all $N\in\IZ$. We discuss here a modification of  the
Betti form that has compact support.

Fix $X$ to be an irreducible closed subvariety of $\cA$ of dimension $d$, such that $\pi|_X \colon X \rightarrow S$ is dominant. 

We are not allowed to integrate $\omega^{\wedge d}$ over $\sman{X}$ as $\omega^{\wedge d}$ may not have compact support. So we modify $\omega$ in the following way.

Suppose we are provided with a base point $s_0 \in \an{S}$. Let
furthermore $\Delta$ be a relatively compact, contractible,
open neighborhood of $s_0$ in $\an{S}$. Denote by $\cA_{\Delta}$ the open
subset $\pi^{-1}(\Delta)$ of $\an{\cA}$. Fix a smooth bump function $\vartheta \colon \an{S}\rightarrow [0,1]$ with compact support $K \subset \Delta$ such that $\vartheta(s_0)=1$. Finally, we define $\theta = \vartheta\circ\pi \colon \an{\cA} \rightarrow [0,1]$.
Then $\theta\omega$ is a semi-positive smooth $(1,1)$-form on $\an{\cA}$;
unlike the Betti form, it may not be closed. By construction, the support
of $\theta\omega$ lies in $\pi^{-1}(K)$ which is compact as $\pi$ is proper and $K$ is compact. 

\begin{remark}
  \label{rmk:nondeg}
  Suppose $X$ is non-degenerate, namely $X$ satisfies one of the two
  equivalent conditions in property (iii) of
  Proposition~\ref{PropBettiForm}. Then $\an{X}$ contains a smooth
  point $P_0$ at which $\omega|_{\sman{X}}^{\wedge d}>0$. Then we will
  take $s_0=\pi(P_0)$.
\end{remark}

\subsection{The graph construction}
\label{ss:graph}

Let $N\in\IZ$. The multiplication-by-$N$ morphism $[N] \colon \cA\rightarrow \cA$ may not
extend to a morphism $\overline{\cA}\rightarrow \overline{\cA}$. We
overcome this by using the graph construction.

Recall that we have identified
$\cA\subset\overline{\cA}\subset \IP_k^n\times \overline{S}
\subset\IP_k^n\times\IP_k^m$. 

We write
$\rho_1,\rho_2\colon \IP_k^n\times\IP_k^n \times \IP_k^m \rightarrow
\IP_k^n\times \IP_k^m$ for the
 two projections $\rho_1(P,Q,s)=(P,s)$ and $\rho_2(P,Q,s)=(Q,s)$.

Consider $\Gamma_N$ the graph of $[N]$, determined by 
\begin{equation*}
  \Gamma_N = \{ (P,[N](P)) : P \in \cA(k) \}. 
\end{equation*}
We consider it as an irreducible closed subvariety of
$\cA\times_S\cA$. 

Let $X$ be an irreducible closed subvariety of $\cA$ of dimension $d$. 
The graph $X_N$ of $[N]$ restricted to $X$ is
an irreducible closed  subvariety of $\Gamma_N$ determined by 
\begin{equation*}
   \{(P,[N]P): P \in X(k) \}.
\end{equation*}

Observe that $\rho_1|_{\Gamma_N} \colon \Gamma_N\rightarrow \cA$ is an isomorphism; it maps $(P,[N](P))$ to $P$.  
So we can use $\rho_1|_{\Gamma_N}^{-1}$ to identify $X$ with $X_N$.

Moreover, $\rho_2|_{\Gamma_N}$ maps $(P,[N](P))$ to $[N](P)$. Therefore
\begin{equation}
\label{eq:relatetoN}
  \rho_2|_{\Gamma_N} \circ \rho_1|_{\Gamma_N}^{-1} = [N].
\end{equation}

Let $\overline{X_N}$ be the Zariski closure of $X_N$ in
$\overline{\cA}\times_{\overline{S}}\overline{\cA}\subset
\IP^n_{\overline{S}}\times_{\overline{S}}\IP^n_{\overline{S}} =
\IP_k^n\times\IP_k^n\times \overline{S}
\subset\IP_k^n\times\IP_k^n\times \IP_k^m$.  Then $\overline{X_N}$ is an irreducible projective variety (which is not necessarily regular) with $\dim \overline{X_N} = \dim X_N = \dim X$.

In the next section, we will use the following line bundles on $\overline{X_N}$. Define
\begin{equation}\label{def:cF}
\cF = \rho_2^*\cO(1,1)|_{\overline{X_N}} = \cO(0,1,1)|_{\overline{X_N}}
\end{equation}
and
\begin{equation}\label{def:cM}
  \cM 
  = \cO(0,0,1)|_{\overline{X_N}}.
\end{equation}

Let us close this subsection by relating the height functions defined
by $\cF$ and $\cM$ with the ones in
$\mathsection$\ref{SubsectionHeight}. Assume $k = \IQbar$. Let $P \in
X(\IQbar)$. Write $P = (P',\pi(P))$ with
 $P' \in \IP_\IQbar^n(\IQbar)$ and $\pi(P) \in\IP_\IQbar^m(\IQbar)$. 
We have $[N](P) = ( P'_N,\pi(P))$ for some $P'_N \in \IP_\IQbar^n(\IQbar)$.

Under the immersion $\overline{X_N} \subseteq \IP_\IQbar^n \times
\IP_\IQbar^n \times \IP_\IQbar^m$, the point $(P,[N]P)$ in
$\overline{X_N}(\IQbar)$ becomes $P_N=(P',P'_N,\pi(P)) \in
(\IP_\IQbar^n \times \IP_\IQbar^n \times \IP_\IQbar^m)(\IQbar)$.
The function $P_N\mapsto h([N](P)) = h(P'_N)+h(\pi(P))$ defined in
\eqref{EqHeightTotal}
represents the height attached by the Height Machine to
$(\overline{X_N},\cF)$
and $P_N\mapsto h(\pi(P))$ represents the height attached to $(\overline{X_N},\cM)$.


\section{Intersection theory and height inequality on the total space}\label{SectionIntThAndHtIneq}

We keep the notation of $\mathsection$\ref{SectionSettingUpHtIneq}. 
So we have a closed immersion $\cA\rightarrow \IP_k^n\times
S$ over $S$ satisfying the properties stated near the beginning 
of $\mathsection$\ref{SectionSettingUpHtIneq}. Moreover, $S$ is a
Zariski open subset of an irreducible projective variety $\overline
S\subset\IP_k^m$. 
We assume in addition $k = \IQbar$. Let $X$ be a closed irreducible subvariety of $\cA$ of dimension $d$ defined over $\IQbar$, such that $\pi|_X \colon X\rightarrow S$ is dominant. Let $\omega$ be the Betti form on $\cA$ as defined in Proposition~\ref{PropBettiForm}.
\begin{proposition}\label{PropAuxHtIneq}
We keep the notation from above and
suppose that $\an{X}$ contains a smooth point at which
$\omega|_{\sman{X}}^{\wedge d} >0$. Then there exists a constant $c_1
> 0$ satisfying the following property.
Let $N \in \IN$ be a power of $2$, there exist a Zariski open dense subset $U_N$ of $X$ defined over $\IQbar$ and a constant $c_2(N)$ such that
\[
h([N]P) \ge c_1 N^2 h(\pi(P)) - c_2(N)\qquad \text{for all }P \in U_N(\IQbar).
\]
\end{proposition}

The goal of this section is to prove Proposition~\ref{PropAuxHtIneq}. 
The key idea is to apply a theorem of Siu \cite[Theorem~2.2.15]{PosAlgGeom}. Let us briefly explain the main points before moving on to the proof.

Let $X$ be as in Proposition~\ref{PropAuxHtIneq}, and let $P \in
X(\IQbar)$. For each $N \in \IN$, we work with $X_N \subseteq
\cA\times_S \cA$, the graph of $[N] \colon \cA \rightarrow \cA$, and
its Zariski closure $\overline{X_N}$ in
$\overline{\cA}\times_{\overline{S}}\overline{\cA} \subseteq \IP_k^n
\times \IP_k^n \times \IP_k^m$. The point $P$ gives rise to a point
$P_N \in \overline{X_N}$; see $\mathsection$\ref{ss:graph}. Consider
the line bundles $\cF = \cO(0,1,1)|_{\overline{X_N}}$ and $\cM =
\cO(0,0,1)|_{\overline{X_N}}$. Choosing representatives as in
last paragraph of
$\mathsection$\ref{ss:graph}  
 our height inequality in Proposition~\ref{PropAuxHtIneq} is equivalent to
\[
h_{\overline{X_N},\cF}(P_N)\ge c_1 N^2 h_{\overline{X_N},\cM}(P_N) - c_2'(N)
\]
for some $c_2'(N)$ independent of $P$ (which may be different from
$c_2(N)$). By the Height Machine  it suffices to find positive
integers $p$ and $q$, independent of $N$, such that $\cF^{\otimes q}
\otimes \cM^{\otimes -p N^2}$ is a big line bundle on
$\overline{X_N}$; we can then take $c_1=p/q$. 

Both $\cF$ and $\cM$ are nef line bundles. Thus a criterion of bigness
by Siu \cite[Theorem~2.2.15]{PosAlgGeom}, states that $\cF^{\otimes q}
\otimes \cM^{\otimes -p N^2}$ is big if $(\cF^{\cdot d}) > d c_1
(\cM^{\otimes N^2} \cdot \cF^{\cdot (d-1)})$.
Note that $(\cM^{\otimes N^2} \cdot \cF^{\cdot (d-1)}) = N^2 (\cM \cdot \cF^{\cdot (d-1)})$ by multi-linearity of intersection numbers. 
Thus our task 
becomes comparing two intersection numbers. Our application continues
to work if the numerical
factor $d = \dim X$ is replaced by any
positive  factor that depends only on the dimension.
So it remains to prove an appropriate lower bound for $(\cF^{\cdot d})$ and an appropriate upper bound for $(\cM \cdot \cF^{\cdot (d-1)})$.

The proof of Proposition~\ref{PropAuxHtIneq} will be organized as follows in this section. We first prove the appropriate lower bound for $(\cF^{\cdot d})$ in Proposition~\ref{prop:intersectionlb}. This is where we use the hypothesis that $\omega|_{\sman{X}}^{\wedge d} >0$ at some smooth point of $\an{X}$. Next we prove the appropriate lower bound for $(\cM \cdot \cF^{\cdot (d-1)})$ in Proposition~\ref{prop:intersectionub}. At this step the assumption of $N$ being a power of $2$ is used. Then we finish the proof of Proposition~\ref{PropAuxHtIneq} by applying Siu's theorem in $\mathsection$\ref{SubsectionProofOfAUxHtIneq}.

\subsection{Bounding an intersection number from below}
Let $X$ be as in Proposition~\ref{PropAuxHtIneq}. 
For each $N \in \IN$, let $\overline{X_N} \subseteq \IP_k^n \times
\IP_k^n \times \IP_k^m$ be as in $\mathsection$\ref{ss:graph}.
In particular, $\dim \overline{X_N} = d$.
 Let $\cF =
\cO(0,1,1)|_{\overline{X_N}}$ be as in \eqref{def:cF}. The top
self-intersection of $\cF$ on $\overline{X_N}$ is bounded from below
in the following proposition.
To prove it, we may replace $X$ by its base change to $\IC$. 

\begin{proposition}
  \label{prop:intersectionlb}
  Suppose $\an{X}$ contains a smooth point at which
  $\omega|_{\sman{X}}^{\wedge d} >0$. Then there exists a constant $\kappa
  > 0$, independent of $N$, such that $(\cF^{\cdot d}) \ge \kappa N^{2d}$ for all
  $N\in\IN$.   
\end{proposition}
\begin{proof}
We fix a point  $P_0 \in \sman{X}$ at which
$\omega|_{\sman{X}}^{\wedge d}$ is positive and let $s_0=\pi(P_0),\Delta,\vartheta,\theta,$
and $K$  be as in $\mathsection$\ref{SubsectionSetUpHtIneqBetti}, see 
Remark~\ref{rmk:nondeg}. In particular $\vartheta(P_0) = \theta\circ \pi(P_0) = 1$. 
We extend $\vartheta$ to a smooth function on  $\an{(\IP^m_\IC)}$ by
setting it $0$ outside of the compact set $K\subset \an{S}$. This
extends $\theta = \vartheta\circ\pi$ to all of
$\an{(\IP^n_\IC\times\IP^m_\IC)}$. 

Let $\alpha$ be the pull-back of the Fubini--Study form 
 under the analytification of the Segre morphism
$\IP_{\IC}^n\times\IP_{\IC}^m \rightarrow \IP_{\IC}^{(n+1)(m+1)-1}$.
We replace $\alpha$ by its restriction to $\an{\overline\cA}$. 
Thus $\alpha$
represents the Chern class of
$\cO(1,1) \in \mathrm{Pic}(\IP^n_{\IC}\times\IP_{\IC}^m)$ restricted
 to $\an{\overline\cA}$, using common notation.

Note that $\alpha$ is strictly positive on all of
$\an{\cA}$.  Since $\Delta$ is relatively compact we can find a constant $C>0$
with
\begin{equation}\label{EqConditionOnCapitalC}
 C \alpha|_{\cA_{\Delta}} - \omega|_{\cA_{\Delta}} \ge 0.
\end{equation}
As the smooth and non-negative function $\theta = \vartheta\circ\pi$ on $\an{\cA}$ has support in
$\pi^{-1}(K)\subset \pi^{-1}(\Delta)=\cA_{\Delta}$ we have
\begin{equation*}
  C \theta\alpha - \theta\omega \ge 0.
\end{equation*}

We pull this $(1,1)$-form back under the holomorphic map
$[N]\colon \an{\cA}\rightarrow\an{\cA}$ and get
\begin{equation}
\label{eq:pbsemipositive}
C  [N]^* (\theta\alpha) - N^2\theta\omega =  C [N]^*(\theta\alpha) - [N]^*(\theta\omega)\ge 0
\end{equation}
where we used $[N]^*\omega = N^2\omega$ and $[N]^* \theta = \theta$; 
the former is a property of the Betti form, see
Proposition~\ref{PropBettiForm}(ii) and
the latter holds as $\theta$ is the pull-back from the base of $\vartheta$.

We define
\begin{equation*}
  \beta = C  [N]^*(\theta\alpha) - N^2\theta\omega,
\end{equation*}
which is a  $(1,1)$-form on $\an{\cA}$.
It is semi-positive by (\ref{eq:pbsemipositive}). 
The support of $\theta$ is
contained in $\pi^{-1}(K)$, which we have identified as compact at the
end of $\mathsection$\ref{SubsectionSetUpHtIneqBetti}.
So
$C[N]^*(\theta\alpha)$ and $N^2\theta\omega$ have compact support on $\an{\cA}$. 

We claim that  $\int_{\sman{X}} (C[N]^*(\theta\alpha))^{\wedge d} \ge
\int_{\sman{X}} (N^2 \theta\omega)^{\wedge d}$.

First observe that both integrals are well-defined as both
$[N]^*(\theta\alpha)$ and $N^2\theta\omega$ have compact support on
$\an{\cA}$; this follows from work of Lelong~\cite{Lelong} which we
use freely below. A textbook proof can be found in
\cite[Theorem~11.21]{voisin:hodge1} and
\cite[$\mathsection$III.2.B]{Demailly}. 
To prove the inequality let us  write $\beta = \gamma - \delta$ with $\gamma = C[N]^*(\theta
\alpha)$ and $\delta = N^2\theta\omega$.
Then
\begin{equation}
\label{eq:diffintegral}  
  \int_{\sman{X}} \gamma^{\wedge d} - \int_{\sman{X}}\delta^{\wedge d}
  = 
\int_{\sman{X}} (\delta+\beta)^{\wedge d} - \int_{\sman{X}} \delta^{\wedge
    d}
= \sum_{i=0}^{d-1} {d \choose i} \int_{\sman{X}} \delta^{\wedge
  i}\wedge \beta^{\wedge (d-i)}
\end{equation}
as the exterior product is commutative on even degree forms. 
We know that $\beta \ge 0$ on $\an{\cA}$ and it is also crucial that
$\delta \ge 0$ on $\an{\cA}$, the latter
follows from $\omega \ge 0$, property (i) of
Proposition~\ref{PropBettiForm},
and from $\theta\ge 0$.
Then $\delta^{\wedge i}\wedge \beta^{\wedge (d-i)}$ is semi-positive
on $\an{\cA}$; see \cite[Proposition~III.1.11]{Demailly}\footnote{As our convention  is somewhat different from Demailly's, let us explain how to apply  \cite[Proposition~III.1.11]{Demailly}. Our definition of semi-positive $(1,1)$-form coincides with that of \textit{positive $(1,1)$-form} of \cite[Chapter~III]{Demailly} by Corollary~1.7 of \textit{loc.cit.}, and thus are precisely the strongly positive $(1,1)$-forms of \cite[Chapter~III]{Demailly} by Corollary~1.9 of \textit{loc.cit.} Therefore we can apply the cited proposition.}. Thus the right-hand side of (\ref{eq:diffintegral}) is
 non-negative, see \cite[Theorem~III.2.7]{Demailly}, and our claim is settled.

The claim implies
\begin{equation}
\label{eq:N2dgrowth}
 C^d \int_{\sman{X}}  [N]^*(\theta\alpha)^{\wedge d} \ge
\kappa' N^{2d} \quad\text{where}\quad \kappa' = \int_{\sman{X}} (\theta\omega)^{\wedge d}.
\end{equation}
We have $\kappa' > 0$. Indeed, $(\theta\omega)^{\wedge d}$ is semi-positive on $\an{\cA}$ because $\omega \ge 0$ (Proposition~\ref{PropBettiForm}(i)) and $\theta \ge 0$ (by construction). But $\omega|_{\sman{X}}^{\wedge d}$ is positive at $P_0 \in \sman{X}$ by choice of $P_0$ and $\theta\circ \pi(P_0) = 1$ by choice of $\theta$. So $(\theta\omega)|_{\sman{X}}^{\wedge d}$ is positive at $P_0 \in \sman{X}$. Thus $\kappa' > 0$.

Next we want to relate the integral on the left in (\ref{eq:N2dgrowth}) with an
intersection number. First we recall that $[N]$ is given in terms of
the graph construction, \textit{cf.} (\ref{eq:relatetoN}). So we may
rewrite
\begin{equation}
\label{eq:chainbegins}
  \int_{\sman{X}} [N]^*(\theta\alpha)^{\wedge d}  =
  \int_{\sman{X}} (\rho_2|_{\Gamma_N} \circ
  \rho_1|_{\Gamma_N}^{-1})^*(\theta\alpha)^{\wedge d}=
  \int_{\sman{X}} (\rho_1|_{\Gamma_N}^{-1})^*\rho_2|_{\Gamma_N}^*(\theta\alpha)^{\wedge d},
\end{equation}
here $\Gamma_N,\rho_1,$ and $\rho_2$ are as defined in $\mathsection$\ref{ss:graph}. 

Because $\rho_1|_{\an{\Gamma}_N} \colon \an{\Gamma}_N \rightarrow \an{\cA}$ is biholomorphic we
can change coordinates and integrate over $X_N$, which is a complex
analytic subset of the graph $\Gamma_N$, itself a complex manifold.
More precisely, we have
\begin{equation}
\label{eq:chain2} 
  \int_{\sman{X}} (\rho_1|_{\Gamma_N}^{-1})^*\rho_2|_{\Gamma_N}^*(\theta\alpha)^{\wedge d}
=\int_{\sman{X}_N}  \rho_2|_{\Gamma_N}^*(\theta\alpha)^{\wedge d}.
\end{equation}

Recall that $\alpha$ is the restriction to $\an{\overline{\cA}}$ 
of a  $(1,1)$-form on  
$\an{(\IP^n_{\IC}\times\IP^m_{\IC})}$.
Moreover,
$\rho_2$ is also defined on all of $\IP^n_{\IC} \times\IP^n_{\IC}\times \IP^m_{\IC}$. So ${\rho_2}|_{\Gamma_N}^*(\theta\alpha)$ is the restriction to
$\Gamma_N$ of a $(1,1)$-form defined on $\an{(\IP^n_{\IC} \times\IP^n_{\IC}\times \IP^m_{\IC})}$. 
Observe that $\sman{X}_N\subset \an{\overline{X_N}}$ and the difference has
dimension strictly less than $d=\dim X_N$.
This justifies
\begin{equation}
\label{eq:chain3}   
  \int_{\sman{X}_N} \rho_2|_{\Gamma_N}^*(\theta\alpha)^{\wedge d}
  =\int_{\an{\overline{X_N}}} \rho_2^*(\theta\alpha)^{\wedge d}
\end{equation}
where we take $\an{\overline{X_N}}$ as a complex analytic subset of
the analytification of $\IP_{\IC}^n\times\IP_{\IC}^n\times \IP_{\IC}^m$ and $\rho_2^*(\theta\alpha)$ as a
$(1,1)$-form on this ambient space.
Now $\theta$ takes values in $[0,1]$ and so
\begin{equation}
\label{eq:chain4}
\int_{\an{\overline{X_N}}} \rho_2^*(\theta\alpha)^{\wedge d}\le
  \int_{\an{\overline{X_N}}} (\rho_2^*\alpha)^{\wedge d}.
\end{equation}

The pull-back $\rho_2^*\alpha$ represents
$\rho_2^* \cO(1,1) \in
\mathrm{Pic}(\IP_{\IC}^n\times\IP_{\IC}^n\times\IP_{\IC}^m)$ in the
Picard group and
has compact support as $\an{(\IP_{\IC}^n\times\IP_{\IC}^n\times \IP_{\IC}^m)}$ is
compact.
But integration coincides with the intersection pairing in the compact
case; see \cite[Theorem~11.21]{voisin:hodge1}. In particular, we have
\begin{equation}
\label{eq:chain5} 
\int_{\an{\overline{X_N}}} (\rho_2^*\alpha)^{\wedge d} = (\rho_2^*
\cO(1,1)^{\cdot d} [\overline {X_N}]) 
\end{equation}
where the intersection takes place in
$\IP_{\IC}^n\times\IP_{\IC}^n\times\IP_{\IC}^m$. We recall \eqref{def:cF}
and apply  the projection formula to obtain 
\begin{equation}
\label{eq:chain6}  
  (\rho_2^*
\cO(1,1)^{\cdot d}  [\overline {X_N}]) = (\cO(0,1,1)^{\cdot d}  [\overline {X_N}]) =  (\cF^{\cdot d}).
\end{equation}

The (in)equalities  (\ref{eq:chainbegins}),
(\ref{eq:chain2}), (\ref{eq:chain3}),
(\ref{eq:chain4}), (\ref{eq:chain5}), and (\ref{eq:chain6}) yield
\begin{equation*}
  \int_{\sman{X}} [N]^*(\theta\alpha)^{\wedge d} \le (\cF^{\cdot d}). 
\end{equation*}
We recall the  lower bound (\ref{eq:N2dgrowth}) 
to  obtain
$(\cF^{\cdot d}) \ge (\kappa'/C^d) N^{2d}$ where $C$ comes from
\eqref{EqConditionOnCapitalC} and $\kappa' > 0$ comes from
\eqref{eq:N2dgrowth}.
The proposition follows with $\kappa = \kappa'/C^d$. 
\end{proof}

\subsection{Bounding an intersection number from above}

We keep the notation from the last subsection with $k=\IQbar$. 
So  $X$ is as above Proposition~\ref{PropAuxHtIneq} with $\dim X=d$
For each $N \in \IN$, let $\overline{X_N} \subseteq \IP_k^n \times
\IP_k^n \times \IP_k^m$ be the graph construction as in $\mathsection$\ref{ss:graph}. In
particular, $\dim \overline{X_N} = d$.  Here we need $\cF =
\cO(0,1,1)|_{\overline{X_N}}$ as defined in (\ref{def:cF}) and also
 $\cM = \cO(0,0,1)|_{\overline{X_N}}$ as defined in \eqref{def:cM}.

\begin{proposition}
  \label{prop:intersectionub}
Assume $d \ge 1$.  There exists a constant $c>0$ depending on the data introduced above
  with the following property. Say $N\ge 1$ is a power of $2$, then
  \begin{equation*}
    (\cM\cdot \cF^{\cdot(d-1)}) \le c N^{2(d-1)}.
  \end{equation*}
\end{proposition}

Let us make some preliminary remarks before the proof. A similar upper
bound for the intersection number was derived by the third-named
author in \cite{habegger:imrn,Hab:Special} using Philippon's
version~\cite{Philippon} of B\'ezout's Theorem for multiprojective
space. The approach here is similar but does not
 refer to Philippon's result. Rather, we rely on the following well-known
positivity property of the intersection theory of multiprojective
space: any effective Weil divisor on a multiprojective space is
nef. This approach was motivated by
K\"uhne's~\cite{kuehne:semiabelianbhc} work on semiabelian varieties.

\begin{proof}[Proof of Proposition \ref{prop:intersectionub}] Recall that $[2] \colon
A \rightarrow A$  is the  multiplication-by-$2$ morphism on $A$. For the symmetric and ample line bundle
$L$ on $A$, we have $[2]^*L \cong L^{\otimes 4}$.
Recall that $A$ is projectively normal in $\IP_{k(S)}^n$. 
By a result of Serre, \cite[Corollaire~2, Appendix~II]{Ast6970}, 
the morphism $[2]$ is represented 
 by homogeneous
polynomials $f_0,\ldots,f_n$ in the $n+1$ projective coordinates of $\IP^n$ of degree $4$, with  coefficients in $k(S)$ and with no common
zeros in $A$. 

Recall that the family $\cA$ is embedded in
$\IP_k^n\times S \subset\IP_k^n\times\IP_k^m$. We can spread out the
$f_0,\ldots,f_n$. More precisely, there exist a Zariski closed, proper
subset $Z\subsetneq \cA$ and polynomials $f_0,\ldots,f_n \in
k[\mathbf{X},\mathbf{S}]$ that are bihomogeneous of degree $(4,D')$ in
 the $(n+1)$-tuple of projective coordinates $\mathbf{X}$ of
$\IP_k^n$
and the $(m+1)$-tuple of projective coordinates $\mathbf{S}$ of
$\IP_k^m$,
 with the following properties:
\begin{enumerate}
 \item [(i)] the polynomials $f_0,\ldots,f_n$ have no common zeros on
 $(\cA\setminus Z)(k)$, and 
 \item[(ii)]  if   $(P,s)\in(\cA\setminus Z)(k)$, then $[2](P,s) = \bigl([f_0(P,s):\cdots : f_n(P,s)],s\bigr)$. 
\end{enumerate}
Moreover, as $f_0,\ldots,f_n$ have no common zero on the generic
fiber, we may assume that $\pi(Z)$ is Zariski closed and proper in
$S$. So we may assume that $Z = \pi^{-1}(\pi(Z))\subsetneq \cA$ and in particular,
$[2]$ maps $\cA\setminus Z$ to itself.

The  $4$ in the bidegree $(4,D')$ comes from $2^2=4$. The degree
$D'$ with respect to the base coordinates $\mathbf{S}$ is more mysterious. However,
by successively iterating we will get it under control. 


For each integer $l\ge 1$ we require polynomials
$f^{(l)}_0,\ldots,f^{(l)}_n$ to describe multiplication-by-$2^l$,
\textit{cf.} \cite[$\mathsection$9]{GaoHab}. In order to obtain information on the
degree with respect to $\mathbf{S}$ we construct them by iterating
the $f^{(1)}_0=f_0,\ldots,f^{(1)}_n=f_n$.
For all $i\in \{0,\ldots,n\}$  we set
\begin{equation*}
  f^{(l+1)}_i(\mathbf{X},\mathbf{S}) = f_i\left(\bigl(f^{(l)}_0(\mathbf{X},\mathbf{S}),\ldots,f^{(l)}_n(\mathbf{X},\mathbf{S})\bigr),\mathbf{S}\right)
\end{equation*}
 for all $i$; it is bihomogeneous in
$\mathbf{X}$ and $\mathbf{S}$. 
So for all $l\ge 1$ 
\begin{enumerate}
 \item [(i)] the polynomials $f^{(l)}_0,\ldots,f^{(l)}_n$ have no common zeros on
 $(\cA\setminus Z)(k)$, and 
 \item[(ii)]  if   $(P,s)\in(\cA\setminus Z)(k)$, then $[2^l](P,s) = \bigl([f^{(l)}_0(P,s):\cdots : f^{(l)}_n(P,s)],s\bigr)$. 
\end{enumerate}

If for all $i$ the polynomials $f^{(l)}_i$ are bihomogeneous of degree
$(D_{l},D'_{l})$, then all $f^{(l+1)}_i$ are bihomogeneous of degree
$(4D_{l},D'+4D'_{l})$. Recall that
$(D_1,D'_1)=(4,D')$, thus the recurrence
relations
$$ D_{l+1}=4D_l  \quad\text{and}\quad D'_{l+1} = D' + 4D'_{l} $$
imply
\begin{equation}
  \label{eq:DlprimeDlbound}
  D_l = 4^l \quad\text{and}\quad   D'_l = \frac{4^l-1}{3}D'\le 4^lD'
\end{equation}
for all $l\ge 1$.
Up-to the constant linear factor $D'$ the bidegrees both
grow like $4^l$.

We proceed as follows to cut out the graph $\overline{X_{N}}$ where
$N=2^l$. 
We start out with $\overline X\subset\IP_k^n\times\IP_k^m$. 
As $X$ dominates $S$ but $Z$ does not, there is an $i$ such that 
$f_i^{(l)}$  does not
vanish identically on $X$, without loss of generality we assume $i=0$.

Then as $i$ varies over $\{1,\ldots,n\}$ we obtain $n$ trihomogeneous
polynomials 
\begin{equation*}
g_i :=  Y_i f_0^{(l)}(\mathbf{X},\mathbf{S}) - Y_0   f_i^{(l)}(\mathbf{X},\mathbf{S})
\end{equation*}
where $Y_0,\ldots,Y_n$ are the projective coordinates on the middle factor of $\IP_k^n\times\IP_k^n\times\IP_k^m$. 
The tridegree of these polynomials is $(D_l,1,D'_l)$. 
Their zero locus on $\overline X\times\IP_k^n$ has the graph
$\overline{X_N}$ as an irreducible component;
by permuting coordinates we consider $\overline X\times\IP_k^n$ as a
subvariety of $\IP_k^n\times\IP_k^n\times\IP_k^m$.
We will see below that this is a proper component of the said
intersection. 
However, there may be further irreducible components in this
intersection,  some could even have dimension greater than $\dim
\overline{X_N}$.

This issue is clarified by the positivity result \cite[Corollary
12.2.(a)]{Fulton}. We apply it to the ambient
variety 
$\IP_k^n\times\IP_k^n\times\IP_k^m$, which becomes $X$ in Fulton's
notation; observe that
 the tangent bundle of a product of projective spaces
is generated by its global sections, \textit{cf.}~\cite[Examples 12.2.1.(a) and
(c)]{Fulton}. For $i\in\{1,\ldots,n\}$, the $V_i$ in Fulton's notation
is the zero set of $g_i$, and $V_{n+1}$ is 
$\overline{X} \times \IP_k^n$. So $r=n+1$ and $V_1,\ldots,V_{n+1}$ are
equidimensional. 
Observe that
\begin{alignat*}1
  \sum_{i=1}^{r} \dim V_i  - (r-1) \dim
  \IP_k^n\times\IP_k^n\times\IP_k^m
  &= (2n+m-1)(r-1) + \dim\overline{X}\times\IP_k^n
  - (r-1)(2n+m)
  \\
  &= \dim \overline{X} = \dim \overline{X_N},
\end{alignat*}
so, and as announced above, $\overline{X_N}$ is a proper component in the intersection of
$V_1,\ldots,V_n,$ and $\overline{X} \times \IP_k^n$.
By Fulton's \cite[Corollary
12.2.(a)]{Fulton} the  cycle class
attached to the intersection of $\overline X\times\IP_k^n$ with the zero
locus of $g_1,\ldots,g_n$
is represented by a positive cycle on
$\IP_k^n\times\IP_k^n\times\IP_k^m$, one of
whose components is $\overline{X_N}$. As $\cO(0,0,1)$ and $\cO(0,1,1)$
are numerically effective we conclude 
\begin{alignat}1
  \label{eq:inupperboundconclusion1}
  (\cO(0,0,1)  \cO(0,1,1)^{\cdot (d-1)}  [\overline{X_N}])
  &\le
  \left(\cO(0,0,1)  \cO(0,1,1)^{\cdot (d-1)}
  \cO(D_l,1,D'_l)^{\cdot n}  [\overline{X} \times\IP_k^n]\right).
\end{alignat}

The cycle $[\overline X\times\IP_k^n]$ is linearly equivalent to 
$\sum_{i+p=n+m-d} a_{ip}H_1^{\cdot i}  H_2^{\cdot p}$, with $H_{1}$ and $H_2$
hyperplane pullbacks of the factors $\IP_k^n\times \IP_k^m\supset
\overline X$, respectively, and with
$a_{ip}$ non-negative integers that depend only on $\overline X$.
Thus the left-hand side of (\ref{eq:inupperboundconclusion1}) is at
most 
\begin{alignat*}1
\sum_{i+p = n+m-d} a_{ip}
\left(  \cO(0,0,1)\cO(0,1,1)^{\cdot (d-1)}
\cO(D_l,1,D'_l)^{\cdot n}\cO(1,0,0)^{\cdot i}\cO(0,0,1)^{\cdot p}\right).
\end{alignat*}
We can expand the sum using linearly of intersection
numbers to find that it equals
\begin{alignat*}1
  \sum_{\substack{i+p=n+m-d\\ j'+p'=d-1\\i''+j''+p'' = n}} a_{ip} {d-1 \choose j',p'}
  {n\choose i'',j'',p''}{D_l}^{i''} {D'_l}^{p''} \left(\cO(1,0,0)^{\cdot
    (i+i'')} \cO(0,1,0)^{\cdot (j'+j'')}\cO(0,0,1)^{\cdot(1+p+p'+p'')}\right)
\end{alignat*}
Only terms with $i+i''\le n$ and $j'+j''\le n$ and $1+p+p'+p''\le m$
contribute to the sum.
On the other hand, any term in the sum satisfies
$i+i'' + j'+j'' + 1+p+p'+p'' =2n+m$. 
So we can assume $i+i''=n$ and $j'+j''=n$ and $1+p+p'+p''=m$ in the
sum
which thus simplifies to
\begin{alignat*}1
  \sum_{\substack{i+p=n+m-d, i+i''=n \\ j'+p'=d-1,j'+j''=n \\ i''+j''+p''=n,p+p'+p''=m-1}} a_{ip} {d-1\choose
    j',p'} {n\choose i'',j'',p''} {D_l}^{i''} {D'_l}^{p''}. 
\end{alignat*}
Note $i''+p''  = n-j'' =j'  = d-1-p'\le d-1$. 
We recall (\ref{eq:DlprimeDlbound}) and conclude that the 
 left-hand side of (\ref{eq:inupperboundconclusion1}) is at most
\begin{equation}
  \label{eq:intersectionnumberub2}
  (4^lD')^{d-1}
  \sum_{\substack{i+p=n+m-d\\ j'+p'=d-1\\i''+j''+p'' = n}}
 a_{ip}
 {d-1\choose
    j',p'} {n\choose i'',j'',p''} 
\le
 (4^l D')^{d-1}2^{d-1}3^n \sum_{\substack{i+p=n+m-d}}
 a_{ip}.
\end{equation}


We recall $N=2^l$ and use the projection formula with the estimates
above to find 
\begin{equation*}
  (\cO(0,0,1)|_{\overline{X_N}} 
  \cO(0,1,1)|_{\overline{X_N}}^{\cdot(d-1)}) \le  c N^{2(d-1)}
\end{equation*}
where $c>0$ depends only on $X$. 
Recall our definition $\cF = \cO(0,1,1)|_{\overline{X_N}}$ and $\cM =
\cO(0,0,1)|_{\overline{X_N}}$. So we get 
$  (\cM\cdot \cF^{\cdot(d-1)}) \le c N^{2(d-1)}$, as desired. 
%
%
\end{proof}

\subsection{Proof of Proposition~\ref{PropAuxHtIneq}}\label{SubsectionProofOfAUxHtIneq}
Now let us prove Proposition~\ref{PropAuxHtIneq} by comparing the intersection number inequalities in
Propositions~\ref{prop:intersectionlb} and \ref{prop:intersectionub}.

Let $X$ be of dimension $d$ as in Proposition~\ref{PropAuxHtIneq}. The
case $d = 0$ is trivial. So we assume $d \ge 1$. We may assume $N=2^l$
with $l \in \IN$. Let $\overline{X_N} \subseteq \IP_k^n \times \IP_k^n \times \IP_k^m$ be as in $\mathsection$\ref{ss:graph}. In particular $\dim \overline{X_N} = d$.

Let $\kappa > 0$ be as in Proposition~\ref{prop:intersectionlb}. Then
$(\cF^{\cdot d}) \ge \kappa N^{2d}$. Let $c > 0$ be as in
Proposition~\ref{prop:intersectionub}. Then $(\cM \cdot \cF^{\cdot
(d-1)}) \le c N^{2(d-1)}$.
We have indicated how to obtain  $\kappa$ and $c$  at the end of the proof of each one of the corresponding propositions.

Fix a rational number $c_1$ such that 
\begin{equation}
  \label{EqConstantInHtIneq}
  0<  c_1  cd < \kappa.
\end{equation}
Let $q$ be a multiple of the denominator of $c_1$.
Using the bounds above and linearity of intersection numbers we get
\begin{equation*}
d (\cM^{\otimes qc_1 N^2}\cdot(\cF^{\otimes q})^{\cdot(d-1)}) = d q^d
c_1 N^2 (\cM \cdot
\cF^{\cdot (d-1)}) \le d q^d c_1 N^2 c N^{2(d-1)} < \kappa q^d  N^{2d} \le
((\cF^{\otimes q})^{\cdot d}).
\end{equation*}
Then $\cF^{\otimes q} \otimes \cM^{\otimes -qc_1N^2}$ is a big line
bundle on $\overline{X_N}$ by a theorem of Siu \cite[Theorem
2.2.15]{PosAlgGeom}.  After possibly replacing $q$ by a multiple the line bundle
 $\cF^{\otimes q} \otimes \cM^{\otimes -qc_1N^2}$ admits a
non-zero global section. Say $h_{\overline{X_N},\cF}$ and
$h_{\overline{X_N},\cM}$ are representives of heights on
$\overline{X_N}$ attached by the Height Machine to
$\cF$ and $\cM$, respectively. After canceling $q$ we conclude that
$h_{\overline{X_N},\cF}- c_1 N^2 h_{\overline{X_N},\cM}$ is bounded
from below on a Zariski open and dense subset of $\overline{X_N}$. 
The  image  of this  subset under the projection $\rho_1$
contains  a Zariski open and dense subset $U_N$ of $X$.
It follows from the end of  $\mathsection$\ref{ss:graph}  that
there exists $c_2(N)$ such that
\begin{equation*}
h([N](P)) \ge 
c_1N^2 h(\pi(P)) - c_2(N)\quad \text{ for all }P \in U_N(\IQbar). \makeatletter\displaymath@qed
\end{equation*}



\section{Proof of the height inequality Theorem~\ref{ThmHtInequality}}\label{SectionHtIneq}

We keep the notation of $\mathsection$\ref{SectionSettingUpHtIneq}. In
particular, $S$ is a regular, irreducible, quasi-projective variety over $\IQbar$ and
$\pi\colon\cA\rightarrow S$ is an abelian scheme of relative dimension
$g \ge 1$. Moreover, we have immersions as in
$\mathsection$\ref{SectionSettingUpHtIneq} and we
assume {\tt (Hyp)} from page~\pageref{def:hyp}.
We use the heights introduced in $\mathsection$\ref{SubsectionHeight}. 
Let $X$ be an irreducible closed subvariety of $\cA$ 
defined over $\IQbar$. We assume that $X$ dominates $S$ and is
non-degenerate as defined in Definition~\ref{DefinitionNonDegenerate}

The upshot of {\tt (Hyp)} is that we obtain from Proposition~\ref{PropBettiForm} the Betti form $\omega$ on $\an{\cA}$. Moreover, part (i) and (iii) of Proposition~\ref{PropBettiForm} implies that, for $d = \dim X$, 
\begin{equation}\label{EqBettiFormHtIneq}
\omega|_{X^{\mathrm{sm},\mathrm{an}}}^{\wedge d} > 0 \quad \text{ at some smooth point of }X^{\mathrm{an}}.
\end{equation}

Our assumption \eqref{EqBettiFormHtIneq} allows us to apply
Proposition~\ref{PropAuxHtIneq} to
$X$. 
 There exists a constant $c_1>0$ as in \eqref{EqConstantInHtIneq} such
 that the following holds. Let $N\in\IN$ be a power of $2$, there
 exists a Zariski open dense subset $U_N\subset X$ and a constant
 $c_2(N)\ge 0$ such that
\begin{equation}\label{EquationHeight1}
h([N]P) \ge c_1 N^2 h(\pi(P)) - c_2(N)
\end{equation}
for all $P\in U_N(\IQbar)$; we stress that $U_N$ and 
 $c_2\ge 0$ may depend  on $N$ in addition to $X,\cA,$ and the various
 immersions such as $\cA \subseteq \IP_\IQbar^n \times \IP_\IQbar^m$.

By the  Theorem of Silverman-Tate, see \cite[Theorem~A]{Silverman} and
Theorem~\ref{thm:silvermantate}, there exist a  constant
$c_0\ge 0$ 
 such that
\begin{equation}
\label{EquationHeight2}
|\hat{h}_{\mathcal{A}}(P)-h(P)| \le
c_0\max\{1,h(\pi(P))\}\le c_0\bigl(1+h(\pi(P))\bigr)
\end{equation}
for all $P\in \mathcal{A}(\IQbar)$.

Next we kill Zimmer constants as in
Masser's~\cite[Appendix~C]{ZannierBook}.
For any $P\in U_N(\IQbar)$, we have
\begin{align*}
\hat{h}_{\cA}([N](P)) &\ge h([N](P))
-c_0\bigl(1+h(\pi([N](P)))\bigr) 
&&\text{(by (\ref{EquationHeight2}))}\\
&= h([N](P))
-c_0\bigl(1+h(\pi(P))\bigr)
&&\text{(as $\pi([N](P))=\pi(P)$)}
\\
&\ge c_1 N^2 h(\pi(P)) -c_2(N) -c_0\bigl(1+h(\pi(P))\bigr)
&&\text{(by (\ref{EquationHeight1})).}
\end{align*}
We use $\hat{h}_{\mathcal{A}}([N]P)=N^2 \hat{h}_{\mathcal{A}}(P)$, 
divide by $N^2$, and rearrange to get
\begin{equation*}
\hat{h}_{\cA}(P) \ge \left(c_1 - \frac{c_0}{N^2} \right) h(\pi(P)) - \frac{c_2(N)+ c_0}{N^2}
\end{equation*}
for all $N\in\IN$ that are powers of $2$  and all $P \in U_N(\IQbar)$. 

Recall that $c_0$ and $c_1$ are independent of $N$. We  fix $N\in\IN$ to be the least
power of $2$ such that $N^2\ge 2c_0/c_1$. As $h(\pi(P))$ is
non-negative we get  
\[
\hat{h}_{\cA}(P) \ge \frac{c_1}{2} h(\pi(P)) - \frac{c_2(N)+ c_0}{N^2}
\]
for all
$P \in U_N(\IQbar)$. Since $N$ is now fixed, the Zariski open dense
subset $U_N$ of $X$ is also fixed.
The theorem follows after adjusting $c_1$ and $c_2$. 
\qed
\begin{remark}
In the proof of Theorem~\ref{ThmHtInequality} we can keep track of the
process to compute the constant $c_1 > 0$. Use the notation in
$\mathsection$\ref{SectionSettingUpHtIneq}. In particular $\omega$ is
the Betti form on $\cA$, we have an immersion $\cA \subseteq \IP_\IC^n
\times \IP_\IC^m$, $\alpha$ is a $(1,1)$-form on $\IP_\IC^n \times
\IP_\IC^m$ representing the Chern class of $\cO(1,1)$, $\Delta \subseteq S^{\mathrm{an}}$ is open and relative compact, and $\theta \colon \cA^{\mathrm{an}} \rightarrow [0,1]$ (which factors through $S^{\mathrm{an}}$) is a smooth function with compact support contained in $\cA_{\Delta} := \pi^{-1}(\Delta)$. The function $\theta$ should furthermore satisfy $\theta(P_0) = 1$ for some $P_0 \in X^{\mathrm{sm}}(\IC)$ such that $(\omega|_{X^{\mathrm{sm}}})^{\wedge d}$ is positive at $P_0$, where $d = \dim X$.

Assume $d \ge 1$. The proof of Theorem~\ref{ThmHtInequality} tells us
that one half of any rational number satisfying the inequality
\eqref{EqConstantInHtIneq} can be taken as $c_1$. So the constant $c_1 > 0$ can be taken to be any rational
number
in  $(0,\kappa/(2c d))$, such that: 
\begin{itemize}
\item $\kappa = \kappa'/C^d$, where $\kappa' = \int_{\sman{X}}(
  \theta\omega )^{\wedge d}$, as in (\ref{eq:N2dgrowth}), and $C$ satisfies $C
  \alpha|_{\cA_{\Delta}} - \omega|_{\cA_{\Delta}} \ge 0$, as in (\ref{EqConditionOnCapitalC}),
\item $c$ is a  constant depending on a certain degree of  $X$ and coming from
  (\ref{eq:intersectionnumberub2}). 
\end{itemize}
\end{remark}


\section{Preparation for counting points}\label{SectionSettingUp}
\subsection{The universal family and non-degeneracy}
\label{subsec:univfamily}
In this section, we fix the basic setup to prove 
Proposition~\ref{PropAlgPtFar}, described as the alternative on page~\pageref{eq:alternative}, and
our main results.

Fix an integer $g \ge 2$.
Recall from $\mathsection$\ref{subsec:hgtineqnondeg} that
$\mathbb{M}_g$ 
denotes the fine moduli space of smooth curves of genus $g$,
with level-$\ell$-structure where $\ell\ge \bestlevel$ is fixed, \textit{cf.}
 \cite[Chapter~XVI, Theorem~2.11 (or above
Proposition~2.8)]{ACG:Curve},
 \cite[(5.14)]{DM:irreducibility}, or \cite[Theorem
 1.8]{OortSteenbrink}. It is known that 
$\mathbb{M}_g$ is a regular, quasi-projective variety of dimension
$3g-3$. We regard it over $\IQbar$; it is irreducible according
to our convention introduced below Theorem~\ref{ThmHtInequality}. 
There exists a universal curve $\mathfrak{C}_g$ over
$\mathbb{M}_g$,  it is smooth and proper over
$\mathbb{M}_g$ and its fibers are smooth curves of genus
$g$. Moreover, $\mathfrak{C}_g\rightarrow \mathbb{M}_g$ is
projective, \textit{cf.} \cite[Corollary to Theorem
1.2]{DM:irreducibility} or  \cite[Remark 2, \S 9.3]{NeronModels}.

Denote by $\mathrm{Jac}(\mathfrak{C}_g)$ the relative Jacobian of
$\mathfrak{C}_g \rightarrow \mathbb{M}_g$. It is an abelian scheme
coming with a natural principal polarization and equipped with level-$\ell$-structure, see
\cite[Proposition 6.9]{MFK:GIT94}.


Recall from $\mathsection$\ref{subsec:hgtineqnondeg} that  $\mathbb{A}_g$ denotes the fine moduli space of principally polarized
abelian varieties of dimension $g$, with level-$\ell$-structure. Moreover,
$\mathbb{A}_g$ is  regular and quasi-projective; see
\cite[Theorem~7.9 and below]{MFK:GIT94} or \cite[Theorem 1.9]{OortSteenbrink}.
We regard it as defined over $\IQbar$;
it is irreducible according to
our convention. Let $\pi \colon \mathfrak{A}_g
\rightarrow \mathbb{A}_g$ be the universal abelian variety; it is an
abelian scheme. Note that $\pi$ is projective; we refer to
Remark~\ref{rmk:projembedding} for this and other details.

As $\mathbb{A}_g$ is a fine moduli space we have the following Cartesian diagram
\begin{equation}
\label{EqUnivJac}
\begin{aligned}
\xymatrix{
\mathrm{Jac}(\mathfrak{C}_g) \ar[r] \ar[d]  \pullbackcorner & \mathfrak{A}_g \ar[d]^{\pi} \\
\mathbb{M}_g \ar[r]_{\tau} & \mathbb{A}_g
}
\end{aligned}
\end{equation}
the bottom arrow is the Torelli morphism. As we have level structure,
the Torelli morphism need not be injective on $\IC$-points, but it is
finite-to-one on such points, \textit{cf.} \cite[Lemma 1.11]{OortSteenbrink}. 



We also fix an ample line bundle $\cM$ on $\overline{\mathbb{A}_g}$,
where  $\overline{\mathbb{A}_g}$ is a, possibly non-regular,
projective  variety containing $\mathbb{A}_g$ as a Zariski open and
dense subset. The Height Machine
provides an equivalence class of height functions of which we fix a
representative $h_{\overline{\mathbb{A}_g},\cM} \colon
\overline{\mathbb{A}_g}(\IQbar)\rightarrow \IR$.



Next we fix a projective embedding of $\mathfrak{A}_g$ over
$\mathbb{A}_g$.
There is a relatively ample line bundle
$\cL$ on $\mathfrak{A}_g/\mathbb{A}_g$ with $[-1]^*\cL = \cL$; see
\cite[Th\'{e}or\`{e}me~XI~1.4]{LNM119}. 
After replacing $\cL$
by $\cL^{\otimes N}$, with $N\ge 4$ large enough, we can assume that
$\cL$ is very ample relative over $\mathbb{A}_g$ and $[-1]^*\cL=\cL$.
By \cite[Proposition~4.4.10(ii) and Proposition~4.1.4]{EGAII}, we
then have a closed immersion $\mathfrak{A}_g \rightarrow
\IP^n_{\IQbar} \times \mathbb{A}_g$ over $\mathbb{A}_g$
arising from $\cL \otimes
\pi^*(\cM|_{\mathbb{A}_g}^{\otimes p})$
for some integer $p\ge 1$, note that $\cM|_{\mathbb{A}_g}$ is ample.
For each $s \in \mathbb{A}_g(\IQbar)$, the fiber
$\mathfrak{A}_{g,s}=\pi^{-1}(s)$ is realized as a
projective subvariety of $\IP^n_\IQbar$ and the induced
closed immersion $\mathfrak{A}_{g,s}\rightarrow\IP^n_\IQbar$ comes
from the restriction $\cL|_{\mathfrak{A}_{g,s}}$ which is ample.
Flatness of $\mathfrak{A}_g\rightarrow\mathbb{A}_g$
implies that
$\dim H^0(\mathfrak{A}_{g,s},\cL_{\mathfrak{A}_{g.s}})$ is independent of $s$.
So $\mathfrak{A}_{g,s}$ is projectively
normal inside $\IP^n_\IQbar$, a property that will play a role later on. 

Recall that $\cL$ is symmetric and very ample on each fiber of
$\mathfrak{A}_{g}$. 
By Tate's Limit Argument we obtain the fiberwise N\'eron--Tate height,
\textit{cf.} (\ref{eq:fiberwiseNT}), 
\begin{equation}
 \label{eq:NTfinemoduli}
  \hat h \colon \mathfrak{A}_{g}(\IQbar)\rightarrow [0,\infty).
\end{equation}


Let $M\ge 1$ be an integer. We write $\mathfrak{A}_g^{[M]}$ for the
$M$-fold fibered power
$\mathfrak{A}_g\times_{\mathbb{A}_g} \cdots \times_{\mathbb{A}_g}\mathfrak{A}_g$
over $\mathbb{A}_g$.
Then $\mathfrak{A}_g^{[M]}\rightarrow \mathbb{A}_g$ is an abelian
scheme.

By taking the product we obtain closed
immersions $\mathfrak{A}_g^{[M]}\rightarrow (\IP_\IQbar^n)^{M}\times
\mathbb{A}_g$. The fiber of $\mathfrak{A}_g^{[M]}\rightarrow
\mathbb{A}_g$ above $s\in \mathbb{A}_g(\IQbar)$ is the $M$-fold power of
$\mathfrak{A}_{g,s}$. The associated fiberwise N\'eron--Tate height
$\hat h \colon \mathfrak{A}_g^{[M]}(\IQbar)\rightarrow [0,\infty)$ is
the sum of the N\'eron--Tate heights, as in (\ref{eq:NTfinemoduli}),
of the $M$ coordinates.

Let us now define the Faltings--Zhang morphism. 
In our setting the relative Picard scheme
$\mathrm{Pic}(\mathfrak{C}_g/\mathbb{M}_g)$ exists as a group scheme
over $\mathbb{M}_g$. It is a union over all $p\in\IZ$
of open and closed subschemes
$\mathrm{Pic}^p(\mathfrak{C}_g/\mathbb{M}_g)$, where $p$ indicates the
degree of a line bundle.
By definition we have
$\mathrm{Jac}(\mathfrak{C}_g) =
\mathrm{Pic}^0(\mathfrak{C}_g/\mathbb{M}_g)$.
We cannot expect to
have a section of $\mathfrak{C}_g\rightarrow\mathbb{M}_g$,
so we cannot expect to find an immersion of $\mathfrak{C}_g$
into $\mathrm{Jac}(\mathfrak{C}_g/\mathbb{M}_g)$. 
As constructed in the proof of
\cite[Proposition~6.9]{MFK:GIT94} we
do have a morphism $\mathfrak{C}_g \rightarrow
\mathrm{Pic}^1(\mathfrak{C}_g/\mathbb{M}_g)$ over $\mathbb{M}_g$.
Let $\mathfrak{C}_g^{[M]}$ and
$\mathrm{Pic}^p(\mathfrak{C}_g/\mathbb{M}_g)^{[M]}$ denote the
respective $M$-th fibered powers over $\mathbb{M}_g$. 
The difference morphism coming from the group scheme law
$\mathrm{Pic}(\mathfrak{C}_g/\mathbb{M}_g)\times_{\mathbb{M}_g}
\mathrm{Pic}(\mathfrak{C}_g/\mathbb{M}_g)\rightarrow
\mathrm{Pic}(\mathfrak{C}_g/\mathbb{M}_g)$ restricts to a morphism
$\mathrm{Pic}^1(\mathfrak{C}_g/\mathbb{M}_g)\times_{\mathbb{M}_g}
\mathrm{Pic}^1(\mathfrak{C}_g/\mathbb{M}_g)
\rightarrow
\mathrm{Jac}(\mathfrak{C}_g/\mathbb{M}_g)$ of schemes
over $\mathbb{M}_g$.
We take the appropriate  product morphism over $\mathbb{M}_g$ to get a
morphism 
\begin{equation}
  \label{eq:prefaltingszhang}
  \mathfrak{C}_g^{[M+1]} \rightarrow
 \mathrm{Jac}(\mathfrak{C}_g/\mathbb{M}_g)^{[M]}
\end{equation}
over $\mathbb{M}_g$. The choice of product is modeled after
(\ref{eq:faltingszhangintro}). More precisely, consider the situation
above a $k$-point of $\mathbb{M}_g$, where $k$ is an algebraically
closed field. The fiber of $\mathfrak{C}_g \rightarrow \mathbb{M}_g$ above this point is a smooth curve $C$
defined over $k$ of genus $g$. For $P_0,\ldots,P_M\in C(k)$ the
morphism (\ref{eq:prefaltingszhang}) maps $(P_0,P_1,\ldots,P_M)
\mapsto (P_1-P_0,P_2-P_0,\ldots,P_M-P_0)$ where the difference takes
place in the Jacobian of $C$. 

Recall (\ref{EqUnivJac}). We take the $M$-fold product and compose
with (\ref{eq:prefaltingszhang}) to obtain a commutative diagram of
morphisms of schemes 
\begin{equation}
  \label{DefFaltingsZhang}
  \begin{aligned}
    \xymatrix{
      \mathfrak{C}_g^{[M+1]}\ar[r] \ar[d] &
      \mathfrak{A}_g^{[M]}
      \ar[d] \\  
      \mathbb{M}_g \ar[r]_{\tau} & \mathbb{A}_g.
    }
  \end{aligned}
\end{equation}

If $S\rightarrow \mathbb{M}_g$ is a morphism of schemes then we define
$\mathfrak{C}_S = \mathfrak{C}_g \times_{\mathbb{M}_g} S$ and
$\mathfrak{C}_S^{[M]} = \mathfrak{C}_g^{[M]}\times_{\mathbb{M}_g} S$.
If $S$ is irreducible, then so is $\mathfrak{C}_S^{[M]}$ by induction
on $M$ and a topological argument using that
$\mathfrak{C}_S\rightarrow S$ is smooth and hence open.
Taking the fibered product with $S$  and composing with (\ref{DefFaltingsZhang})
yields a commutative diagram of morphisms of schemes
\begin{equation*}
  \begin{aligned}
    \xymatrix{
      \mathfrak{C}_S^{[M+1]}\ar[r] \ar[d] &
      \mathfrak{A}_g^{[M]}
      \ar[d] \\  
      S\ar[r]_{\tau\circ(S\rightarrow\mathbb{M}_g)} & \mathbb{A}_g.
    }
  \end{aligned}
\end{equation*}
By the
universal property
of the fibered product we get a morphism of schemes 
\begin{equation}
  \label{DefFaltingsZhang3}
  \begin{aligned}
    \xymatrix{
      \mathfrak{C}_S^{[M+1]}\ar[r]^-{\mathscr{D}_{M}}
\ar[dr]  &
      \mathfrak{A}_g^{[M]}      \times_{\mathbb{A}_g} S
      \ar[d] \\  
      & S
    }
  \end{aligned}
\end{equation}
over $S$. We call $\mathscr{D}_{M}$ the \textit{Faltings--Zhang morphism} (over
$S$).
Then $\mathscr{D}_{M}$  is proper since the diagonal arrow in
(\ref{DefFaltingsZhang3}) is proper.


Let for the moment $S\rightarrow \mathbb{M}_g$ be the identity.
If $s \in \mathbb{M}_g(\IQbar)$, then $\mathfrak{C}_s$ is the curve
parametrized by $s$, and $\mathfrak{A}_{g,\tau(s)}$ is its Jacobian.
To embed $\mathfrak{C}_s$ into $\mathfrak{A}_{g,\tau(s)}$ we must work with a
base point $P\in
\mathfrak{C}_s(\IQbar)$. Then 
$\mathfrak{C}_s-P = \mathscr{D}_1(\{P\} \times \mathfrak{C}_s)$
is an irreducible curve inside $\mathfrak{A}_g$ lying above $\tau(s)$.
Hence it provides a closed immersion   $\mathfrak{C}_s-P \subset
\IP_\IQbar^n$.


Let $\deg X$ denote the degree of an irreducible closed subvariety $X$ of
$\IP^n_\IQbar$ and let $h(X)$ denote its height, \textit{cf.} \cite{BGS}.

\begin{lemma}
  \label{lem:uniformdegreebound}
  There exists a constant $c$ such that the following two properties
  hold for all $s\in \mathbb{M}_g(\IQbar)$.
  \begin{enumerate}
  \item [(i)] We have
    $\deg(\mathfrak{C}_s-P) \le c$
    for all $P\in
    \mathfrak{A}_{g,\tau(s)}(\IQbar)$.
  \item[(ii)]
    There exists $P_s\in \mathfrak{C}_s(\IQbar)$ such that
    $h(\mathfrak{C}_s-P_s) \le  c \max\{1,
    h_{\overline{\mathbb{A}_g},\cM}(\tau(s))\}$.    
  \end{enumerate}
\end{lemma}
\begin{proof}
  We need a quasi-section of
  $\mathfrak{C}_g\rightarrow\mathbb{M}_g$ as provided by \cite[Corollaire~17.16.3(ii)]{EGAIV}. 
  So there is an affine scheme $S$
  and a morphism $S\rightarrow \mathfrak{C}_g$ that factors through a
  surjective, quasi-finite, \'etale morphism $ S\rightarrow
  \mathbb{M}_g$. We consider the product
  $\mathfrak{C}_g\times_{\mathbb{M}_g} S\rightarrow
  \mathfrak{C}_g^{[2]}$ composed with the Faltings--Zhang morphism
  $\mathscr{D}_1\colon
  \mathfrak{C}_g^{[2]}\rightarrow\mathfrak{A}_g\times_{\IA_g}
  \mathbb{M}_g$ over $\mathbb{M}_g$ and then the projection of
  $\mathfrak{A}_g$. This is a morphism of schemes
  $\mathfrak{C}_g\times_{\mathbb{M}_g} S\rightarrow
  \mathfrak{A}_g$. Its 
  image is a constructible subset of $\mathfrak{A}_g$. 
  So it is a union of finitely many 
  irreducible Zariski locally closed subsets $\{X_i\}_i$ of
  $\mathfrak{A}_g$. We have
  the following property. 

  Given a point $s\in \mathbb{M}_g(\IQbar)$, there is an $i$ such that
  the fiber of $\pi|_{X_i}\colon X_i \rightarrow\mathbb{A}_g$ above
  $\tau(s)$ is a finite union of irreducible curves, up-to finitely
  many points one of these curves is
  $\mathfrak{C}_s-P_s$ with $P_s\in \mathfrak{C}_s(\IQbar)$.

  We  have a closed immersion $\mathfrak{A}_g\rightarrow \IP_\IQbar^n
  \times \mathbb{A}_g$.
  Moreover, a sufficiently  large positive power of $\cM$ induces a
  closed immersion of $\overline{\mathbb{A}_g}\rightarrow\IP_\IQbar^m$
   for some $m$.
  Thus, we consider $\mathbb{A}_g$ as a 
  Zariski locally closed subset
  $\IP_\IQbar^m$. We identify each 
  $X_i$ with its image in $\IP^n_\IQbar\times
  \IP^m_\IQbar$, an irreducible  Zariski locally closed set.  Then
  $\mathfrak{C}_s-P_s \subset\IP^n_\IQbar$ arises as an irreducible
  component of the intersection of some Zariski
  closure $\overline{X_i}$
   with $\IP^n_\IQbar\times
  \{\tau(s)\}$.

  We use the Segre embedding
  $\IP_\IQbar^n\times\IP_\IQbar^m\rightarrow \IP_\IQbar^{(n+1)(m+1)-1}$
  to embed our situation into projective space. By B\'ezout's Theorem
  \cite[Example~8.4.6]{Fulton},
  $\deg(\mathfrak{C}_s-P_s)$ is bounded from above uniformly in $s$.
  Translating a curve inside $\mathfrak{A}_{g,\tau(s)}$ by a point of
  $\mathfrak{A}_{g,\tau(s)}(\IQbar)$ does not
  change its degree. So if $P\in \mathfrak{A}_{g,\tau(s)}$, then
  $\deg(\mathfrak{C}_s-P)=\deg(\mathfrak{C}_s-P_s)$.  
  This yields (i).

  Part (ii) follows as (i) but this time we use
  the Arithmetic B\'ezout Theorem, still executing the intersection after
  applying the Segre embedding. Indeed, recall that
  $\mathfrak{C}_s-P_s$ as an irreducible
  of the intersection of some $\overline{X_i}$ with
  $\IP_\IQbar^n\times\{\tau(s)\}$. 
  The height and degree of $\overline{X_i}$ are
  bounded from above independently of $s$; the same holds for the
  degree of $\IP_\IQbar^n \times\{\tau(s)\}$. The height of
  $\IP_\IQbar^n \times\{\tau(s)\}$ is bounded from above linearly in terms of
  $h(\tau(s))$. Finally, we can apply \cite[Th\'eor\`eme
  3]{HauteursAlt3}.
  Finally, note that  by the Height
  Machine the absolute logarithmic Weil height $h( \tau(s))$,
  where $\tau(s)$ is understood as an element of
  $\IP^m_{\IQbar}(\IQbar)$, is bounded from
  above linearly in terms of $h_{\overline{\mathbb{A}_g},\cM}(\tau(s))$. 
\end{proof}

\subsection{Non-degeneracy of $\mathscr{D}_M(\mathfrak{C}_S^{[M+1]})$
  for large $M$}  
  In this subsection all varieties  are defined over the field $\IC$. 
We keep the notation of the previous subsection and let
$S$ be an irreducible variety with a quasi-finite 
   morphism
$S\rightarrow\mathbb{M}_g$. 
Note that
$\mathscr{D}_M(\mathfrak{C}_S^{[M+1]})$ is Zariski closed
in $\mathfrak{A}_g^{[M]}\times_{\mathbb{A}_g} S$ 
because $\mathscr{D}_M$ is proper. We endow this image with the
reduced induced scheme structure.

The following non-degeneracy theorem proved by the second-named author
is crucial to prove our main result.
It confirms that
Theorem~\ref{ThmHtInequality} can be applied to
$\mathscr{D}_M(\mathfrak{C}_S^{[M+1]})$ for $M \ge 3g-2$.
We obtain a height inequality on a Zariski open dense subset. 

\begin{theorem}[\!\!{\cite[Theorem~1.2']{GaoBettiRank}}]\label{ThmNonDegForBd}
Let $S$ be an irreducible variety with a (not necessarily dominant) quasi-finite morphism $S \rightarrow \mathbb{M}_g$. Assume $g \ge 2$ and $M
\ge 3g-2$. Then $\mathscr{D}_M( \mathfrak{C}_S^{[M+1]} )$, which is a closed 
irreducible subvariety of $\mathfrak{A}_g^{[M]}\times_{\mathbb{A}_g} S$, is 
non-degenerate in the sense of
Definition~\ref{DefinitionNonDegenerate}.
\end{theorem}

The fibered product in the theorem involves
$S\rightarrow\mathbb{M}_g\xrightarrow{\tau} \mathbb{A}_g$.

More precisely, the meaning of the conclusion of the theorem is as follows. For the abelian scheme $\pi \colon \cA = \mathfrak{A}_g^{[M]} \times_{\mathbb{A}_g}S \rightarrow S$ and for the irreducible subvariety $X:=\mathscr{D}_M( \mathfrak{C}_S^{[M+1]} )$ of $\cA$, there exists a open non-empty subset $\Delta$ of $\an{S}$, with the Betti map $b_{\Delta} \colon \cA_{\Delta} = \pi^{-1}(\Delta) \rightarrow \mathbb{T}^{2g}$, such that 
\begin{equation}\label{EqNonDegXDMC}
\mathrm{rank}_{\IR} (\mathrm{d}b_{\Delta}|_{\sman{X}})_x = 2\dim X\quad \text{for some }x \in \sman{X} \cap \cA_{\Delta},
\end{equation}
when $g \ge 2$ and $M \ge 3g-2$.

\begin{proof}
This theorem, essentially \cite[Theorem~1.2']{GaoBettiRank}, is a consequence of Theorem~1.3 of \textit{loc.cit.} Because of its importance to the current paper, we hereby give more details of the deduction.

We start by showing the result for the case where $\mathfrak{C}_S
\rightarrow S$ admits a section $\epsilon$. 
In this case $\epsilon$ induces an Abel--Jacobi embedding $j_{\epsilon} \colon \mathfrak{C}_S \rightarrow \mathrm{Jac}(\mathfrak{C}_S/S)$, which is a closed immersion of $S$-schemes. The modular map is the Cartesian diagram
\[
\xymatrix{
\mathrm{Jac}(\mathfrak{C}_S/S) \ar[r]^-{\iota} \ar[d] \pullbackcorner & \mathfrak{A}_g \ar[d] \\
S \ar[r] & \mathbb{A}_g
}
\]
with the bottom morphism being the composite of the given $S \rightarrow \mathbb{M}_g$ with the Torelli map $\tau \colon \mathbb{M}_g \rightarrow \mathbb{A}_g$. The Torelli map $\tau$ is quasi-finite; see \cite[Lemma 1.11]{OortSteenbrink}. Thus the bottom morphism is quasi-finite. Hence $\iota$ is quasi-finite.

We wish to apply \cite[Theorem~1.3]{GaoBettiRank} to the subvariety $j_{\epsilon}(\mathfrak{C}_S)$ of the abelian scheme $\mathrm{Jac}(\mathfrak{C}_S/S) \rightarrow S$. We need to verify the hypotheses. First of all $\iota|_{j_{\epsilon}(\mathfrak{C}_S)}$ is generically finite because $\iota$ is quasi-finite. Hypothesis (a) is satisfied since $\dim j_{\epsilon}(\mathfrak{C}_S) = \dim S + 1 > \dim S$. For hypotheses (b) and (c), note that for any $s\in S(\IC)$, the fiber $j_{\epsilon}(\mathfrak{C}_S)_s$ is the Abel--Jacobi embedding of $\mathfrak{C}_s$ in its Jacobian via the point $\epsilon(s)$. Thus hypothesis (b) is satisfied because each curve generates its Jacobian, and hypothesis (c) holds true since $g \ge 2$.

Thus we can apply \cite[Theorem~1.3.(ii)]{GaoBettiRank} and obtain that $\mathscr{D}_M(\mathfrak{C}_S^{[M+1]})$ is non-degenerate\footnote{Observe that $\mathscr{D}_M(\mathfrak{C}_S^{[M+1]}) = \mathscr{D}_M^{\cA}(j_{\epsilon}(\mathfrak{C}_S^{[M+1]}))$, with $\mathscr{D}_M^{\cA}$ be as in \cite[Theorem~1.3.(ii)]{GaoBettiRank} with $\cA = \mathfrak{A}_g\times_{\mathbb{A}_g}S \cong \mathrm{Jac}(\mathfrak{C}_S/S)$. See below \eqref{eq:prefaltingszhang}.} if $M \ge j_{\epsilon}(\mathfrak{C}_S) = \dim S + 1$. But $\dim S \le \dim \mathbb{M}_g = 3g-3$. Hence  $\mathscr{D}_M(\mathfrak{C}_S^{[M+1]})$ is non-degenerate if $M \ge 3g-2$.


For an arbitrary $S$, the generic fiber of $\mathfrak{C}_S \rightarrow
S$ has a rational point over some finite extension of $K(S)$, the
function field of $S$. Thus there exists a quasi-finite \'{e}tale
dominant (not necessarily surjective) morphism $\rho \colon S'
\rightarrow S$, with $S'$ irreducible, such that $\mathfrak{C}_{S'} = \mathfrak{C}_S \times_S
S' \rightarrow S'$ admits a section. Thus $X' :=
\mathscr{D}_M(\mathfrak{C}_{S'}^{[M+1]})$, as a subvariety of
$\cA':=\mathfrak{A}_g^{[M]}\times_{\mathbb{A}_g} S'$,  is
non-degenerate by the previous case. So there exists a connected, open non-empty subset $\Delta'$ of $\an{S'}$, with the Betti map $b_{\Delta'} \colon \cA'_{\Delta'} \rightarrow \mathbb{T}^{2g}$, such that for some $x' \in \sman{X'} \cap \cA'_{\Delta'}$ we have
\[
\mathrm{rank}_{\IR} (\mathrm{d}b_{\Delta'}|_{\sman{X'}})_{x'} = 2\dim X'.
\]
We may furthermore shrink $\Delta'$ so that $\rho|_{\Delta'}$ is a diffeomorphism. In particular $\Delta:=\rho(\Delta')$ is open in $\an{S}$.

Denote by $\rho'_{\cA} \colon \cA' = \cA\times_S S' \rightarrow \cA$
the projection to the first factor. Then
$\rho'_{\cA}|_{\cA'_{\Delta'}}$ is a diffeomorphism. Both $\cA
\rightarrow S$ and $\cA' \rightarrow S'$ carry
level-$\ell$-structures. By construction and uniqueness properties of
the Betti map, we may assume that $b_{\Delta} \colon \cA_{\Delta} \rightarrow \mathbb{T}^{2g}$ equals $b_{\Delta'} \circ  (\rho'_{\cA}|_{\cA'_{\Delta'}})^{-1}$. Thus
\[
\mathrm{rank}_{\IR} (\mathrm{d}b_{\Delta}|_{\sman{X}})_{x} = 2\dim X'
\]
with $x = \rho_{\cA}(x')$. So \eqref{EqNonDegXDMC} holds true because $\dim X' = \dim X$. Hence we are done.
\end{proof}

\subsection{Technical lemmas}

The following lemma will be useful in the proofs of the desired
bounds. Let for the moment $k$ be an algebraically closed field and
$M\ge 1,n\ge 1$ integers. If $Z$ is a Zariski closed subset of $(\IP_k^n)^M$
we let $\deg Z$ denote the sum of the degrees of all irreducible
components of $Z$ with respect to $\cO(1,\ldots,1)$. 

\begin{lemma}\label{LemmaNAlon} 
  Let $C \subset \IP_k^n$ be an
  irreducible curve defined over $k$ and let $Z \subseteq (\IP_k^n)^M$
  be a Zariski closed subset of $(\IP_k^n)^M$ such that $C^M = C \times
  \cdots \times C \not\subseteq Z$. Then there exists a number $B$,
  depending only on $M, \deg C,$ and $\deg Z$, satisfying the following
  property. If $\Sigma \subset C(k)$ has cardinality $\ge B$, then
  $\Sigma^M = \Sigma \times \cdots \times \Sigma \not\subseteq Z(k)$.
\end{lemma}
\begin{proof}
  Let us prove this lemma by induction on $M$. The case $M = 1$ follows
  easily from B\'{e}zout's Theorem.

  Assume the lemma is proved for $1, \ldots, M-1$. Let $q \colon
  (\IP_k^n)^M \rightarrow \IP_k^n$ be the projection to the first
  factor.

  The number of irreducible components of $Z \cap C^M$ and their degrees
  are bounded from above in terms of $M,\deg C,$ and $\deg Z$ by
  B\'ezout's Theorem applied to the Segre embedding.
  Let $Z'$ be the union of all  irreducible
  components $Y$ of $Z \cap C^M$ with $\dim q(Y)\ge 1$, let $Z''$ be the
  union of all other irreducible components.

  Note that $q(Z')\subset C$. For all $P \in C(k)$ the fiber
  $q|_{Z'}^{-1}(P)=Z'\cap (\{P\}\times (\IP_k^n)^{M-1})$ has
  dimension at most
  $\dim Z'-1\le M-2$. So the projection of $q|_{Z'}^{-1}(P)$ to the final
  factors $(\IP_k^n)^{M-1}$ does not contain $C^{M-1}$. By B\'ezout's
  Theorem the degree of this projection is bounded in terms of $M,\deg
  C,$ and $\deg Z$. We apply the induction hypothesis to the projection
  of $q|_Y^{-1}(P)$ to $(\IP_k^n)^{M-1}$ and obtain a number $B'$,
  depending only on $M,\deg C,$ and $\deg Z$ satisfying the following
  property. If $\Sigma \subseteq C(k)$ has cardinality $\ge B'$, then
  $\{P\} \times \Sigma^{M-1} \not\subseteq Z'(k)$ for all $P \in C(k)$.

  Now $\dim q(Z'') = 0$, so $q(Z'')$ is a finite set of cardinality
  at most $B''$, the number of irreducible components of
  $Z\cap C^M$. 

  The lemma follows with $B=\max\{B',B''+1\}$.  
\end{proof}

In the next lemma we use the Faltings--Zhang morphism in a single
abelian variety $A$, \textit{i.e.}, $\mathscr{D}_M \colon
A^{M+1}\rightarrow A^M$ defined by $(P_0,\ldots,P_M)\mapsto
(P_1-P_0,\ldots,P_M-P_0)$. 

\begin{lemma}\label{LemmaPacketsAlternative3}
  Let $A$ be an abelian variety defined over $\overline{\IQ}$ and
  suppose $C$ is a smooth curve of genus $g\ge
  2$ contained in $A$. If $Z$ is an irreducible Zariski
  closed and proper subset of $\mathscr{D}_M(C^{M+1})$, 
  then
  \begin{equation*}
    \#\{ P \in C(\overline{\IQ}) : (C-P)^M \subset Z
    \} \le 84(g-1).
  \end{equation*}
\end{lemma}
\begin{proof}
For simplicity denote by $\Xi = \{ P \in C(\overline{\IQ}) : (C-P)^M
\subset Z \}$. 
Fix $P_0\in \Xi$. It suffices to prove that there are only $84(g-1)$
possibilities for $P_1 - P_0$ when $P_1$ runs over $\Xi$.

Say $P_1 \in \Xi$ and let $i\in \{0,1\}$. 
Note that
$Z\subsetneq \mathscr{D}_M(C^{M+1})$ and so
 $\dim Z < \dim \mathscr{D}_M(C^{M+1})\le M+1$, 
as $\mathscr{D}_M(C^{M+1})$ irreducible. 
Note that $(C-P_i)^M\subset Z$, and so both have dimension $M$.
As $Z$ is irreducible we find 
 $(C - P_i)^M=Z$ for $i=0$ and $i=1$.

Applying the first projection  $A^M\rightarrow A$ yields $C-P_1=C-P_0$. In other words, $P_1-P_0$
stabilizes $C$. By Hurwitz's Theorem \cite{Hurwitz1892}, a smooth curve over $\IQbar$ of
genus $g\ge 2$ has at most $84(g-1)$ automorphisms.
Hence we are done.
\end{proof}


\section{N\'{e}ron--Tate distance between points on curves}\label{SectionDistanceCurve}

The goal of this section is to prove Proposition~\ref{PropAlgPtFar}, below.
Namely we will show that $\IQbar$-points on smooth curves are
rather ``sparse'', in the sense that the N\'{e}ron--Tate distance
between two $\IQbar$-points on a smooth curve $C$ is in
general large compared with the Weil height of $C$. 

We use the notation from $\mathsection$\ref{subsec:univfamily}. Recall
that we have fixed a projective compactification
$\overline{\mathbb{A}_g}$ of $\mathbb{A}_g$ over $\IQbar$, an ample
line bundle $\cM$, and a height function
$h_{\overline{\mathbb{A}_g},\cM} \colon
\mathbb{A}_g(\IQbar)\rightarrow \IR$ attached to this pair.
We also fixed a closed immersion $\mathfrak{A}_g$ into
$\IP^n_\IQbar\times \mathbb{A}_g$ over $\mathbb{A}_g$ and
let $\tau\colon \mathbb{M}_g\rightarrow\mathbb{A}_g$ denote the
Torelli morphism. 
If $s\in \mathbb{M}_g(\IQbar)$ then $\mathfrak{C}_s$
is a smooth curve of genus $g$ defined over $\IQbar$.
Moreover, if $P,Q\in \mathfrak{C}_s(\IQbar)$, then $P-Q$ is a
well-defined element of $\mathfrak{A}_g(\IQbar)$ and so is its
N\'eron--Tate height  $\hat h(P-Q)$, see (\ref{eq:NTfinemoduli}).

\begin{proposition}\label{PropAlgPtFar}
  Let $S$ be an irreducible closed subvariety of $\mathbb{M}_g$ defined over $\IQbar$. There
  exist positive constants $c_1, c_2, c_3, c_4$ depending on the choices
  made above and on $S$ with the following property. Let $s\in
  S(\IQbar)$ with $h_{\overline{\mathbb{A}_g},\cM}(\tau(s))
  \ge c_1$. There exists a subset $\Xi_s \subset \mathfrak{C}_s(\IQbar)$
  with $\#\Xi_s\le c_2$ such that any $P \in \mathfrak{C}_s(\IQbar)$
  satisfies the following alternative.
  \begin{enumerate}
  \item[(i)] Either $P \in \Xi_s$;
  \item[(ii)] or $\#\{Q \in \mathfrak{C}_s(\IQbar) : \hat{h}(Q-P) \le h_{\overline{\mathbb{A}_g},\cM}(\tau(s))/c_3\} < c_4$.
  \end{enumerate}
\end{proposition}



\begin{proof}
  We fix an immersion of $\mathbb{M}_g$ into
  some projective space and let $\overline{\mathbb{M}_g}$ denote its
  Zariski closure. 
By a standard triangle inequality estimate, there exist constants
$c''>0$ and $c'''\ge 0$ such that
\begin{equation}\label{EqTriangleInequality}
h(s) \ge  c'' h_{\overline{\mathbb{A}_g},\cM}(\tau(s)) - c'''
\end{equation}
for all $s \in \mathbb{M}_g(\IQbar)$; see
\cite[Lemma~4]{Silverman:heightest11}, the left-hand side is the Weil
height and represents a height function coming from an ample line
bundle on $\overline{\mathbb{M}_g}$.
If $h_{\overline{\mathbb{A}_g},\cM}(\tau(s))\ge c_1$ with $c_1\ge
2 c'''/c''$
then $h(s) \ge c''
h_{\overline{\mathbb{A}_g},\cM}(\tau(s))/2$. We find that it suffices to
prove the alternative with $h_{\overline{\mathbb{A}_g},\cM}(\tau(s))$
replaced by $h(s)$ in (ii) and adjusting $c_3$. 
Our proof is by induction on $\dim S$.

If $\dim S = 0$, then the proposition follows by enlarging $c_1$.

If $\dim S\ge 1$, we fix $M = 3g-2$.
Applying
Theorem~\ref{ThmNonDegForBd} to the immersion
 $S^{\mathrm{sm}} \hookrightarrow \mathbb{M}_g$,
we conclude that the closed irreducible subvariety
 $X:=\mathscr{D}_M(\mathfrak{C}_{S^{\mathrm{sm}}}^{[M+1]})$
of the abelian scheme
$\cA=\mathfrak{A}_g^{[M]}\times_{\mathbb{A}_g} S^{\mathrm{sm}} \rightarrow
S^{\mathrm{sm}}$
is non-degenerate. Hence we can apply Theorem~\ref{ThmHtInequality} to
$\cA$ 
and 
$X$
(and the compactification $\overline{S}$ is the Zariski closure of $S$ in $\overline{\mathbb{M}_g}$). 
So, combined with \eqref{EqTriangleInequality}, there exist constants
$c > 0$ and $c'$ as well as a Zariski open dense subset $U$ of
$X$,
satisfying the following property.
For all $s \in S(\IQbar)$ and
all $P, Q_1,\ldots,Q_M \in \mathfrak{C}_s(\IQbar)$, we have
\begin{equation}\label{EqHtInequalityBdRatPt} 
c h(s)\le 
\hat{h}(Q_1-P) + \cdots + \hat{h}(Q_M - P) + c' \quad
\text{if}\quad (Q_1-P,\ldots, Q_M-P) \in U(\IQbar).
\end{equation}

Observe that $\pi(X)=S^{\mathrm{sm}}$, where
$\pi\colon \mathcal{A} \rightarrow S^{\mathrm{sm}}$ 
is the structure
morphism.
Therefore, $S \setminus \pi(U)$ is not Zariski dense in
$S$. 
Let $S_1,\ldots, S_r$ be the irreducible components of the
Zariski closure of $S\setminus \pi(U)$ in $S$. Then $\dim S_j \le \dim
S - 1$ for all $j$.

By the induction hypothesis, this proposition holds for all $S_j$. Thus it
remains to prove the conclusion of this proposition for curves above
\begin{equation}
\label{eq:soutsideSi}
s \in S(\IQbar) \setminus \bigcup_{j=1}^r S_j(\IQbar) \subset \pi(U(\IQbar)).
\end{equation}
 First we construct $\Xi_s$ and then
we will show that we are in one of the two alternatives. 

It is convenient to fix a base point $P_s\in
\mathfrak{C}_s(\IQbar)$ and consider $(\mathfrak{C}_s - P_s)^M$
as a subvariety of $\cA_s = \pi^{-1}(s)$. 

Let us set $W = X\setminus U$, it is a Zariski closed and proper
subset of $X$.  By (\ref{eq:soutsideSi}) we find  $W_s\subsetneq X_s = \mathscr{D}_M(\mathfrak{C}_s^{[M+1]})$. 

Let $Z$ be an irreducible component of
$W_s$. 
Consider the set
\[
\Xi_Z := \{ P \in \mathfrak{C}_s(\IQbar) : ( \mathfrak{C}_s- P)^M
\subseteq Z \}. 
\]
Apply Lemma~\ref{LemmaPacketsAlternative3} to $A =
(\mathfrak{A}_g)_{\tau(s)}, C=\mathfrak{C}_s-P_s\subset A$, and $Z$. As $Z\subsetneq
\mathscr{D}_M(\mathfrak{C}_s^{[M+1]})$  we have $\#\Xi_Z \le 84(g-1)$.

Let $\Xi_s = \bigcup_Z \Xi_Z$ where $Z$ runs over all irreducible
components of $W_s$. The number of irreducible components  is
bounded from above in an algebraic family.
So the number of irreducible components of $W_s$
is bounded from above by a number that  is independent of $s$; but it may depend on $W$.  We take
$c_2$ to be such a number multiplied with $84(g-1)$. Thus
$\#\Xi_s \le c_2$ if (\ref{eq:soutsideSi}) and with $c_2$ independent
of $s$ and $P$.

Say $P\in \mathfrak{C}_s(\IQbar)$ and $P\not\in \Xi_s$.
So we are not in case (i) of the proposition.
Then $(\mathfrak{C}_s-P)^M
\not\subseteq W_s$. We want to apply Lemma~\ref{LemmaNAlon} to
$\mathfrak{C}_s-P$ and $W_s$.

Recall that the abelian scheme $\mathfrak{A}_{g}$ is embedded in
$\IP_{\IQbar}^n\times \mathbb{A}_g$  over
$\mathbb{A}_g$, \textit{cf.}
$\mathsection$\ref{subsec:univfamily}.
So $\cA$ is embedded in $(\IP_{\IQbar}^n)^M\times S^{\mathrm{sm}}$ over $S^{\mathrm{sm}}$. 
We may identify $\mathfrak{C}_s-P$ with
 a smooth curve in $\IP_{\IQbar}^n$. 
The degree of $\mathfrak{C}_s-P$ as a subvariety of $\IP_\IQbar^n$ 
is bounded independently of $s$ by
Lemma~\ref{lem:uniformdegreebound}(i); applying the Torelli
morphism $\tau$ does not affect the degree. 
Moreover, $W_s$ is Zariski closed in
$X_s\subset \cA_{s}$. Still holding $s$ fixed we may take $W_s$ as a Zariski
closed subset of $(\IP_{\IQbar}^n)^M$. Being the fiber above $s$ of a
subvariety of $(\IP_{\IQbar}^n)^M\times S^{\mathrm{sm}}$, we find that the degree
of $W_s$ is bounded from above independently of $s$. 
From
Lemma~\ref{LemmaNAlon} we
thus obtain a number $c_4$, depending only  on these bounds and with the
following property. Any subset $\Sigma \subset
\mathfrak{C}_s(\IQbar)$ with cardinality $\ge c_4$ satisfies $(\Sigma-P)^M
\not\subset W_s$. 
It is crucial that $c_4$ is independent of $s$.

We work with $\Sigma = \{Q \in \mathfrak{C}_s(\IQbar) : \hat{h}(Q-P)
\le  h(s)/c_3 \}$ with $c_3 = 2M/c$. 
If $\#\Sigma < c_4$, then we are in alternative (ii) of the proposition. 

Finally, let us assume $\#\Sigma \ge c_4$. 
The discussion above implies
that there exist $Q_1,\ldots,Q_M \in
\Sigma$ such that $(Q_1 - P, \ldots, Q_M-P) \not\in
W_s(\IQbar)$, \textit{i.e.}, $(Q_1 - P, \ldots, Q_M-P)\in U(\IQbar)$.
Thus we can apply \eqref{EqHtInequalityBdRatPt} and obtain
\begin{equation*}
  h(s) \le \frac{1}{c }\left( M\frac{c h(s)}{2M}
    +c' \right) = \frac{1}{2} h(s) + \frac{c'}{c}.
\end{equation*}
Hence
$h(s)\le  2c'/c$. Now (\ref{EqTriangleInequality}) implies
$h_{\overline{\mathbb{A}_g},\cM}(\tau(s)) < c_1$ if $c_1 >
(2c'/c+c''')/c''$.
So this case cannot occur if
$h_{\overline{\mathbb{A}_g},\cM}(\tau(s)) \ge c_1$ and $c_1$ is
sufficiently large. 
\end{proof}


\section{Proof of Theorems~\ref{ThmBdRatIntro}, \ref{ThmBdFinRank}, and \ref{thm:BdTorIntroNF}}\label{SectionRatPt}

The goal of this section is to prove the theorems and the corollary in
the introduction. To this end let $g\ge 2$; we retain the notation of
$\mathsection$\ref{subsec:univfamily}. In particular,
$\pi\colon\mathfrak{A}_g\rightarrow \mathbb{A}_g$ is the universal
family of principally polarized abelian varieties of dimension $g$
with level-$\ell$-structure where $\ell\ge \bestlevel$ and
$\tau\colon \mathbb{M}_g\rightarrow \mathbb{A}_g$ is the Torelli morphism.

\begin{proposition}
  \label{prop:premazur}        
  The exist constants $c_1\ge 0,c_2\ge 1$ depending on the choices made
  above with the following property. Let $s\in \mathbb{M}_g(\IQbar)$
  with $h_{\overline{\mathbb{A}_g},\cM}(\tau(s)) \ge c_1$. Suppose $\Gamma$ is
  a finite rank subgroup of $\mathfrak{A}_{g,\tau(s)}(\IQbar)$ with rank
  $\rho\ge 0$. If $P_0 \in \mathfrak{C}_s(\IQbar)$, then
  $$
    \# (\mathfrak{C}_s(\IQbar)-P_0)\cap \Gamma \le
    c_2^{1+\rho}.
  $$
\end{proposition}

The proof combines Vojta's approach to the Mordell Conjecture with the
results obtained in $\mathsection$\ref{SectionDistanceCurve}. We
will use R\'emond's quantitative
version~\cite{Remond:Decompte,remond:vojtasup} of Vojta's method. A
similar approach was used in the authors's earlier work~\cite{DGH1p}
which also contains a review of Vojta's method in $\mathsection$2. Let
us recall the fundamental facts before proving Proposition~\ref{prop:premazur}.

Suppose we are given an abelian variety $A$ of dimension $g$ that is
defined over $\IQbar$ and is presented with a symmetric and very
ample line bundle $L$. We assume also that we have a closed immersion
of $A$ into some projective space $\IP_\IQbar^n$ determined by a basis
of the global sections of $L$. We assume that $A$ becomes a
projectively normal subvariety of $\IP_\IQbar^n$. This is the case if
$L$ is an at least fourth power of a symmetric and ample line bundle. 

Suppose $C$ is an irreducible  curve in $A$. Then let $\deg
C$ denote  the degree of $C$ considered as subvariety of
$A\subseteq \IP_\IQbar^n$, \textit{i.e.}, $\deg C = (C.L)$. Moreover, let $h(C)$ 
denote the height of $C$. 

On the ambient projective space  we have the 
 Weil height $h\colon
\IP_\IQbar^n(\IQbar)\rightarrow [0,\infty)$. Tate's Limit Argument,
compare \eqref{eq:fiberwiseNT},
applied to $h$ yields the N\'eron--Tate height $\hat h_L\colon A(\IQbar)\rightarrow
[0,\infty)$. It vanishes precisely on the points of finite order.
Moreover, it follows  from Tate's construction  that 
there exists a constant $c_{\mathrm{NT}}\ge 0$, which depends on $A$, such that
\begin{equation}
  \label{def:cNT}
  |\hat h_{L}(P) - h(P)|\le c_{\mathrm{NT}}
\end{equation}
for all $P\in A(\IQbar)$.

Finally, we need a measure for the heights of homogeneous polynomials that
define the addition and substraction on $A$, as required in R\'emond's
~\cite{remond:vojtasup}. Consider the $n+1$ global sections of $\cO(1)$
corresponding to the projective coordinates of $\mathbb{P}_\IQbar^n$.
They restrict to  global sections $\xi_0,\ldots,\xi_n$ of
$L$ on $A$. Let $f\colon A\times A\rightarrow A\times A$ denote the
morphism induced by $(P,Q)\mapsto (P+Q,P-Q)$, and let $p_1,p_2\colon
A\times A\rightarrow A$ be the first and section projection,
respectively. For all $i,j\in \{0,\ldots,n\}$ there are $P_{ij} \in
\IQbar[\mathbf{X},\mathbf{X}']$ with 
\begin{equation}
  \label{eq:Pijforh1}
  f^*(p_1^*\xi_i\otimes
  p_2^*\xi_j)=P_{ij}((p_1^*\xi_0,\ldots,p_1^*\xi_n),(p_2^*\xi_0,\ldots,p_2^*\xi_n))
\end{equation}
and where $P_{ij}$ is bihomogeneous of bidegree $(2,2)$ in $\mathbf{X}
= (X_0,\ldots,X_n)$ and $\mathbf{X}' = (X'_0,\ldots,X'_n)$; see 
 \cite[Proposition~5.2]{remond:vojtasup} with $a=b=1$
for the existence of the $P_{ij}$. Here we require that
$\xi_0,\ldots,\xi_n$ constitute a basis of $H^0(A,L)$. Let $h_1$
denote the Weil height of the point in projective space whose
coordinates are all coefficients of all $P_{ij}$.

We point out a minor omission in \cite[$\mathsection$2]{DGH1p}: $h_1$
there must also involve both addition and subtraction on $A$, and not just
the addition.

The lemma below is \cite[Corollary~2.3]{DGH1p} which is itself a
standard application of R\'emond's explicit formulation of the Vojta
and Mumford inequalities. We thus obtain a bound that is exponential in the
rank of the subgroup $\Gamma$ for points of sufficiently large
N\'eron--Tate height. 

\begin{lemma}
  \label{lem:vojtamumford}
  Let $C$ be an irreducible curve in $A$.
There exists a constant $c=c(n,\deg C)\ge 1$ depending
only on $n$ and $\deg C$ with the following property. Suppose $\Gamma$
is a subgroup of $A(\IQbar)$ of finite rank $\rho\ge 0$. If $C$ is not
the translate of an algebraic subgroup of $A$, then
\begin{equation*}
\# \left \{P \in C(\IQbar)\cap \Gamma : \hat h_{L}(P) >
c\max\{1,h(C),c_{\mathrm{NT}},h_1\} \right\} \le c^\rho.
\end{equation*}
\end{lemma}

\begin{proof}[Proof of Proposition \ref{prop:premazur}]
  As in $\mathsection$\ref{subsec:univfamily} 
  we have a
  closed immersion 
  $\mathfrak{A}_g \rightarrow\IP^n_\IQbar\times \mathbb{A}_g$ over $\mathbb{A}_g$.
  Let $s\in \mathbb{M}_g(\IQbar)$  with
  \begin{equation}
    \label{eq:hslbc1}
    h_{\overline{\mathbb{A}_g},\cM}(\tau(s)) \ge \max\{1,c_1\},
  \end{equation}
  where $c_1$
  comes from Proposition \ref{PropAlgPtFar} applied to $S = \mathbb{M}_g$.

  We now bound two quantities attached to the abelian variety
  $A=\mathfrak{A}_{g,\tau(s)}$ taken with its closed immersion into
  $\IP_\IQbar^n$. Observe that this closed immersion satisfies the
  condition imposed at the beginning of this section with $L=\cL|_A$
  where $\cL$ is as in $\mathsection$\ref{subsec:univfamily}.
  These quantities may depend on $s$. Below, $c>0$ denotes a constant
  that depends on the fixed data such as $g,n,$ and the ambient
  objects such as $\mathfrak{A}_g$ but not on  $s$.
  We will increase $c$ freely and without notice.

  \medskip
  \textbf{Bounding $c_{\mathrm{NT}}$.} For this we require the
  Silverman--Tate Theorem, Theorem \ref{thm:silvermantate}, applied to
  $\pi \colon \mathfrak{A}_g \rightarrow \mathbb{A}_g$. Recall that $h$ is the Weil 
  height on $\IP^n_{\IQbar}(\IQbar)$.
  For all $P\in
  A(\IQbar)$ we have $|h(P)-\hat h(P)|\le c
  \max\{1,h_{\overline{\mathbb{A}_g},\cM}(\tau(s))\}$;
  note that we can bound $h_{\overline S}(\pi(P))$ from above linearly
  in terms of $h_{\overline{\mathbb{A}_g},\cM}(\tau(s))$ by the Height Machine.
  So we may take 
  \begin{equation}
    \label{eq:choosecNT}
    c_{\mathrm{NT}} = c\max\{1,h_{\overline{\mathbb{A}_g},\cM}(\tau(s))\}.    
  \end{equation}

  \medskip
  \textbf{Bounding $h_1$.}\label{page:boundh1} 
  Recall that $f\colon A^2\rightarrow A^2$ sends $(P,Q)$ to
  $(P+Q,P-Q)$. We know that $P_{ij}$ as above exist. Here we will
  construct such a family with controlled height. To this end we
  consider points $P=[\zeta_0:\cdots:\zeta_n],Q=[\eta_0:\cdots:\eta_n]
  \in A(\IQbar)$. Then $f(P,Q) = ([\nu^+_0:\cdots
  :\nu^+_n],[\nu^-_0:\cdots:\nu^-_n])$. Recall that
  $A=\mathfrak{A}_{g,\tau(s)}$ is presented as a projectively normal
  subvariety of $\IP^n_\IQbar$ by the construction in
  $\mathsection$\ref{subsec:univfamily}. By \eqref{eq:Pijforh1}  there
  is for each $i,j\in \{0,\ldots,n\}$
  a bihomogeneous polynomial $P_{ij}$ of bidegree $(2,2)$ that is
  independent of $P$ and $Q$, with
  \begin{equation}
    \label{eq:inhomlinearsystemforh1}
    \nu^+_i \nu^-_j = \lambda P_{ij}((\zeta_0,\ldots,\zeta_n),(\eta_0,\ldots,\eta_n))
  \end{equation}
  for some non-zero $\lambda\in\IQbar$ that may depend on $(P,Q)$. 
  We  eliminate $\lambda$ and consider  
  \begin{equation}
    \label{eq:linearsystemforh1}
    \nu^+_i \nu^-_jP_{i'j'}((\zeta_0,\ldots,\zeta_n),(\eta_0,\ldots,\eta_n)) - \nu^+_{i'} \nu^-_{j'}P_{ij}((\zeta_0,\ldots,\zeta_n),(\eta_0,\ldots,\eta_n))=0
  \end{equation}
  as a system of homogeneous linear equations 
  parametrized by $(i,j),(i',j')\in \{0,\ldots,n\}^2$, the unknowns are
  the coefficients of the $P_{ij}$. As each $P_{ij}$ is bihomogeneous
  of bidegree $(2,2)$, the number of unknowns is
  $N=(n+1)^2 {n+2 \choose 2}^2$ which is independent of $s$. 

  Each pair of points $(P,Q)\in A(\IQbar)^2$ yields one system of
  linear equations.
  We know that there is a non-trivial solution $(P_{ij})_{ij}$ that solves for all
  $(P,Q)$ simultaneously and such that some $P_{ij}$ does not vanish
  identically on $A\times A$. Our goal is to find such a common solution of
  controlled height.

  First, observe that a common solution  for when
  $(P,Q)$ runs over all torsion points of $A^2(\IQbar)$ is a common
  solution for all pairs $(P,Q)$. Indeed, this follows as torsion
  points of $A^2(\IQbar)$ lie Zariski dense in $A^2$. 
  Second, observe that the full system has finite rank $M< N$ so it
  suffices to consider only finite many torsion points $(P,Q)$.

  Our task is thus to find a  common solution to
  (\ref{eq:linearsystemforh1}) for all $(i,j),(i',j')$, where some
  $P_{ij}$ does not vanish identically on $A\times A$, and where 
  $[\zeta_0:\cdots:\zeta_n]$,$[\eta_0:\cdots:\eta_n]$, and
  $[\nu^{\pm}_0:\cdots :\nu^{\pm}_n]$ are certain torsion points on
  $A(\IQbar)$. We may assume that some $\zeta_i$ is $1$ and similarly
  for  $\eta_i$ and
  $\nu^{\pm}_i$. So all coordinates are in $\IQbar$. Moreover, the
  height  of each torsion point is at most $c_{\mathrm{NT}}$ by (\ref{def:cNT}).
  The resulting system of linear equations is represented by an
  $M\times N$ matrix with algebraic coefficients. By elementary
  properties of the height, each coefficient in the system has affine
  Weil height $c \cdot c_{\mathrm{NT}}$, for $c$ large enough. It is tempting, but
  unnecessary, to invoke Siegel's Lemma to find a non-trivial solution.
  As $M,N$ are bounded in terms of $n$, 
  Cramer's Rule establishes the existence of a basis of 
  non-zero solution such that the Weil height of the coefficient
  vector is at most $c \cdot c_{\mathrm{NT}}$.
  Among this basis there
  is one solution where one $P_{ij}$ does not vanish identically on
  $A\times A$. By \eqref{eq:choosecNT} we find
  \begin{equation}
    \label{eq:h1bound}
    h_1 \le     c\max\{1,h_{\overline{\mathbb{A}_g},\cM}(\tau(s))\}
  \end{equation}
  for the projective height of the tuple $(P_{ij})_{ij}$. 
  Thus \eqref{eq:linearsystemforh1} holds and so we get
  \eqref{eq:inhomlinearsystemforh1} on all of $A^2(\IQbar)$, at least with  $\lambda$ a rational
  function on $A\times A$ that is not identically zero. But $\lambda$ cannot
  vanish anywhere on $A\times A$, as otherwise the left-hand side of
  \eqref{eq:inhomlinearsystemforh1} would vanish at some point of
  $A\times A$ for all $(i,j)$. Hence $\lambda$ is a non-zero constant. Replacing
  $P_{ij}$ by $\lambda P_{ij}$ does not change the projective height;
  we get \eqref{eq:Pijforh1} with the desired bound for $h_1$. 

  \medskip
  \textbf{Bounding height and degree of a curve.} 
  By Lemma \ref{lem:uniformdegreebound} we have
  \begin{equation}
    \label{eq:degreehgtbound}
    \deg(\mathfrak{C}_s-P_s)\le c\quad\text{and}\quad 
    h(\mathfrak{C}_s-P_s)\le
    c \max\{1,h_{\overline{\mathbb{A}_g},\cM}(\tau(s))\}
  \end{equation}
  for some
  $P_s\in \mathfrak{C}_s(\IQbar)$.

  \bigskip

  We now follow the argumentation in~\cite{DGH1p}. Let $\Gamma$ be a
  subgroup of $\mathfrak{A}_{g,\tau(s)}(\IQbar)$ for finite rank $\rho$.
  We first prove the proposition in the case $P_0=P_s$. We apply Lemma
  \ref{lem:vojtamumford}  to the curve $C =
  \mathfrak{C}_s-P_s\subset \mathfrak{A}_{g,\tau(s)}=A$ and use the bounds
  (\ref{eq:choosecNT}), \eqref{eq:h1bound}, and 
  (\ref{eq:degreehgtbound}).
  Note that $C$ is a smooth curve of genus
  $g\ge 2$. So it cannot be the translate of an algebraic subgroup of
  $A$.
  It follows that the number of points $P\in
  \mathfrak{C}_s(\IQbar)$ with $P-P_s\in \Gamma$ and $\hat h(P-P_s) >
  R^2$ where 
  \begin{equation}
    \label{def:bigR}
    R= (c\max\{1,h_{\overline{\mathbb{A}_g},\cM}(\tau(s))\})^{1/2}    
  \end{equation}
  is at most  $c^{\rho}\le c^{1+\rho}$.

  The burden of this paper is  to find a bound of the same quality 
  for the number of pairwise distinct points $P_1,P_2,P_3,\ldots$  in
  $\mathfrak{C}_s(\IQbar)$ with $\hat h(P_i-P_s) \le R^2$. This is where
  Proposition~\ref{PropAlgPtFar} enters. Recall our assumption
  (\ref{eq:hslbc1}) on $s$.
  Let $c_2'$ be the constant $c_2$ from Proposition~\ref{PropAlgPtFar}; it is independent of $s$. As
  $\# \Xi_s\le c_2'$ we may assume $P_i\not\in \Xi_s$ for all $i$. 
  So we may assume that each $P_i$ is in the second alternative of
  Proposition~\ref{PropAlgPtFar}.

  As in \cite{DGH1p} we use the Euclidean norm $|\cdot|$ defined by $\hat h^{1/2}$ on
  the $\rho$-dimensional  $\IR$-vector space $\Gamma\otimes\IR$. 
  Let $r\in (0,R]$. By an elementary ball packing argument, any subset of
  $\Gamma\otimes\IR$ contained in a closed ball of radius $R$ is
  covered by at most $(1+2R/r)^{\rho}$ closed balls of radius $r$ centered at
  elements of the given subset; see \cite[Lemme~6.1]{Remond:Decompte}.
  We apply this geometric argument to $R$ as in (\ref{def:bigR}) and
  to  $r$, the positive square-root of 
  $h_{\overline{\mathbb{A}_g},\cM}(\tau(s))/c_3 =
  \max\{1,h_{\overline{\mathbb{A}_g},\cM}(\tau(s))\}/c_3$.
  The contribution of the height
  $h_{\overline{\mathbb{A}_g},\cM}(\tau(s))$ cancels in the quotient
  $R/r$. We find $R/r\le c$. So the
  number of balls in the covering is at most $c^{1+\rho}$. 

  By Proposition~\ref{PropAlgPtFar}(ii) the number of the $P_i$'s that
  map to a single closed ball of radius $r$ is
  at most $c_4$. Thus after increasing $c$ we find that $\#\{P_i \in \mathfrak{C}_s(\IQbar)\cap \Gamma : \hat{h}(P_i-P_s) \le R^2\} \le c_4 c^{1+\rho}$, 
as desired. 
  This completes the proof of the proposition in the case $P_0=P_s$
  for sufficiently large $c_2$.

  The case of a general base point follows easily as our estimates
  depend only on the rank $\rho$ of $\Gamma$. Indeed,
  let $P_0\in \mathfrak{C}_s(\IQbar)$ be an arbitrary point and let
  $\Gamma'$ be the subgroup of $\mathfrak{A}_{g,\tau(s)}(\IQbar)$ generated by
  $\Gamma$ and $P_0-P_s$. Its rank is at most $\rho +1$.

  Now if $Q\in \mathfrak{C}_s(\IQbar)-P_0$ lies in $\Gamma$, then
  $Q+P_0-P_s\in \mathfrak{C}_s(\IQbar)-P_s$ lies in $\Gamma'$. The
  number of such $Q$ is at most $c_2^{2+\rho}$ by
  what we already proved. 
  The proposition follows as $c_2^{2+\rho}\le (c_2^2)^{1+\rho}$ and
  since we may replace $c_2$ by $c_2^2$. 
\end{proof}

\begin{proof}[Proof of Theorem~\ref{ThmBdRatIntro}] 
  It is possible to deduce Theorem~\ref{ThmBdRatIntro} from
  Theorem~\ref{ThmBdFinRank}, which we prove below. However in view of
  the importance of Theorem~\ref{ThmBdRatIntro}, we hereby give it a
  complete proof.

  This proof works for any level $\ell\ge\bestlevel$, but we may fix
  $\ell=\bestlevel$ for definiteness. Let
  $\mathbb{A}_g,\overline{\mathbb{A}_g},\cM,$
  and $h_{\overline{\mathbb{A}_g},\cM}$ be as in $\mathsection$\ref{subsec:univfamily}.

  Our curve $C$ corresponds to an $F$-rational point $s_F$ of $\mathbb{M}_{g,1}$, the coarse
  moduli space of smooth genus $g$ curves without level structure.

  The fine moduli space $\mathbb{M}_g$ of smooth genus $g$ curves
  with level-$\ell$-structure is a finite cover of $\mathbb{M}_{g,1}$.
  For this proof it is convenient to recall that 
  $\mathbb{M}_g$ is defined over
  the cyclotomic field generated by a third root of unity;
  recall the convention that we fixed a third root of unity and that
  $\mathbb{M}_g$ is geometrically irreducible.
  Say $s\in\mathbb{M}_g(\IQbar)$ maps to $s_F$. Then
  $F'=F(s)$ is a number field and  
  $[F':F]$ is bounded above
  only in terms of $g$ and $\ell$. We may identify $C_{F'} = C\otimes_F F'$ with $\mathfrak{C}_s$, the 
  fiber of
  $\mathfrak{C}_g\rightarrow\mathbb{M}_g$ above $s$.

  Constructing the
  Jacobian commutes with finite field extension. We thus view 
  $\Gamma = \mathrm{Jac}(C)(F)$ as a subgroup of $\mathrm{Jac}(C)(\IQbar)
  = \mathrm{Jac}(C_{F'})(\IQbar)$.

  To prove the theorem we may assume $C(F)\not=\emptyset$. So fix
  $P_0\in C(F)$. We consider the Abel--Jacobi embedding
  $C-P_0\subset \mathrm{Jac}(C)$ defined over $F$. Then $\# C(F)\le \#
  (C_{F'}(\IQbar)-P_0)\cap\Gamma = \#
  (\mathfrak{C}_s(\IQbar)-P_0)\cap \Gamma$. If
  $h_{\overline{\mathbb{A}_g},\cM}(\tau(s))\ge c_1$, the theorem
  follows from Proposition~\ref{prop:premazur}. 
  Note that in this case, the constant $c$ in \eqref{eq:CKexpinrank} is independent of $d$.

  So we may assume that the height 
  of $\tau(s)$ is less than $c_1$. 
  As $[F':\IQ]\le [F':F][F:\IQ]\le [F':F]d$,
  Northcott's Theorem implies
 that $\tau(s)$ comes   from a finite set in $\mathbb{A}_{g}(\IQbar)$
 that depends only on $g,d,$ and $\ell$. The same holds for $s$ and
 thus $F'=F(s)$
 since the Torelli morphism $\tau$ is finite-to-$1$ and thus   has
 fibers of bounded cardinality. 
 This means that the remaining $C$ are twists   in
 finitely many $F'$-isomorphism classes. But then it suffices to
 apply R\'emond's estimate \cite[page 643]{DPvarabII} to a single
 $C_{F'}$ and    use $\#C(F) \le \# (C_{F'}(\IQbar)-P_0)\cap\Gamma$ to
 conclude the theorem. 

Silverman's older result~\cite[Theorem 1]{Silverman:twists} also
 handles uniformity among twists.
\end{proof}

Let us explain how to obtain some extra uniformity in the second case
of the proof. More precisely, we show that the constant $c(g,d)$ in
\eqref{eq:CKexpinrank} grows polynomially in $d$. We retain the above
proof's notation.

Denote by $\rho \colon \mathbb{A}_g = \mathbb{A}_{g,\ell} \rightarrow
\mathbb{A}_{g,1}$ the natural morphism to the coarse moduli space which forgets the level structure.
We recall that
$\overline{\mathbb{A}_g}$ 
is presented as a closed subvariety
of
projective space induced by a basis of global sections of a positive
powers of ample line bundle $\cM$. Let
$\iota\colon\mathbb{A}_{g,1}\rightarrow\IP_{\mathbb{Q}}^m$ be an
immersion, as before Theorem~\ref{ThmBdFinRank}.
By 
\cite[Lemma~4]{Silverman:heightest11}, there exist $c' > 0 $ and $c''
\ge 0$ depending on the immersions such that
$h(\iota(\rho(t))) \le c'
h_{\overline{\mathbb{A}_g},\cM}(t) + c''$ for all $t\in
\mathbb{A}_g(\IQbar)$.
So $h(\iota(\rho(\tau(s)))) \le c'c_1+c''$
is bounded uniformly in the second case.

By fundamental work of Faltings~\cite[$\mathsection 3$ including
 the proof of Lemma~3]{Faltings:ES}, see also
\cite[the remarks below Proposition~V.4.4 and Proposition~V.4.5]{FaltingsChai}, the \textit{stable Faltings height} of
$\mathfrak{A}_{g,\tau(s)}$ is bounded from above in terms of
$c'c_1+c''$ and $g$ only.
The height $h_{\mathrm{DP}}(\mathfrak{A}_{g,\tau(s)})$ used by David and
Philippon is bounded similarly by work of Bost and David,
see~\cite[Corollaire~6.9]{DPvarabII} and \cite{Pazuki:12}. 

In R\'emond's bound  \cite[page 643]{DPvarabII}
for $\#(C_{F'}(\IQbar)-P_0)\cap \Gamma$,
the base in the exponential depends polynomially on 
$D\max\{1,h_{\mathrm{DP}}(\mathfrak{A}_{g,\tau(s)})\}$,
where $D$ is the degree over
$\IQ$ of a suitable field of definition of $\mathfrak{A}_{g,\tau(s)}$.
As this abelian variety can be defined over $F'$ we may assume $D\le
[F':F]d$ is bounded linearly in $d$.
Recall that $\deg (C_{F'}-P_0)$ is bounded from above uniformly. 
So R\'emond's bound implies that  $c(g,d)$
in (\ref{eq:CKexpinrank}) can be chosen to grow at most
polynomially in $d$.

The definition of $h_{\mathrm{DP}}(\mathfrak{A}_{g,\tau(s)})$
involves  theta functions and a different kind of level structure.
Using standard results on heights and by
going down and up in the level structure it is likely that one can
bound $h_{\mathrm{DP}}(\mathfrak{A}_{g,\tau(s)})$ from above directly in terms of
$h_{\overline{\mathbb{A}_g},\cM}(\tau(s))$. For this one would need to
work with a different level $\ell$ in the proof of
Theorem~\ref{ThmBdRatIntro}.

\begin{proof}[Proof of Theorem~\ref{ThmBdFinRank}] 
  We keep the same notation as in the proof of
  Theorem~\ref{ThmBdRatIntro}. So $\ell = \bestlevel$ and
  $\mathbb{A}_g,\overline{\mathbb{A}_g},\cM,$ and
  $h_{\overline{\mathbb{A}_g},\cM}$
  are as in $\mathsection$\ref{subsec:univfamily}.

  Let $C$ be a smooth curve of genus $g\ge 2$ defined over $\IQbar$,
  and let $\Gamma$ be a finite rank subgroup of
  $\mathrm{Jac}(C)(\IQbar)$. Let $P_0 \in C(\IQbar)$.
  
  The curve $C$
  corresponds to a $\IQbar$-point $s_{\mathrm{c}}$ of
  $\mathbb{M}_{g,1}$.

  The fine moduli space $\mathbb{M}_g$ of smooth genus $g$ curves
  with level-$\ell$-structure is a finite covering of $\mathbb{M}_{g,1}$. So there exists an $s \in \mathbb{M}_g(\IQbar)$ 
  that maps to $s_{\mathrm{c}}$. Thus $C$ is isomorphic, over $\IQbar$, to the fiber $\mathfrak{C}_s$ of the universal family 
  $\mathfrak{C}_g\rightarrow\mathbb{M}_g$. We thus view $\Gamma$ as a
  finite rank subgroup of $\mathrm{Jac}(\mathfrak{C}_s)(\IQbar)$, and
  $P_0 \in \mathfrak{C}_s(\IQbar)$. 
  
  Consider the Abel--Jacobi embedding $C-P_0 \subseteq
  \mathrm{Jac}(C)$. Then $\# (C(\IQbar)-P_0) \cap \Gamma = \#
  (\mathfrak{C}_s(\IQbar) - P_0)\cap \Gamma$. If $h_{\overline{\mathbb{A}_g},\cM}(\tau(s)) \ge c_1$, then $\# (C(\IQbar)-P_0) \cap \Gamma \le c_2^{1+\rho}$ 
  by Proposition~\ref{prop:premazur}. 

  Thus it suffices to find a constant $c'_1\ge 0$ that is independent
  of $C$ and such 
  that $h(\iota([\mathrm{Jac}(C)])) \ge c_1'$ implies
  $h_{\overline{\mathbb{A}_g},\cM}(\tau(s)) \ge c_1$. 
  
  As after the proof of Theorem~\ref{ThmBdRatIntro}
  denote by $\rho \colon \mathbb{A}_g = \mathbb{A}_{g,\ell}
  \rightarrow \mathbb{A}_{g,1}$ the natural morphism.
  As in the proof of Theorem~\ref{ThmBdFinRank} we use
  $h(\iota(\rho(t)))\le c'h_{\overline{\IA_{g}},\cM}(t)  + c''$ for all $t\in
  \mathbb{A}_g(\IQbar)$.
  The theorem follows since $\rho(\tau(s)) = [\mathrm{Jac}(C)]$.
  %
\end{proof}
\begin{remark}
It is possible to prove Theorem~\ref{ThmBdRatIntro} (without the
dependency claims on $c(g,d)$) using
Theorem~\ref{ThmBdFinRank}. Let $C$ be a smooth curve of genus $g \ge
2$ defined over a number field $F \subseteq \IQbar$. Then by taking $\Gamma =
\mathrm{Jac}(C)(F)$ in Theorem~\ref{ThmBdFinRank}, we can conclude
Theorem~\ref{ThmBdRatIntro} if $h(\iota([\mathrm{Jac}(C)])) \ge c_1$. 
  The case $h(\iota([\mathrm{Jac}(C)])) < c_1$ can be 
  handled as in the
  proof of Theorem~\ref{ThmBdRatIntro}, and one can furthermore obtain 
  extra uniformity
  for $c_2$ in Theorem~\ref{ThmBdFinRank} by applying R\'emond's bound
  \cite[page~643]{DPvarabII} as after the proof of Theorem~\ref{ThmBdRatIntro}.
\end{remark}


\begin{proof}[Proof of Theorem~\ref{thm:BdTorIntroNF}]
Let $C$ be a smooth curve of genus $g\ge 2$  defined over a number field $F\subset\IQbar$.
  
Apply Theorem~\ref{ThmBdFinRank} to $C_{\IQbar}$, $P_0 \in C(\IQbar)$
and $\Gamma = \jac(C)({\IQbar})_{\mathrm{tors}}$, whose rank is $0$. Then we obtain $c_1 \ge 0$ and $c_2 \ge 1$ such that
  \[
  \#(C(\IQbar) - P_0) \cap \jac(C)({\IQbar})_{\mathrm{tors}} \le c_2
  \]
  if $h(\iota([\mathrm{Jac}(C_\IQbar)])) \ge c_1$.
  
By the Northcott property and Torelli's Theorem, there are up-to
$\IQbar$-isomophism only finitely many $C_{\IQbar}$'s  defined over a number field $F$ with $[F:\IQ] \le d$ such that  $h(\iota([\mathrm{Jac}(C_\IQbar)])) < c_1$. By applying Raynaud's result on the Manin--Mumford Conjecture to each one of these finitely many curves separately, we obtain Theorem~\ref{thm:BdTorIntroNF}.
\end{proof}



\appendix
\renewcommand{\thesection}{\Alph{section}}
\setcounter{section}{0}



\section{The Silverman--Tate Theorem revisited}\label{app:silvermantate}

Our goal in this appendix is to present a treatment of the
Silverman--Tate Theorem, \cite[Theorem~A]{Silverman}, using the language
of Cartier divisors. Using Cartier divisors as opposed to Weil
divisors  allows us to relax the flatness 
hypotheses imposed on $\pi$ in the notation of \cite[$\mathsection$3]{Silverman}. Apart from this minor tweak
we closely follow the original argument presented by Silverman. 

Suppose $S$ is a {regular},
{irreducible}, {quasi-projective} variety over $\IQbar$.
Let $\pi \colon \cA\rightarrow S$ be an abelian scheme.
We write $\eta$ for the generic point of $S$ and $\cA_\eta$ for the
generic fiber of $\pi$. Then $\cA_\eta$ is an abelian variety defined
over $\IQbar(\eta)$. 

Suppose we are presented with a closed immersion
$\cA\rightarrow \IP_{\IQbar}^n\times S$ over $S$ and
 with a projective variety $\overline S$
containing $S$ as a Zariski open and dense subset.
We will assume that $\overline S$ is embedded into $\IP_\IQbar^m$.
We do not assume that $\overline S$ is regular. 

We identify $\cA$ with a subvariety of $\IP_\IQbar^n\times S$.
Moreover, let $\overline{\cA}$ denote the Zariski closure of 
$\cA$ in $\IP_S^n=\IP_{\IQbar}^n\times\overline S \subset
\IP_{\IQbar}^n\times\IP_\IQbar^m$. 

We set $\overline\cL = \cO(1,1)|_{\overline{\cA}}$ and  $\cL
= \overline\cL|_{\cA}$. We  will  assume in addition that 
$[-1]^*\cL_\eta \cong \cL_\eta$ 
where  $\cL_\eta$ is the restriction of $\cL$ to $\cA_\eta$. 
This implies
$[2]^*\cL_\eta \cong \cL_\eta^{\otimes 4}$.

\newcommand{\hcA}{h}

Given these immersions, we have several height functions.
For
$(P,s)\in \overline{\cA}(\IQbar)\subset \IP^n_\IQbar(\IQbar)\times \IP^m_\IQbar(\IQbar)$ we define
$\hcA(P,s) = h(P)+h(s)$ using the Weil height.
Moreover, for $s\in \overline S(\IQbar)\subset \IP^m_\IQbar(\IQbar)$
we define $h_{\overline S}(s) = h(s)$. 
Finally, 
for all  $P\in \cA(\IQbar)$ we denote by
  \begin{equation*}
    \hat h_{\cA}(P) = \lim_{N\rightarrow \infty}
    \frac{\hcA([N](P))}{N^2}
  \end{equation*}
the N\'eron--Tate height with respect to $\cL$;  it is
well-known that the limit converges, \textit{cf.} the reference around
(\ref{eq:fiberwiseNT}). 
  
We will prove the following variant of the Silverman--Tate Theorem.

\begin{theorem}
  \label{thm:silvermantate}
  There exists a constant $c>0$  such that for all 
  $P \in \cA(\IQbar)$ we have 
  \begin{equation*}
    \bigl| \hat h_{\cA}(P) -
    \hcA(P)\bigr| \le c  \max\{1,h_{\overline
      S}(\pi(P))\}. 
  \end{equation*}
\end{theorem}

The constant $c$ depends on $\cA$ and on the various immersions but not on $P$. 
The proof is distributed over  the  next subsections.

\subsection{Extending multiplication-by-$2$}
We keep the notation from the  previous subsection.
We have  constructed a (very naive) projective model $\overline\cA$ of $\cA$.
Note that $\overline\cA$ and $\overline S$ may fail to be regular.
Moreover, the natural morphism
$\overline\cA\rightarrow \overline S$, which we also denote by $\pi$,  may fail to be smooth or even
flat.

Multiplication-by-$2$ is a morphism $[2]\colon \cA\rightarrow \cA$
that extends to a rational map
$\overline{\cA}\dashrightarrow\overline{\cA}$.
We consider the graph of $[2]$ on $\cA$ as a subvariety of
$\cA\times_S\cA$.
Let $\overline{\cA}'$ be the Zariski closure of this graph inside
$\overline{\cA}\times_{\overline S}\overline{\cA}$. 
 Write $\rho\colon
\overline{\cA}'\rightarrow\overline{\cA}$ for the restriction of the projection onto the
first factor and  $[2]$ for the restriction onto the second factor.
We may identify $\cA$ with a Zariski open subset of
$\overline{\cA}'$. Under this identification, $\rho$ restricts to the
identity on $\cA$ and $[2]$ restricts to multiplication-by-$2$ on
$\cA$. 

The following diagram  commutes
\begin{equation*}
  \xymatrix{
 & \overline{\cA} \ar@{}[d]|-*[@]{\supset}      &\ar[l]_\rho \overline{\cA}' \ar@{}[d]|-*[@]{\supset}   \ar[r]^{[2]} & \overline{\cA}\ar@{}[d]|-*[@]{\supset} \\
& \cA \ar@{=}[r]  & {\cA}      \ar[r]^{[2]} &{\cA} 
  }  
\end{equation*}
where the first and third inclusions are equal and the middle one
comes from the identification involved in the graph construction.

\subsection{Proof of the Silverman--Tate Theorem}

We keep the notation from the  previous subsection.

\begin{proposition}
  There exists a constant $c_1>0$ such that 
  \begin{equation}
    \label{eq:duplicationbound}
\left|   \hcA([2](P)) - 4
   \hcA(P)\right|  \le c_1  \max\{1,h_{\overline
   S}(\pi(P))\}
  \end{equation}
  holds for all $P\in \cA(\IQbar)$. 
\end{proposition}
\begin{proof}
  We define
  \begin{equation}
  \label{def:cF2}
  \cF' =  [2]^* \overline{\cL} \otimes
  \rho^*\overline{\cL}^{\otimes(-4)} \in \mathrm{Pic}(\overline{\cA}'). 
\end{equation}
  
Recall that we have identified $\cA$ with a Zariski open subset of
$\overline{\cA}'$. 
The  restriction of $[2]^*\overline{\cL}$ to the generic fiber
$\cA_\eta\subset\cA\subset\overline{\cA}'$
coincides with $[2]^*\cL_\eta$ and the restriction of $\rho^*\overline{\cL}$  to $\cA_\eta$
is identified with $\cL_\eta$. Using our assumption
 $[2]^*\cL_\eta \cong
\cL_\eta^{\otimes 4}$ on the generic fiber $\cA_\eta$ we see that  $\cF'$ is trivial on $\cA_\eta$.

By \cite[Corollaire~21.4.13 (pp.~361 of EGA IV-4, in Errata et Addenda, liste~3)]{EGAIV} 
applied to
$\cA\rightarrow S$ there exists a line
bundle $\cM$ on $S$ such that $\pi|_{\cA}^*\cM \cong \cF'|_{\cA}$.

Let us first desingularize the compactified base $\overline S$ by applying
Hironaka's Theorem. Thus there is a proper, birational morphism
$b\colon \overline S'\rightarrow\overline S$ that is an isomorphism
above $S$
 such that  $\overline S'$
is regular. We consider $S$ as Zariski open in $\overline S'$.  Note that $b$ is even
projective and $\overline S'$ is integral. So $\overline S'$ is an
irreducible, regular, projective variety. 

Now consider the base change $\overline
\cA'\times_{\overline S}\overline S'$. This new scheme
may  fail to be irreducible or even
reduced. However, recall that $b$ is an isomorphism above the regular $S\subset
\overline S$. So $(\overline
\cA'\times_{\overline S}\overline S')_S =
\overline{\cA}'\times_{\overline S}S$  is isomorphic to $\cA$ and thus
integral. We may consider $\cA$ as an open subscheme of
$\overline
\cA'\times_{\overline S}\overline S'$. It must be contained in an irreducible component of
$\overline
\cA'\times_{\overline S}\overline S'$. We endow this irreducible component with the
reduced induced structure and obtain an integral, closed subscheme
$\overline{\overline{\cA}}\subset \overline
\cA'\times_{\overline S}\overline S'$. We get a commutative diagram
\begin{equation*}
  \xymatrix{
\cA\ar@{}[r]|-*[@]{\subset}\ar[d]_\pi &
 \overline{\overline{\cA}} \ar[r]^f \ar[d]_{\overline\pi} & \overline{\cA}' \ar[d]_{\pi\circ\rho}&\\
S \ar@{}[r]|-*[@]{\subset} &\overline{S}' \ar[r] &      \overline{S} &
  }  
\end{equation*}
The horizontal  morphisms compose to the identity on the domain. 

We consider $S$ as a Zariski open subset of $\overline S'$. 
As $\overline S'$ is regular, we can extend $\cM$ to a line bundle on
the regular $\overline S'$, \textit{cf.} \cite[Corollary~11.41]{GoertzWedhorn}. 
The pull-back $f^*\cF'\otimes\overline\pi^*\cM^{\otimes(-1)}$ is
trivial on $\cA\subset\overline{\overline{\cA}}$. 

By Hironaka's Theorem there is a proper, birational morphism $\beta\colon\widetilde
{\cA}\rightarrow\overline{\overline\cA}$ that is an isomorphism above $\cA$
(which is regular) 
such that $\widetilde{\cA}$ is regular. We may identify $\cA$ with a Zariski open subset of
$\widetilde \cA$.

Now we pull everything back to the regular  $\widetilde\cA$.
More precisely, we set $\cF = \beta^*f^*\cF'$.
Then
$\cF\otimes\beta^*\overline\pi^*\cM^{\otimes(-1)}$ is trivial when restricted to
$\cA$. 

To a Cartier divisor $D$ we attach its line bundle $\cO(D)$.
As $\widetilde \cA$ is integral we may fix a Cartier divisor $D$ on
$\widetilde \cA$ with $\cO(D) \cong \cF\otimes \beta^*\overline\pi^*
\cM^{\otimes(-1)}$. Let $\mathrm{cyc}(D)$ denote the Weil divisor of $\widetilde\cA$
attached to $D$. The linear equivalence class of $\mathrm{cyc}(D)$
restricted to $\cA$ is trivial. By \cite[Proposition~11.40]{GoertzWedhorn}
$\mathrm{cyc}(D)$ is linearly equivalent to a Weil divisor $\sum_{i=1}^r n_i
Z_i$ with  $Z_i \subset\widetilde{\cA}\ssm \cA$ irreducible and of
codimension $1$ in $\widetilde{\cA}$. 

We let $\widetilde\pi$ denote the composition
$\widetilde\cA\rightarrow \overline{\overline\cA}\rightarrow \overline{S}'$. 
Let us consider $\widetilde\pi(Z_i) = Y_i$. As $\widetilde\pi$ is proper, each $Y_i$ is
an irreducible closed subvariety of $\overline S'$. Moreover, $Y_i
\subset \widetilde\pi(\widetilde{\cA}\ssm \cA) \subset \overline S'\ssm S$. So $Y_i$ has
dimension at most $\dim \overline S' - 1$. But $Y_i$ could have
codimension at least $2$ and thus fail to be the support of a Weil divisor.
On the regular $\overline S'$ a 
Cartier divisor is the same thing as a Weil divisor;
see \cite[Theorem~11.38(2)]{GoertzWedhorn}.
For each $i$ we fix a Cartier divisor  $E_i$ of $\overline S'$
such that   $\mathrm{cyc}(E_i)$ equals a prime Weil divisor supported on an irreducible
subvariety containing $Y_i$.
 Since $\mathrm{cyc}(E_i)$ is effective, we find that $E_i$ is effective, see \cite[Theorem~11.38(1)]{GoertzWedhorn}
 and its proof. 
An effective Cartier divisor and its image under the cycle map
$\mathrm{cyc}(\cdot)$ have equal support. So the subscheme of $\overline S'$ attached to  $E_i$
contains $Y_i$. 

The pull-back $\widetilde\pi^*E_i$ is well-defined as a Cartier
divisor, we do not require that $\pi$ is flat, \textit{cf.}
\cite[Proposition~11.48(b)]{GoertzWedhorn}. 
By \cite[Corollary~11.49]{GoertzWedhorn} the inverse image
$\widetilde\pi^{-1}(E_i)$, taken as  a subscheme of $\widetilde{\cA}$
is the subscheme attached to  $\widetilde\pi^*E_i$ and $\widetilde\pi^*E_i$ is
effective. 

Note that $\widetilde\pi^{-1}(E_i)\supset \widetilde\pi^{-1}(Y_i)\supset Z_i$. 
The support satisfies $\mathrm{Supp}(\widetilde\pi^* E_i)\supset Z_i$.
Moreover, as $\widetilde\pi^*E_i$ is effective,  $\mathrm{cyc}(\widetilde\pi^* E_i)$ is effective and
 $\mathrm{Supp} (\mathrm{cyc}(\widetilde\pi^* E_i)) =
\mathrm{Supp}(\widetilde\pi^*E_i)$.
Thus
\begin{equation*}
\pm   \sum_{i=1}^r n_i Z_i \le \mathrm{cyc}\left(\widetilde\pi^*\sum_{i=1}^r |n_i|
  E_i\right).
\end{equation*}

Recall that $\mathrm{cyc}(D) = \mathrm{cyc}(\mathrm{div}\phi) +
\sum_{i=1}^r n_i Z_i$ for some rational function $\phi$ on
$\widetilde{\cA}$. Therefore,
\begin{equation*}
0\le   \mathrm{cyc}\left(\pm (D-\mathrm{div}\phi) + \widetilde\pi^*\sum_{i=1}^r |n_i| E_i\right).
\end{equation*}

Since $\widetilde{\cA}$ is regular and in particular normal, we find  that 
\begin{equation}
\label{eq:thisiseffective}\pm(D - \mathrm{div}\phi) + \widetilde\pi^*\sum_{i=1}^r |n_i|E_i
\end{equation}
is an effective Cartier divisor for both  signs;
see \cite[Theorem~11.38(1)]{GoertzWedhorn} and its proof.
Moreover, its support equals the support of 
$$0\le \mathrm{cyc}\left(\pm (D-\mathrm{div}\phi) + \widetilde\pi^*\sum_{i=1}^r |n_i|
  E_i\right) =
\pm \mathrm{cyc}(D-\mathrm{div}\phi) + \sum_{i=1}^r |n_i|\mathrm{cyc}(\widetilde\pi^*E_i).$$
Thus the support of (\ref{eq:thisiseffective}) lies in
$\bigcup_{i=1}^r \mathrm{Supp}(\widetilde\pi^* E_i)$.

We apply $\cO(\cdot)$ and pass  again  to line bundles.
Let us denote $\mathcal{E} = \mathcal{O}(\sum_{i=1}^r |n_i| E_i)$, a
line bundle on $\overline S'$. The line bundle attached to 
 (\ref{eq:thisiseffective}) is $(\mathcal{F}\otimes\beta^*\overline\pi^*\mathcal{M}^{\otimes(-1)} 
 )^{\otimes(\pm 1)}\otimes\widetilde\pi^*\mathcal{E}$.
 Since (\ref{eq:thisiseffective}) is effective, both
$(\mathcal{F}\otimes\beta^*\overline\pi^*\mathcal{M}^{\otimes(-1)} 
 )^{\otimes(\pm 1)}\otimes \widetilde\pi^*\mathcal{E}$
have a non-zero global
 section.

By the Height Machine this translates to 
 \begin{equation*}
   h_{\widetilde \cA,(\mathcal{F}\otimes\beta^*\overline\pi^*\mathcal{M}^{\otimes(-1)} 
 )^{\otimes(\pm 1)}\otimes \widetilde\pi^*\mathcal{E}}(P)\ge O(1)
 \end{equation*}
 for all $\widetilde P\in \widetilde\cA(\IQbar)$ 
with $\widetilde\pi(\widetilde P)\not\in\bigcup_i \mathrm{Supp}(E_i)$. 
By functoriality properties of  the Height Machine we obtain 
 \begin{equation*}
  |h_{\overline \cA',\cF'}(f(\beta(\widetilde P)))| \le h_{\overline S',\cE}(\widetilde\pi(\widetilde P)) + |h_{\overline
    S', \cM}(\overline{\pi}(\beta(\widetilde P)))|
  +O(1)
 \end{equation*}
 for the same $\widetilde P$. 
 We recall (\ref{def:cF2}) and again use the Height Machine to find
 \begin{equation*}
\left|   \hcA([2](P')) - 4
   \hcA(\rho(P'))\right|   \le h_{\overline
   S',\cE}(\widetilde\pi(\widetilde P)) + |h_{\overline S',
     \cM}(\widetilde\pi(\widetilde P))|
  +O(1)
 \end{equation*}
where $P'=f(\beta(\widetilde P))$. Observe that all points of $\cA(\IQbar)$ are
in the image of $f\circ \beta$. 

We recall that the desingularization morphism $\overline S'\rightarrow\overline S$ is an
isomorphism above $S$ and that we have identified $\cA$ with a Zariski
open subset of $\overline{\cA}'$ and of $\overline{\cA}$.
Under these identifications and if $P'$ corresponds to $P\in
\cA(\IQbar)$, then $[2](P')$ is the duplicate of $P$,
$\rho(P')= P$, and $\widetilde\pi(\widetilde P) = \pi(\rho(P'))=\pi(P)$.  We apply the Height
 Machine a final time and use that $h_{\overline S}$ arises from the
 Weil height restricted to $\overline{S}(\IQbar)$. We
 find 
 \begin{equation*}
\left|   \hcA([2](P)) - 4
   \hcA(P)\right|   \le c_1
 \max\{1,h_{\overline S}(\pi(P))\}
 \end{equation*}
 for all $P\in \cA(\IQbar)$   with  $\widetilde\pi(\widetilde
 P)\not\in \bigcup_i \mathrm{Supp}(E_i)$, under the identifications
 above. 

Let $P\in\cA(\IQbar)$. As the $Y_i$ lie in $\widetilde\pi(\widetilde\cA\ssm
\cA)$ we can choose all $E_i$ above to avoid $\pi(P)$.  After doing this finitely often (using
noetherian induction) and replacing the $E_i$ from before and
adjusting $c_1$, we find 
\begin{equation*}
\left|   \hcA([2](P)) - 4
   \hcA(P)\right|  \le c_1 \max\{1,h_{\overline S}(\pi(P))\}
\end{equation*}
for all $P\in \cA(\IQbar)$ where $c_1>0$ is independent of $P$.
\end{proof}

\begin{proof}[Proof of Theorem~\ref{thm:silvermantate}]
  Having (\ref{eq:duplicationbound}) at our disposal the proof follows
  a well-known argument.   
  Indeed, say $l\ge k\ge 0$ are integers.
  Then applying the triangle inequality to the appropriate telescoping
  sum yields
  \begin{alignat*}1
    \left|\frac{\hcA([2^l](P))}{4^l}- \frac{\hcA([2^k](P))}{4^k}\right|
    &\le \sum_{m=k}^{l-1}
        \left|\frac{\hcA([2^{m+1}](P))}{4^{m+1}}-
        \frac{\hcA([2^m](P))}{4^m}\right|\\
        &\le \sum_{m=k}^{l-1}
        4^{-(m+1)}    \left|{\hcA([2^{m+1}](P))}- 4{\hcA([2^m](P))}\right|.
  \end{alignat*}
  We apply (\ref{eq:duplicationbound}) to $[2^m](P)$ and find that the
  sum is bounded by $c_1 x \sum_{m=k}^{l-1}4^{-(m+1)} \le c_1 x 4^{-k}$
  where $x = \max\{1,h_{\overline S}(\pi(P))\}$.
  So $\left(\hcA([2^l](P))/4^l\right)_{l\ge 1}$ is a Cauchy sequence
  with limit
  $\hat h_{\overline\cA}(P)$. 
  Taking $k=0$ and $l\rightarrow \infty$ we obtain from the estimates
  above that
  $|\hat h_{\overline\cA}(P) - \hcA(P)|\le c_1x$, as
  desired.   
\end{proof}



\section{Full version of Theorem~\ref{ThmHtInequality}}\label{AppHtIneqApp}

The goal of this section is to prove the full version of
Theorem~\ref{ThmHtInequality}, \textit{i.e.}, without assuming {\tt
(Hyp)}. Let $S$ be an irreducible quasi-projective variety defined over $\IQbar$ and let $\pi\colon\cA\rightarrow S$ be an abelian scheme of relative dimension $g \ge 1$.

Let $\cL$ be a relative ample line bundle on $\cA \rightarrow S$ with $[-1]^*\cL=\cL$, and
let $\cM$ be an ample line bundle on a compactification $\overline{S}$
of $S$. All data above are assumed to be defined over $\IQbar$. Set
$\hat{h}_{\cA,\cL} \colon \cA(\IQbar) \rightarrow \IR$ to be the
fiberwise N\'{e}ron--Tate height $\hat{h}_{\cA,\cL}(P) =
\hat{h}_{\cA_s,\cL_s}(P)$ with $s = \pi(P)$, and $h_{\overline{S},\cM}
\colon \overline{S}(\IQbar) \rightarrow \IR$ to be a representative of
 the height provided by the Height Machine; \textit{cf.} \cite[Chapter~2 and 9]{BG}.

The main result of this appendix is the following theorem.

\begin{theorem}\label{ThmHtInequalityFull}
Let $X$ be an irreducible subvariety of $\cA$
defined over $\IQbar$.   Suppose $X$ is non-degenerate,
as defined in Definition~\ref{DefinitionNonDegenerateApp}. Then there
exist constants $c_1>0$ and $c_2\ge 0$ and a Zariski open dense subset
$U$ of $X$ with 
\begin{equation}\label{EqHtInequalityFull}
\hat{h}_{\cA,\cL}(P) \ge c_1 h_{\overline{S},\cM}( \pi(P) ) - c_2 \quad \text{for all}\quad P \in U(\IQbar).
\end{equation}
\end{theorem}

Compared to Theorem~\ref{ThmHtInequality}, $\cA \rightarrow S$ is no longer required to satisfy {\tt (Hyp)}. Other minor improvements are that $S$ is not required to be regular and $X$ is not required to be closed.

In Definition~\ref{DefinitionNonDegenerate}, we defined non-degenerate subvarieties using the generic rank of the Betti map if $\cA \rightarrow S$ satisfies {\tt (Hyp)}. For an arbitrary $\cA \rightarrow S$, the definition is similar. But we need to first of all extend our construction of the Betti map, Proposition~\ref{PropBettiMap}, to an arbitrary $\cA \rightarrow S$.

\subsection{Betti map}
In this subsection, we extend the construction of the Betti map, \textit{i.e.}, Proposition~\ref{PropBettiMap}, to an arbitrary $\cA \rightarrow S$ with $S$ regular. 
In this subsection, let $S$ be an  irreducible, regular, quasi-projective variety over $\IC$ and let $\pi \colon \cA \rightarrow S$ be an abelian scheme of relative dimension $g \ge 1$.

\begin{proposition}\label{PropBettiMapApp}
Let $s_0 \in S(\IC)$. Then there exist an open neighborhood $\Delta$ of $s_0$ in $S^{\mathrm{an}}$, and a map $b_{\Delta} \colon \cA_{\Delta} := \pi^{-1}(\Delta) \rightarrow \mathbb T^{2g}$, called the Betti map, with the following properties.
\begin{enumerate}
\item[(i)] For each $s \in \Delta$ the restriction $b_{\Delta}|_{\cA_s(\IC)} \colon \cA_s(\IC) \rightarrow \mathbb T^{2g}$ is a group isomorphism.
\item[(ii)] For each $\xi \in \mathbb T^{2g}$ the preimage $b_{\Delta}^{-1}(\xi)$ is a complex analytic subset of $\cA_{\Delta}$.
\item[(iii)] The product
$(b_{\Delta},\pi) \colon \cA_{\Delta} \rightarrow \mathbb
T^{2g} \times \Delta$ is a real analytic isomorphism.
\end{enumerate}
\end{proposition}

Just as in the case of
Proposition~\ref{PropBettiMap}, the Betti map
is uniquely determined  by properties (i) and (iii)
up-to the action of $\mathrm{GL}_{2g}(\IZ)$ if $\Delta$ is
connected. 
 Composing with an $\alpha \in \mathrm{GL}_{2g}(\IZ)$ does not change the rank. So by the discussion on the uniqueness above, any map $\cA_{\Delta} \rightarrow \mathbb{T}^{2g}$ satisfying the three properties listed in Proposition~\ref{PropBettiMapApp} will be called \textit{Betti map}.
\begin{proof} Our proof of Proposition~\ref{PropBettiMapApp} follows the construction in \cite[$\mathsection$3-$\mathsection$4]{GaoBettiRank}. We divide it into several steps.

By \cite[$\mathsection$2.1]{GenestierNgo}, $\cA \rightarrow S$ carries a polarization of type $D = \mathrm{diag}(d_1,\ldots,d_g)$ for some positive integers $d_1|d_2|\cdots|d_g$.

  \textbf{Case: Moduli space with level structure.} 
Fix $\ell \ge \bestlevel$ with $(\ell,d_g) = 1$. We start by proving Proposition~\ref{PropBettiMapApp} for $S = \mathbb
A_{g,D,\ell}$, the moduli space of abelian varieties of
dimension $g$ polarized of type $D$ with level-$\ell$-structure. It is a fine moduli
space; see \cite[Theorem~2.3.1]{GenestierNgo}. Let $\pi^{\mathrm{univ}}_D \colon \mathfrak{A}_{g,D,\ell} \rightarrow \mathbb A_{g,D,\ell}$ be the universal abelian variety.

The universal covering $\mathfrak{H}_g \rightarrow \mathbb
A_{g,D,\ell}^{\mathrm{an}}$ \cite[Proposition~1.3.2]{GenestierNgo},
where $\mathfrak{H}_g$ is the Siegel upper half space, gives a family of abelian varieties
$\cA_{\mathfrak{H}_g,D} \rightarrow \mathfrak{H}_g$ 
fitting into the diagram
\[
\xymatrix{
\cA_{\mathfrak{H}_g,D} := \mathfrak{A}_{g,D,\ell} \times_{\mathbb A_{g,D,\ell}}\mathfrak{H}_g \ar[r] \ar[d] & \mathfrak{A}_{g,D,\ell}^{\mathrm{an}} \ar[d]^{\pi^{\mathrm{univ}}_D} \\
\mathfrak{H}_g \ar[r] & \mathbb A_{g,D,\ell}^{\mathrm{an}}.
}
\]
The family $\cA_{\mathfrak{H}_g,D} \rightarrow \mathfrak{H}_g$ is polarized of type $D$. 
For the universal covering
$u \colon \IC^g \times \mathfrak{H}_g \rightarrow \cA_{\mathfrak{H}_g,D}$
and for each $Z\in \mathfrak{H}_g$, the kernel of
$u|_{\mathbb{C}^g \times \{Z\}}$ is $D\mathbb{Z}^g + Z \mathbb{Z}^g$.
Thus the map
$\IC^g \times \mathfrak{H}_g \rightarrow \IR^g \times \IR^g \times \mathfrak{H}_g \rightarrow \IR^{2g}$,
where the first map is the inverse of $(a,b,Z) \mapsto (Da + Z  b,Z)$
and the second map is the natural projection, descends to a real
analytic map
\[
b^{\mathrm{univ}} \colon \cA_{\mathfrak{H}_g,D} \rightarrow \mathbb T^{2g}.
\]
 Now for each $s_0 \in \mathbb A_{g,D,\ell}(\IC)$, there exists a
 contractible, relatively compact, open neighborhood $\Delta$ of $s_0$ in $\mathbb A_{g,D,\ell}^{\mathrm{an}}$ such that $\mathfrak{A}_{g,D,\ell,\Delta}:=(\pi^{\mathrm{univ}}_D)^{-1}(\Delta)$ can be identified with $\cA_{\mathfrak{H}_g,\Delta'}$ for some open subset $\Delta'$ of $\mathfrak{H}_g$. The composite $b_{\Delta} \colon \mathfrak{A}_{g,D,\ell,\Delta} \cong \cA_{\mathfrak{H}_g,D,\Delta'} \rightarrow \mathbb T^{2g}$ clearly  satisfies the three properties listed in Proposition~\ref{PropBettiMapApp}. Thus $b_{\Delta}$ is the desired Betti map in this case.

\medskip
  \textbf{Case: With level structure.} 
  Assume that $\cA \rightarrow S$ carries level-$\ell$-structure for some $\ell \ge \bestlevel$ with $(\ell,d_g) = 1$. 
As $\IA_{g,D,\ell}$ is a fine moduli space there exists a Cartesian diagram
\begin{equation*}
\xymatrix{
\cA \ar[r]^-{\iota} \ar[d]_{\pi} \pullbackcorner & \mathfrak{A}_{g,D,\ell} \ar[d] \\
S \ar[r]^-{\iota_S} & \mathbb{A}_{g,D,\ell}.
}
\end{equation*}

Now let $s_0 \in S(\IC)$. Applying Proposition~\ref{PropBettiMapApp}
to the universal abelian variety
$\mathfrak{A}_{g,D,\ell} \rightarrow \mathbb A_{g,D,\ell}$ and
$\iota_S(s_0) \in \mathbb A_{g,D,\ell}(\IC)$, we obtain an 
open neighborhood $\Delta_0$ of $\iota_S(s_0)$ in $\mathbb
A_{g,D,\ell}^{\mathrm{an}}$ and a real analytic map
\[
b_{\Delta_0} \colon \mathfrak{A}_{g,\Delta_0} \rightarrow \mathbb T^{2g}
\]
satisfying the properties listed in Proposition~\ref{PropBettiMapApp}.

Now let $\Delta = \iota_S^{-1}(\Delta_0)$. Then $\Delta$ is an open neighborhood of $s_0$ in $S^{\mathrm{an}}$. Denote by $\cA_{\Delta} = \pi^{-1}(\Delta)$ and define
\[
b_{\Delta} = b_{\Delta_0} \circ \iota \colon \cA_{\Delta} \rightarrow \mathbb T^{2g}.
\]
Then $b_{\Delta}$ satisfies the properties listed in Proposition~\ref{PropBettiMapApp} for $\cA \rightarrow S$. Hence $b_{\Delta}$ is our desired Betti map.

\medskip
  \textbf{Case: General case.}
Let $s_0 \in S(\IC)$ and $\ell\ge\bestlevel$ be a prime with $(\ell,d_g)=1$.

Fix any irreducible component $S_0$ of the kernel $\mathrm{ker}[\ell]$ of $[\ell] \colon
\cA\rightarrow \cA$. It is Zariski open in $\mathrm{ker}[\ell]$ as $S$ is
regular, so we consider it with its natural open subscheme structure.
Then $S_0\rightarrow \mathrm{ker}[\ell]$ is both a closed and open
immersion. So $S_0\rightarrow S$, the composition with
the finite \'etale morphism $\mathrm{ker}[\ell]\rightarrow S$, is finite and \'etale.
The upshot is that the base change of $\cA\rightarrow S$ by $S_0\rightarrow S$
admits an $\ell$-torsion section. After repeating this finitely many
times we obtain a finite and \'etale morphism $\rho \colon S'
\rightarrow S$ where $S'$ is irreducible and such that $\cA':=
\cA\times_S S' \rightarrow S'$ has level-$\ell$-structure.
Note that $S'$ is regular as $S$ is regular and regularity ascends
along \'etale morphisms. Moreover,
 $\cA' \rightarrow S'$ is still polarized of type $D$.


Let $s_0' \in \rho^{-1}(s_0)$. Applying Proposition~\ref{PropBettiMapApp} to $\cA' \rightarrow S'$ and $s_0' \in S'(\IC)$, we obtain an open neighborhood $\Delta'$ of $s_0'$ in $(S')^{\mathrm{an}}$ and a map $b_{\Delta'} \colon \cA'_{\Delta'} \rightarrow \mathbb{T}^{2g}$ satisfying the properties listed in Proposition~\ref{PropBettiMapApp}.

Let $\Delta = \rho(\Delta')$. Up to shrinking $\Delta'$, we may assume that $\rho|_{\Delta'} \colon \Delta' \rightarrow \Delta$ is a homeomorphism and that $\Delta$ is an open neighborhood of $s_0$ in $S^{\mathrm{an}}$. Thus $\cA'_{\Delta'} \cong \cA_{\Delta}$. Now define
\[
b_{\Delta} \colon \cA_{\Delta} \rightarrow \mathbb{T}^{2g}
\]
to be the composite of the inverse of $\cA'_{\Delta'} \cong \cA_{\Delta}$ and $b_{\Delta'}$. Then $b_{\Delta}$ is our desired Betti map.
\end{proof}

Here is an easy property of the generic rank of the Betti map.
\begin{lemma}\label{LemmaBettiRankZarOpenApp}
Let $b_{\Delta} \colon \cA_{\Delta} \rightarrow \mathbb T^{2g}$ be a Betti map as in Proposition~\ref{PropBettiMapApp}. 
Let $X$ be an irreducible subvariety of $\cA$ with $\an{X}\cap\cA_\Delta\not=\emptyset$. Let $U$ be a Zariski open dense subset of $X$. Then
\begin{equation}\label{EqBettiRankSmallerUnderRestriction}
\max_{x \in X^{\mathrm{sm}}(\IC) \cap \cA_{\Delta}} \mathrm{rank}_{\IR} (\mathrm{d}b_{\Delta}|_{X^{\mathrm{sm,an}}})_x = \max_{u \in U^{\mathrm{sm}}(\IC) \cap \cA_{\Delta}} \mathrm{rank}_{\IR} (\mathrm{d}b_{\Delta}|_{U^{\mathrm{sm,an}}})_u.
\end{equation}
\end{lemma}
\begin{proof}
The statement \eqref{EqBettiRankSmallerUnderRestriction} is true on
replacing ``$=$'' by  ``$\ge$'', as
 $\sman{X}\supset \sman{U}$. 

For the converse inequality we set
 $\max_{x \in \sman{X} \cap \cA_{\Delta}} \mathrm{rank}_{\IR} (\mathrm{d}b_{\Delta}|_{X^{\mathrm{sm,an}}})_x  = r$
and pick
 $x \in \sman{X} \cap \cA_{\Delta}$ satisfying
$\mathrm{rank}_{\IR} (\mathrm{d}b_{\Delta}|_{X^{\mathrm{sm,an}}})_x =
r$. Then there exists an open neighborhood  $V$ of $x$ in 
$\sman{X}$ such that $\mathrm{rank}_{\IR}
(\mathrm{d}b_{\Delta}|_{X^{\mathrm{sm,an}}})_u = r$ for all $u \in
V$. But $U^{\mathrm{sm}}(\IC) \cap V \not= \emptyset$ since
$U^{\mathrm{sm}}\not=\emptyset$ is Zariski open in $X$ and $V$ is Zariski dense
in $X$. Thus there exists a $u \in U^{\mathrm{sm}}(\IC) \cap V$. Then
we must have $\mathrm{rank}_{\IR}
(\mathrm{d}b_{\Delta}|_{U^{\mathrm{sm,an}}})_u = r$ and the lemma
follows.
\end{proof}

\subsection{Non-degenerate subvariety and
 Theorem~\ref{ThmHtInequality}}
 We keep the notation as in the beginning of this appendix. 
\begin{definition}\label{DefinitionNonDegenerateApp}
An irreducible subvariety $X$ of $\cA$ is said to be
\textit{non-degenerate} if there exists an open non-empty subset
$\Delta$ of $S^{\mathrm{sm},\mathrm{an}}$,
   with the Betti map
$b_{\Delta} \colon \cA_{\Delta}:=\pi^{-1}(\Delta) \rightarrow
\mathbb{T}^{2g}$ as in Proposition~\ref{PropBettiMapApp}, such that
$X^\mathrm{sm,an}\cap \cA_{\Delta}\not=\emptyset$ and 
\[
\max_{x \in X^{\mathrm{sm}}(\IC) \cap \cA_{\Delta}} \mathrm{rank}_{\IR} (\mathrm{d}b_{\Delta}|_{X^{\mathrm{sm,an}}})_x = 2\dim X.
\]
\end{definition}



Now we are ready to prove Theorem~\ref{ThmHtInequalityFull}.

\begin{proof}[Proof of Theorem~\ref{ThmHtInequalityFull}]
  Let $\ell\ge 3$ be a prime.
  We will  reduce the current theorem to Theorem~\ref{ThmHtInequality}
  by successively assuming, in addition to the hypothesis of Theorem~\ref{ThmHtInequalityFull}, that
\begin{enumerate}
  \item[(i)] $X$ is Zariski closed in $\cA$,
  \item[(ii)] $\pi|_X \colon X\rightarrow S$ is dominant,
  \item[(iii)] $S$ is regular, 
  \item[(iv)] $\cA \rightarrow S$ is $S$-isogenous to an abelian scheme which carries a principal polarization,
  \item[(v)] $\cA \rightarrow S$ carries a principal polarization,
  \item[(vi)] $\cA$ carries a level $\ell$-structure, and
  \item[(vii)]  we have the same hypothesis as Theorem~\ref{ThmHtInequality}.
%
\end{enumerate}
We will proceed the proof with six d\'{e}vissage steps.
In d\'evissage step $n$ we will deduce
the theorem under the hypotheses (i),\ldots,($n-1$) 
from the theorem under the hypotheses (i),\ldots,($n$). 


\ul{First d\'{e}vissage: reduction to the case where $X$ is Zariski
  closed in $\cA$.}

Let $\overline{X}$ denote the
Zariski closure of $X$ in $\cA$. Then $X$ 
is a Zariski open dense subset of $\overline{X}$ and $\dim X = \dim
\overline{X}$. 
Therefore, $\overline X$ is non-degenerate if $X$ is non-degenerate.
Now if
\eqref{EqHtInequalityFull} holds true on a Zariski open dense subset
$U$ of $\overline{X}$, then \eqref{EqHtInequalityFull} clearly holds
true on $U \cap X$, which is Zariski open and dense in $X$. Thus it suffices to prove \eqref{EqHtInequalityFull} with $X$ replaced by $\overline{X}$. 

\ul{Second d\'{e}vissage: reduction to the case where $\pi|_{X}\colon X\rightarrow 
  S$ is dominant.} 
  
  As $X$ is non-degenerate, there exists a non-empty open subset $\Delta$ of $S^{\mathrm{sm,an}}$, with Betti map $b_{\Delta}$, such that $\mathrm{rank}_{\IR}(\mathrm{d}b_{\Delta}|_{X^{\mathrm{sm,an}}})_x = 2 \dim X$ for some $x \in X^{\mathrm{sm}}(\IC) \cap \cA_{\Delta}$.

Endow the Zariski closed set $S'=\pi(X)$ with the reduced induced
subscheme structure and set $\cA' = \cA \times_S S' =
\pi^{-1}(S')$. Then $X \times_S S'$ identifies with $X$ via the natural projection $\cA' \rightarrow \cA$. Hence there exists a non-empty open subset $\Delta'$
of $(S')^{\mathrm{sm,an}}$ with $\pi(x) \in \Delta'\subset \Delta$ and
\[
 \mathrm{rank}_{\IR} (\mathrm{d}b_{\Delta'}|_{X^{\mathrm{sm,an}}})_x = 2\dim X.
\]
Thus $X$ is a non-degenerate subvariety of $\cA'$. On the other hand,
the conclusion of  Theorem~\ref{ThmHtInequalityFull} does not change
with $\cA \rightarrow S$ replaced by $\cA' \rightarrow S'$, $\cL$
replaced by $\cL|_{\cA'}$ and $\cM$ replaced by $\cM|_{\overline{S'}}$,
 where $\overline{S'}$ is the Zariski closure of $S'$ in $\overline{S}$.
Hence it suffices to prove Theorem~\ref{ThmHtInequalityFull} after
these replacements and thus we may assume that $X$ dominates $S$.

\ul{Third d\'{e}vissage: reduction to the case where $S$ is regular.}

Recall that $S^{\mathrm{sm}}$ is the regular locus of $S$.
Now $\pi|_X\colon X\rightarrow S$ is dominant, so
$X'=X\cap \pi^{-1}(S^{\mathrm{sm}})$ is Zariski open and dense in
$X$. 
Since $X$ is non-degenerate it follows by definition that  $X'$ is
non-degenerate. Moreover, the conclusion of
Theorem~\ref{ThmHtInequalityFull} does not change if we replace $\cA
\rightarrow S$ by $\cA' = \pi^{-1}(S^{\mathrm{sm}})
\rightarrow S^{\mathrm{sm}}$, $\cL$ by $\cL|_{\cA'}$,
and $X$ by $X'$.
Finally, observe that $X'$ is Zariski closed in $\cA'$
and $\pi(X') = \pi(X)\cap S^{\mathrm{reg}}$, so $\pi|_{X'}\colon X'
\rightarrow S^{\mathrm{reg}}$ is dominant. 

\ul{Fourth d\'{e}vissage: reduction to the case where $\pi \colon \cA \rightarrow S$ is $S$-isogenous to an abelian scheme which carries a principal polarization.}


By \cite[$\mathsection$23,
Corollary~1]{MumfordAbVar}, each abelian variety over an algebraic
closed field is isogenous to a principally polarized one. Applying
this to the geometric generic fiber of $\cA \rightarrow S$, we obtain a
quasi-finite \'{e}tale dominant morphism 
 $\rho \colon S' \rightarrow S$ with $S'$ irreducible and the
 following property: There exists a principally polarized $A'_0$ that
 is isogenous over $\IQbar(S')$ to  the
 generic fiber $A'$ of $\cA':= \cA\times_{S} S' \rightarrow S'$. 
 Up to replacing $S'$ by an open dense subscheme, we may furthermore assume that $A'_0$ extends to an abelian scheme $\cA'_0 \rightarrow S'$. 
 Denote by $\rho_{\cA} \colon \cA' =\cA \times_S S' \rightarrow \cA$ the natural projection; it is a quasi-finite \'{e}tale dominant morphism.

As regularity ascends along \'etale morphisms and as $S$ is regular we conclude that $S'$ is regular. Thus $\cA'_0 \rightarrow S'$ carries a principal polarization by \cite[Th\'{e}or\`{e}me~XI~1.4]{LNM119}, and the isogeny $A'_0 \rightarrow A'$ extends to an $S'$-isogeny $\cA'_0 \rightarrow \cA'$ by \cite[Lemme~XI~1.15]{LNM119}.

There is  an irreducible component $X'$ of $\rho_{\cA}^{-1}(X)$ with
$\dim X' = \dim X$. Then $X'$ is Zariski closed in $\cA'$, 
the image $\rho_{\cA}(X')$ is Zariski dense in $X$,  and thus $X'$ dominates $S'$ (it even surjects to $S'$
since $\cA'\rightarrow S'$ is proper and $X'$ is closed). We
claim that $X'$, as a subvariety of $\cA'$, is non-degenerate.
Indeed, $\rho_{\cA}(X')$ contains a Zariski open dense subset $U$ of
$X$. Since $X$ is a non-degenerate
subvariety of $\cA$, so is $U$ by
Lemma~\ref{LemmaBettiRankZarOpenApp}. So there exists an open subset
$\Delta$ of $\an{S}$ with the Betti map $b_{\Delta} \colon
\cA_{\Delta} \rightarrow \mathbb{T}^{2g}$ such that
$\mathrm{rank}_{\IR}(\mathrm{d}b_{\Delta}|_{\ansm{U}})_u = 2 \dim U =
2 \dim X$ 
for all $u$ from a non-empty open subset of $\an{U}$
. Take $\Delta'$ to be a connected component of $\rho^{-1}(\Delta)$ such that $X'
\cap (\pi')^{-1}(\Delta') \not= \emptyset$. Set $\cA'_{\Delta'} =
(\pi')^{-1}(\Delta')$, and replace $\Delta$ by $\rho(\Delta')$.
Note that $\rho|_{\Delta'}\colon \Delta' \cong \Delta$ is then
bianalytic after possibly shrinking $\Delta'$ (and so is $\rho_{\cA}\colon \cA'_{\Delta'} \cong
\cA_{\Delta}$). Now $b_{\Delta} \circ \rho_{\cA}|_{\cA'_{\Delta'}} \colon \cA'_{\Delta'} \rightarrow \mathbb{T}^{2g}$ satisfies the three properties listed in Proposition~\ref{PropBettiMapApp}. So $b_{\Delta} \circ \rho_{\cA}|_{\cA'_{\Delta'}}$ is the Betti map, which we denote for simplicity by $b_{\Delta'}$; see below Proposition~\ref{PropBettiMapApp}. 
For $u' \in  (\rho_{\cA}|_{\cA'_{\Delta'}})^{-1}(u)\cap \an{X'}$
 and for sufficiently general $u$
, we have 
$\mathrm{rank}_{\IR}(\mathrm{d}b_{\Delta'}|_{\ansm{X'}})_{u'} = 2 \dim X$. So $X'$, as a subvariety of the abelian scheme $\cA'$ over $S'$, is non-degenerate.

Now we have a non-degenerate subvariety $X'$ of the abelian scheme $\pi' \colon \cA' \rightarrow S'$. The line bundle $\rho_{\cA}^*\cL$ on $\cA'$ is relatively ample. Suppose that $\cM'$ is an ample line bundle on some compactification $\overline{S'}$ of $S'$.

Assume that Theorem~\ref{ThmHtInequalityFull} holds for $\pi' \colon \cA' \rightarrow S'$, $\rho_{\cA}^*\cL$, $\cM'$, and $X'$. Thus there exist constants $c_1'>0$, $c_2' \ge 0$ and a Zariski open non-empty subset $U'$ of $X'$ with
\[
\hat{h}_{\cA',\rho_{\cA}^*\cL}(P') \ge c_1' h_{\overline{S'},\cM'}(\pi'(P')) - c_2' \quad \text{for all} \quad P' \in U'(\IQbar).
\]
Denote by $P = \rho_{\cA}(P')$. 
By the Height Machine we have $\hat{h}_{\cA',\rho_{\cA}^*\cL}(P') =
\hat{h}_{\cA,\cL}(P)$.

By  \cite[Lemma~4]{Silverman:heightest11} applied to $\rho \colon S' \rightarrow S$ and the line bundles $\cM'$ and $\cM$, there exist $c' = c' (\rho,\cM',\cM) > 0$ and $c'' = c''(\rho,\cM',\cM) \ge 0$ such that $h_{\overline{S'},\cM'}(\pi'(P')) \ge c' h_{\overline{S},\cM}(\rho(\pi'(P'))) - c'' = c' h_{\overline{S},\cM}(\pi(P)) - c''$ for all $P' \in \cA'(\IQbar)$. Hence the height inequality above implies
\[
\hat{h}_{\cA,\cL}(P) \ge c_1'c' h_{\overline{S},\cM}(\pi(P)) - (c_1'c'' + c_2')\quad \text{for all}\quad P \in \rho_{\cA}(U')(\IQbar).
\]
Now that $\rho_{\cA}(U')$ contains a Zariski open non-empty (hence dense) subset $U$ of $X$ by Chevalley's Theorem. 
Thus Theorem~\ref{ThmHtInequalityFull} also holds true for $\pi \colon \cA \rightarrow S$, $\cL$, $\cM$, and $X$.

In summary, we have shown that it suffices to prove Theorem~\ref{ThmHtInequalityFull} for $\pi' \colon \cA' \rightarrow S'$, $\rho_{\cA}^*\cL$, $\cM'$ and $X'$. Thus we are reduced to the case where the generic fiber of $\cA \rightarrow S$ is isogenous to a principally polarized abelian variety.

\ul{Fifth d\'{e}vissage: reduction to the case where $\pi \colon \cA \rightarrow S$ carries a principal polarization.}

From the previous d\'{e}vissage, there exists a principally polarized abelian scheme $\pi_0 \colon \cA_0 \rightarrow S$ with an $S$-isogeny $\lambda \colon \cA_0 \rightarrow \cA$. Note that $\lambda$ is a finite \'{e}tale morphism. The line bundle $\lambda^*\cL$ on $\cA_0$ is relatively ample. By the Height Machine we have $\hat{h}_{\cA_0,\lambda^*\cL}(P') = \hat{h}_{\cA,\cL}(\lambda(P'))$ for all $P' \in \cA_0(\IQbar)$.

There is  an irreducible component $X_0$ of $\lambda^{-1}(X)$ with
$\dim X_0 = \dim X$. Then $X_0$ is Zariski closed in $\cA_0$ and thus $X_0$ dominates $S$ (it even surjects to $S$
since $X = \lambda(X_0)$). We
claim that $X_0$, as a subvariety of $\cA_0$, is non-degenerate. Assume this. Then it suffices to prove the height inequality \eqref{EqHtInequalityFull} with $\cA \rightarrow S$ replaced by $\cA_0 \rightarrow S$, $X$ replaced by $X_0$, and $\cL$ replaced by $\lambda^*\cL$. 

It remains to prove that $X_0$ is a non-degenerate subvariety of $\cA_0$. To do this, we need some preparation on Betti maps. Let $\Delta$ be an open subset of $S^{\mathrm{an}}$ with the Betti map $b_{\Delta} \colon \cA_{\Delta} \rightarrow \mathbb{T}^{2g}$. Set $\cA_{0,\Delta} = \pi_0^{-1}(\Delta)$, and denote by $\lambda_{\Delta}$ the restriction of $\lambda \colon \cA_0 \rightarrow \cA$ to $\cA_{0,\Delta}$. Up to shrinking $\Delta$ we have a Betti map $b_{0,\Delta} \colon \cA_{0,\Delta} \rightarrow \mathbb{T}^{2g}$. By property (iii) of Proposition~\ref{PropBettiMapApp}, we have two real analytic isomorphisms $(b_{0,\Delta},\pi_0) \colon \cA_{0,\Delta} \cong \mathbb{T}^{2g} \times \Delta$ and $(b_{\Delta},\pi) \colon \cA_{\Delta} \cong \mathbb{T}^{2g} \times \Delta$. Thus there exists a real analytic map $\lambda' \colon \mathbb{T}^{2g} \times \Delta \rightarrow \mathbb{T}^{2g} \times \Delta$ such that the following diagram commutes
\[
\xymatrix{
\cA_{0,\Delta} \ar[r]^-{(b_{0,\Delta},\pi_0)}_-{\sim} \ar[d]_-{\lambda_{\Delta}} & \mathbb{T}^{2g} \times \Delta \ar[d]^-{\lambda'} \\
\cA_{\Delta} \ar[r]^-{(b_{\Delta},\pi)}_-{\sim} & \mathbb{T}^{2g} \times \Delta.
}
\]
As $\lambda$ is a finite map, $(\lambda')^{-1}(r)$ is a finite set for each $r \in \mathbb{T}^{2g} \times\Delta$. As $\lambda$ is an $S$-morphism, for each $s \in \Delta$ we have $\lambda'({\mathbb{T}^{2g} \times \{s\}}) \subseteq \mathbb{T}^{2g} \times \{s\}$.

By property (i) of Proposition~\ref{PropBettiMapApp}, for each $s \in \Delta$ the restriction $\lambda'|_{\mathbb{T}^{2g} \times \{s\}}$ is a group homomorphism $\mathbb{T}^{2g}  \rightarrow  \mathbb{T}^{2g}$. Thus $\ker(\lambda'|_{\mathbb{T}^{2g} \times \{s\}})$ is a finite, hence discrete, subgroup of $\mathbb{T}^{2g}$. In particular, $\ker(\lambda'|_{\mathbb{T}^{2g} \times \{s\}})$ is locally constant. Up to shrinking $\Delta$, we may assume $\ker(\lambda'|_{\mathbb{T}^{2g} \times \{s\}}) = H$ for each $s \in \Delta$. Set $\lambda_{\mathbb{T}} \colon \mathbb{T}^{2g} \rightarrow \mathbb{T}^{2g}$ the quotient by the finite subgroup $H$. 
Then the diagram above induces a commutative diagram
\begin{equation}
\xymatrix{
\cA_{0,\Delta} \ar[r]^-{b_{0,\Delta}} \ar[d]_-{\lambda_{\Delta}} & \mathbb{T}^{2g} \ar[d]^-{\lambda_{\mathbb{T}}} \\
\cA_{\Delta} \ar[r]^-{b_{\Delta}} & \mathbb{T}^{2g}.
}
\end{equation}
Note that $\lambda_{\mathbb{T}}$ is a local homeomorphism.

Now we turn to proving that $X_0$ is non-degenerate in $\cA_0$. Indeed, as $X$ is a non-degenerate subvariety of $\cA$, there exists an open subset $\Delta$ of $S^{\mathrm{an}}$ with the Betti map $b_{\Delta} \colon \cA_{\Delta} \rightarrow \mathbb{T}^{2g}$ such that $\mathrm{rank}_{\IR}(\mathrm{d}b_{\Delta}|_{X^{\mathrm{an,sm}}})_x = 2 \dim X$ for all $x$ from a non-empty open subset of $X^{\mathrm{an}}$. For $x_0 \in  \lambda^{-1}(x)\cap X_0^{\mathrm{an}}$
 and for sufficiently general $x$
, we have 
$\mathrm{rank}_{\IR}(\mathrm{d}(\lambda_{\mathbb{T}}\circ b_{0,\Delta})|_{X_0^{\mathrm{an,sm}}})_{x_0} = 2 \dim X = 2 \dim X_0$. But $\lambda_{\mathbb{T}}$ is a local homeomorphism, so $\mathrm{rank}_{\IR}(\mathrm{d} b_{0,\Delta}|_{X_0^{\mathrm{an,sm}}})_{x_0} = 2 \dim X_0$. Thus $X_0$ is non-degenerate.


\ul{Sixth d\'{e}vissage: reduction to the case where $\cA/S$
  carries level $\ell$ structure.}

As in the treatment of the general case in the proof of
Proposition~\ref{PropBettiMapApp}
there exists a finite and \'etale morphism $S'\rightarrow S$ where
$S'$ is regular and irreducible
 such that $\cA' := \cA\times_S S'$ carries level $\ell$-structure.

Denote by $\rho_{\cA} \colon \cA' \rightarrow \cA$ the natural
projection. By a similar argument as the fourth d\'{e}vissage step, it
suffices to prove the height inequality \eqref{EqHtInequalityFull} with
$\cA \rightarrow S$ replaced by $\cA' \rightarrow S'$, $X$ replaced by
an irreducible component of $\rho_{\cA}^{-1}(X)$ with $\dim X'=\dim
X$, and $\cL$ replaced by $\rho_{\cA}^*\cL$, and $\cM$ replaced by
an ample line bundle on some compactification of $S'$.
  As in the fourth d\'evissage  $X'$ dominates $S'$. Finally,
  $\cA'/S'$ still carries a principal polarization.

\ul{Seventh d\'{e}vissage: reduction to  Theorem {\ref{ThmHtInequality}}.}

It remains to prove the height inequality (\ref{EqHtInequalityFull})
with the extra hypotheses (i) - (v) listed above using
Theorem~\ref{ThmHtInequality}. In this theorem we assumed in addition that 
the
fiberwise  N\'eron--Tate height on $\cA(\IQbar)$
is induced by  a closed immersion $\cA \rightarrow \IP^n_{\IQbar}
\times S$ satisfying the second and third bullet  at the beginning of
$\mathsection$\ref{SectionSettingUpHtIneq} and
that the height on $\overline S(\IQbar)$ is the restriction of the absolute
logarthmic Weil height coming from a closed immersion
$\overline S\rightarrow\IP^m_{\IQbar}$. 

A basis of the global sections of the line bundle $\cM^{\otimes p}$,
for some $p$ large enough, gives rise to a closed immersion $\overline S \subseteq
\IP_{\IQbar}^m$. This gives the first bullet point at the beginning of
$\mathsection$\ref{SectionSettingUpHtIneq}. Note that the Weil
height $h$ on $\IP_\IQbar^m(\IQbar)$ restricted to $\overline
S(\IQbar)$ via this immersion differs from 
$p h_{\overline S,\cM}$ by a bounded function.

For the line bundle $\cL$ on $\cA$, which is ample relative over $S$, we have that $\cL^{\otimes 4}$ is relatively very ample on $\cA/S$. Thus by  \cite[Proposition~4.4.10.(ii) and Proposition~4.1.4]{EGAII}, there is a closed immersion $\cA \rightarrow \IP^n_S = \IP^n \times S$ given by global sections of $\cL^{\otimes 4} \otimes \pi^*\cM^{\otimes q}$ for some large $q$. When restricted to the generic fiber $A$ of $\cA \rightarrow S$, we get a closed immersion $A \rightarrow \IP^n_{k(S)}$ which arises from a basis of the global sections of $L^{\otimes 4}$, where $L$ is the restriction of $\cL$ over the generic fiber $A$. Moreover $L$ is ample since $\cL$ is relatively ample, and $L$ is symmetric since $[-1]^*\cL = \cL$. Thus we also have the second and third bullet points at the beginning of $\mathsection$\ref{SectionSettingUpHtIneq}.

Note that the height function $\hat{h}_{\cA}$ defined in \eqref{eq:fiberwiseNT} is then
\[
\hat{h}_{\cA} \colon \cA(\IQbar) \rightarrow [0,\infty), \quad P \rightarrow \hat{h}_{\cA_s,\cL_s^{\otimes 4}}(P)
\]
where $s = \pi(P)$. So $\hat{h}_{\cA,\cL} = (1/4)\hat{h}_{\cA}$.

The full
hypothesis of Theorem~\ref{ThmHtInequality}  is now satisfied for $\cA$
and $X$,
\textit{e.g.}, \texttt{(Hyp)} is just (iv) and (v).
We get constants $c_1> 0$ and $c_2$ and a Zariski open dense subset $U$ of $X$ such that
\[
    \hat{h}_{\cA}(P) \ge c_1 h( \pi(P) ) - c_2 \quad \text{for all}\quad P \in U(\IQbar).
\]
Thus \eqref{EqHtInequalityFull} holds true with $c_1$ replaced by $(c_1p)/4$ and $c_2$ replaced by $c_2/4 + O_S(1)$, where $O_S(1)$ is a bounded function on $S(\IQbar)$. So we are done.
\end{proof}




\bibliographystyle{alpha}
\bibliography{literature}

\def\cprime{$'$}
\begin{thebibliography}{CGHX21}

\bibitem[ACG11]{ACG:Curve}
E.~Arbarello, M.~Cornalba, and P.~Griffiths.
\newblock {\em Geometry of Algebraic Curves, II (with a contribution by
  J.~Harris)}, volume 268 of {\em Grundlehren der mathematischen
  Wissenschaften}.
\newblock Springer-Verlag, Berlin, 2011.

\bibitem[ACZ20]{ACZBetti}
Yves Andr\'{e}, Pietro Corvaja, and Umberto Zannier.
\newblock The {B}etti map associated to a section of an abelian scheme (with an
  appendix by {Z.~Gao}).
\newblock {\em Inv. Math.}, 222:161--202, 2020.

\bibitem[Alp18]{AlpogeRatPt}
L.~Alpoge.
\newblock The average number of rational points on genus two curves is bounded.
\newblock {\em arXiv:1804.05859}, 2018.

\bibitem[Alp20]{AlpogeThesis}
L.~Alpoge.
\newblock {\em Points on Curves}.
\newblock PhD thesis, Princeton University, 2020.

\bibitem[BG06]{BG}
E.~Bombieri and W.~Gubler.
\newblock {\em {H}eights in {D}iophantine {G}eometry}.
\newblock Cambridge University Press, 2006.

\bibitem[BGS94]{BGS}
J.-B. Bost, H.~Gillet, and C.~Soul{\'e}.
\newblock {H}eights of projective varieties and positive {G}reen forms.
\newblock {\em J. Amer. Math. Soc.}, 7(2):903--1027, 1994.

\bibitem[BLR90]{NeronModels}
S.~Bosch, W.~L{\"u}tkebohmert, and M.~Raynaud.
\newblock {\em N\'eron models}, volume~21 of {\em Ergebnisse der Mathematik und
  ihrer Grenzgebiete (3)}.
\newblock Springer-Verlag, Berlin, 1990.

\bibitem[Bom90]{Bombieri:Mordell}
E.~Bombieri.
\newblock The {M}ordell conjecture revisited.
\newblock {\em Ann. Scuola Norm. Sup. Pisa Cl. Sci. (4)}, 17(4):615--640, 1990.

\bibitem[CGHX21]{CGHX:18}
S.~Cantat, Z.~Gao, P.~Habegger, and J.~Xie.
\newblock The geometric {B}ogomolov conjecture.
\newblock {\em Duke Math. J.}, 170(2):247--277, 2021.

\bibitem[Cha41]{Chabauty}
C.~Chabauty.
\newblock Sur les points rationnels des courbes alg\'{e}briques de genre
  sup\'{e}rieur \`a l'unit\'{e}.
\newblock {\em C. R. Acad. Sci. Paris}, 212:882--885, 1941.

\bibitem[Col85]{Coleman:effCha}
R.~.F. Coleman.
\newblock Effective {C}habauty.
\newblock {\em Duke Math. J.}, 52(3):765--770, 1985.

\bibitem[CVV17]{CVV:17}
S.~Checcoli, F.~Veneziano, and E.~Viada.
\newblock On the explicit torsion anomalous conjecture.
\newblock {\em Trans. Amer. Math. Soc.}, 369(9):6465--6491, 2017.

\bibitem[dD97]{deDiego:97}
T.~de~Diego.
\newblock Points rationnels sur les familles de courbes de genre au moins 2.
\newblock {\em J. Number Theory}, 67(1):85--114, 1997.

\bibitem[Dem12]{Demailly}
J.P. Demailly.
\newblock {\em Complex analytic and differential geometry}.
\newblock 2012.
\newblock Available at \url{https://www-fourier.ujf-grenoble.fr/\textasciitilde
  demailly/manuscripts/agbook.pdf}.

\bibitem[DGH19]{DGH1p}
V.~Dimitrov, Z.~Gao, and P.~Habegger.
\newblock Uniform bound for the number of rational points on a pencil of
  curves.
\newblock {\em Int. Math. Res. Not. IMRN},
  (rnz248):https://doi.org/10.1093/imrn/rnz248, 2019.

\bibitem[DKY20]{DeMarcoKriegerYeUniManinMumford}
L.~DeMarco, H.~Krieger, and H.~Ye.
\newblock Uniform {M}anin-{M}umford for a family of genus $2$ curves.
\newblock {\em Ann. of Math.}, 191:949--1001, 2020.

\bibitem[DM69]{DM:irreducibility}
P.~Deligne and D.~Mumford.
\newblock The irreducibility of the space of curves of given genus.
\newblock {\em Inst. Hautes \'{E}tudes Sci. Publ. Math.}, (36):75--109, 1969.

\bibitem[DNP07]{DaNaPh:07}
S.~David, M.~Nakamaye, and P.~Philippon.
\newblock Bornes uniformes pour le nombre de points rationnels de certaines
  courbes.
\newblock In {\em Diophantine geometry}, volume~4 of {\em CRM Series}, pages
  143--164. Ed. Norm., Pisa, 2007.

\bibitem[DP02]{DPvarabII}
S.~David and P.~Philippon.
\newblock Minorations des hauteurs normalis\'ees des sous-vari\'et\'es de
  vari\'et\'es abeliennes. {II}.
\newblock {\em Comment. Math. Helv.}, 77(4):639--700, 2002.

\bibitem[DP07]{DaPh:07}
S.~David and P.~Philippon.
\newblock Minorations des hauteurs normalis\'{e}es des sous-vari\'{e}t\'{e}s
  des puissances des courbes elliptiques.
\newblock {\em Int. Math. Res. Pap. IMRP}, (3):Art. ID rpm006, 113, 2007.

\bibitem[Fal83]{Faltings:ES}
G.~Faltings.
\newblock {E}ndlichkeitss{\"a}tze f{\"u}r abelsche {V}ariet{\"a}ten {\"u}ber
  {Z}ahlk{\"o}rpern.
\newblock {\em Invent. Math.}, 73:349--366, 1983.

\bibitem[Fal91]{Faltings:DAAV}
G.~Faltings.
\newblock Diophantine approximation on abelian varieties.
\newblock {\em Ann. of Math. (2)}, 133(3):549--576, 1991.

\bibitem[FC90]{FaltingsChai}
G.~Faltings and C.-L. Chai.
\newblock {\em Degeneration of abelian varieties}, volume~22 of {\em Ergebnisse
  der Mathematik und ihrer Grenzgebiete (3) [Results in Mathematics and Related
  Areas (3)]}.
\newblock Springer-Verlag, Berlin, 1990.
\newblock With an appendix by David Mumford.

\bibitem[Ful98]{Fulton}
W.~Fulton.
\newblock {\em Intersection theory}, volume~2 of {\em Ergebnisse der Mathematik
  und ihrer Grenzgebiete. 3. Folge. A Series of Modern Surveys in Mathematics
  [Results in Mathematics and Related Areas. 3rd Series. A Series of Modern
  Surveys in Mathematics]}.
\newblock Springer-Verlag, Berlin, second edition, 1998.

\bibitem[Gao20a]{GaoBettiRank}
Z.~Gao.
\newblock Generic rank of {B}etti map and unlikely intersections.
\newblock {\em Compos. Math.}, 156(12):2469--2509, 2020.

\bibitem[Gao20b]{GaoMixedAS}
Z.~Gao.
\newblock Mixed {A}x-{S}chanuel for the universal abelian varieties and some
  applications.
\newblock {\em Compos. Math.}, 156(11):2263--2297, 2020.

\bibitem[GH19]{GaoHab}
Z.~Gao and P.~Habegger.
\newblock {Heights in Families of Abelian Varieties and the Geometric Bogomolov
  Conjecture}.
\newblock {\em Ann. of Math.}, 189(2):527--604, 2019.

\bibitem[GN09]{GenestierNgo}
A.~Genestier and B.C. Ng\^{o}.
\newblock Lecture on {S}himura varieties.
\newblock In {\em Autour de motifs, Ecole d'\'et\'e Franco-Asiatique de
  G\'eom\'etrie Alg\'ebrique et de Th\'eorie des Nombres/Asian-French Summer
  School on Algebraic Geometry and Number Theory. Vol.~I}, pages 187--236.
  Panor. Synth\`ese 29, Soc. Math. France, 2009.

\bibitem[Gro61]{EGAII}
A.~Grothendieck.
\newblock \'{E}l\'ements de g\'eom\'etrie alg\'ebrique. {II}. \'{E}tude globale
  \'{e}l\'{e}mentaire de quelques classes de morphismes.
\newblock {\em Inst. Hautes \'Etudes Sci. Publ. Math.}, (8), 1961.

\bibitem[Gro67]{EGAIV}
A.~Grothendieck.
\newblock \'{E}l\'ements de g\'eom\'etrie alg\'ebrique. {IV}. \'{E}tude locale
  des sch\'emas et des morphismes de sch\'emas {I}-{IV}.
\newblock {\em Inst. Hautes \'Etudes Sci. Publ. Math.}, (20,24,28,32),
  1964--1967.

\bibitem[GW10]{GoertzWedhorn}
U.~G\"{o}rtz and T.~Wedhorn.
\newblock {\em Algebraic geometry {I}}.
\newblock Advanced Lectures in Mathematics. Vieweg + Teubner, Wiesbaden, 2010.
\newblock Schemes with examples and exercises.

\bibitem[Hab09]{habegger:imrn}
P.~Habegger.
\newblock On the bounded height conjecture.
\newblock {\em Int. Math. Res. Not. IMRN}, (5):860--886, 2009.

\bibitem[Hab13]{Hab:Special}
P.~Habegger.
\newblock {S}pecial {P}oints on {F}ibered {P}owers of {E}lliptic {S}urfaces.
\newblock {\em J.Reine Angew. Math.}, 685:143--179, 2013.

\bibitem[HP16]{HabeggerPilaENS}
P.~Habegger and J.~Pila.
\newblock O-minimality and certain atypical intersections.
\newblock {\em Ann. Sci. École Norm. Sup.}, 49:813--858, 2016.

\bibitem[Hur92]{Hurwitz1892}
A.~Hurwitz.
\newblock \"{U}ber algebraische {G}ebilde mit eindeutigen {T}ransformationen in
  sich.
\newblock {\em Math. Ann.}, 41(3):403--442, 1892.

\bibitem[KRZB16]{KatzRabinoffZB}
E.~Katz, J.~Rabinoff, and D.~Zureick-Brown.
\newblock Uniform bounds for the number of rational points on curves of small
  {M}ordell-{W}eil rank.
\newblock {\em Duke Math. J.}, 165(16):3189--3240, 2016.

\bibitem[K{\"u}h20]{kuehne:semiabelianbhc}
L.~K{\"u}hne.
\newblock The {B}ounded {H}eight {C}onjecture for {S}emiabelian {V}arieties.
\newblock {\em Compos. Math}, 156:1405--1456, 2020.

\bibitem[Lan65]{Lang:Division}
S.~Lang.
\newblock Division points on curves.
\newblock {\em Ann. Mat. Pura Appl. (4)}, 70:229--234, 1965.

\bibitem[Lan78]{Lang:EllipticCurves78}
S.~Lang.
\newblock {\em Elliptic curves: {D}iophantine analysis}, volume 231 of {\em
  Grundlehren der Mathematischen Wissenschaften [Fundamental Principles of
  Mathematical Sciences]}.
\newblock Springer-Verlag, Berlin-New York, 1978.

\bibitem[Laz04]{PosAlgGeom}
R.~Lazarsfeld.
\newblock {\em {P}ositivity in {A}lgebraic {G}eometry {I}}.
\newblock Springer, 2004.

\bibitem[Lel57]{Lelong}
P.~Lelong.
\newblock Int\'{e}gration sur un ensemble analytique complexe.
\newblock {\em Bull. Soc. Math. France}, 85:239--262, 1957.

\bibitem[Maz86]{mazur1986arithmetic}
B.~Mazur.
\newblock Arithmetic on curves.
\newblock {\em Bulletin of the American Mathematical Society}, 14(2):207--259,
  1986.

\bibitem[Maz00]{Mazur:00}
B.~Mazur.
\newblock Abelian varieties and the {M}ordell-{L}ang conjecture.
\newblock In {\em Model theory, algebra, and geometry}, volume~39 of {\em Math.
  Sci. Res. Inst. Publ.}, pages 199--227. Cambridge Univ. Press, Cambridge,
  2000.

\bibitem[MFK94]{MFK:GIT94}
D.~Mumford, J.~Fogarty, and F.~Kirwan.
\newblock {\em Geometric invariant theory}, volume~34 of {\em Ergebnisse der
  Mathematik und ihrer Grenzgebiete (2) [Results in Mathematics and Related
  Areas (2)]}.
\newblock Springer-Verlag, Berlin, third edition, 1994.

\bibitem[Mok91]{Mok11Form}
N.~Mok.
\newblock Aspects of {K}\"ahler geometry on arithmetic varieties.
\newblock In {\em Several complex variables and complex geometry, {P}art 2
  ({S}anta {C}ruz, {CA}, 1989)}, volume~52 of {\em Proc. Sympos. Pure Math.},
  pages 335--396. Amer. Math. Soc., Providence, RI, 1991.

\bibitem[Mum70]{Mumford:quadeq}
D.~Mumford.
\newblock Varieties defined by quadratic equations.
\newblock In {\em Questions on {A}lgebraic {V}arieties ({C}.{I}.{M}.{E}., {III}
  {C}iclo, {V}arenna, 1969)}, pages 29--100. Edizioni Cremonese, Rome, 1970.

\bibitem[Mum74]{MumfordAbVar}
D.~Mumford.
\newblock {\em Abelian Varieties, 2nd ed.}
\newblock Oxford University Press, London, 1974.

\bibitem[OS80]{OortSteenbrink}
F.~Oort and J.~Steenbrink.
\newblock The local {T}orelli problem for algebraic curves.
\newblock In {\em Journ\'{e}es de {G}\'{e}ometrie {A}lg\'{e}brique d'{A}ngers,
  {J}uillet 1979/{A}lgebraic {G}eometry, {A}ngers, 1979}, pages 157--204.
  Sijthoff \& Noordhoff, Alphen aan den Rijn---Germantown, Md., 1980.

\bibitem[Paz12]{Pazuki:12}
F.~Pazuki.
\newblock Theta height and {F}altings height.
\newblock {\em Bull. Soc. Math. France}, 140(1):19--49, 2012.

\bibitem[Phi86]{Philippon}
P.~Philippon.
\newblock {L}emmes de z{\'e}ros dans les groupes alg{\'e}briques commutatifs.
\newblock {\em Bull. Soc. Math. France}, 114:355--383, 1986.

\bibitem[Phi95]{HauteursAlt3}
P.~Philippon.
\newblock Sur des hauteurs alternatives {III}.
\newblock {\em J. Math. Pures Appl.}, 74:345--365, 1995.

\bibitem[Pin89]{PinkThesis}
R.~Pink.
\newblock {\em Arithmetical compactification of mixed {S}himura varieties}.
\newblock PhD thesis, Bonner Mathematische Schriften, 1989.

\bibitem[Pin05a]{Pink05}
R.~Pink.
\newblock {A} combination of the conjectures of {M}ordell-{L}ang and
  {A}ndr\'e-{O}ort.
\newblock In {\em {G}eometric methods in algebra and number theory}, volume 235
  of {\em {P}rogr. {M}ath.}, pages 251--282. Birk{\"a}user, 2005.

\bibitem[Pin05b]{Pink}
R.~Pink.
\newblock {A} {C}ommon {G}eneralization of the {C}onjectures of
  {A}ndr\'e-{O}ort, {M}anin-{M}umford, and {M}ordell-{L}ang.
\newblock {\em Preprint}, page 13pp, 2005.

\bibitem[Ray70]{LNM119}
M.~Raynaud.
\newblock {\em Faisceaux amples sur les sch\'{e}mas en groupes et les espaces
  homog\`enes}.
\newblock Lecture Notes in Mathematics, Vol. 119. Springer-Verlag, Berlin-New
  York, 1970.

\bibitem[Ray83]{Raynaud:MMcurves}
M.~Raynaud.
\newblock Courbes sur une vari\'{e}t\'{e} ab\'{e}lienne et points de torsion.
\newblock {\em Invent. Math.}, 71(1):207--233, 1983.

\bibitem[R{\'e}m00a]{Remond:Decompte}
G.~R{\'e}mond.
\newblock D{\'e}compte dans une conjecture de {L}ang.
\newblock {\em Invent. Math.}, 142(3):513--545, 2000.

\bibitem[R{\'e}m00b]{remond:vojtasup}
G.~R{\'e}mond.
\newblock In{\'e}galit\'{e} de {V}ojta en dimension sup{\'e}rieure.
\newblock {\em Ann. Scuola Norm. Sup. Pisa Cl. Sci. (4)}, 29(1):101--151, 2000.

\bibitem[Sil83]{Silverman}
J.~H. Silverman.
\newblock {H}eights and the specialization map for families of abelian
  varieties.
\newblock {\em J. Reine Angew. Math.}, 342:197--211, 1983.

\bibitem[Sil93]{Silverman:twists}
J.~H. Silverman.
\newblock A uniform bound for rational points on twists of a given curve.
\newblock {\em J. London Math. Soc. (2)}, 47(3):385--394, 1993.

\bibitem[Sil11]{Silverman:heightest11}
J.~H. Silverman.
\newblock Height estimates for equidimensional dominant rational maps.
\newblock {\em J. Ramanujan Math. Soc.}, 26(2):145--163, 2011.

\bibitem[Sto19]{Stoll:Uniform}
M.~Stoll.
\newblock Uniform bounds for the number of rational points on hyperelliptic
  curves of small {M}ordell-{W}eil rank.
\newblock {\em J. Eur. Math. Soc. (JEMS)}, 21(3):923--956, 2019.

\bibitem[Voi02]{voisin:hodge1}
C.~Voisin.
\newblock {\em Hodge theory and complex algebraic geometry. {I}}, volume~76 of
  {\em Cambridge Studies in Advanced Mathematics}.
\newblock Cambridge University Press, Cambridge, {E}nglish edition, 2002.

\bibitem[Voj91]{Vojta:siegelcompact}
P.~Vojta.
\newblock Siegel's theorem in the compact case.
\newblock {\em Ann. of Math. (2)}, 133(3):509--548, 1991.

\bibitem[Wal87]{Ast6970}
M.~Waldschmidt.
\newblock Nombres transcendants et groupes alg\'{e}briques.
\newblock {\em Ast\'{e}risque}, (69-70):218, 1987.
\newblock With appendices by Daniel Bertrand and Jean-Pierre Serre.

\bibitem[Zan12]{ZannierBook}
U.~Zannier.
\newblock {\em Some problems of unlikely intersections in arithmetic and
  geometry}, volume 181 of {\em Annals of Mathematics Studies}.
\newblock Princeton University Press, Princeton, NJ, 2012.
\newblock With appendixes by David Masser.

\bibitem[Zha98]{ZhangEquidist}
S.~Zhang.
\newblock Equidistribution of small points on abelian varieties.
\newblock {\em Ann. of Math. (2)}, 147(1):159--165, 1998.

\end{thebibliography}


\vfill


\end{document}